\documentclass[reqno]{amsart}

\usepackage{amsmath, amssymb, amsthm}
\usepackage{hyperref}
\usepackage{cleveref}
\usepackage{dsfont}
\usepackage{color}

\newtheorem{theorem}{Theorem}[section]
\newtheorem{lemma}[theorem]{Lemma}
\newtheorem{corollary}[theorem]{Corollary}
\newtheorem{proposition}[theorem]{Proposition}

\theoremstyle{definition}
\newtheorem{definition}[theorem]{Definition}

\newtheorem{assumption}[theorem]{Assumption}

\theoremstyle{remark}
\newtheorem{remark}[theorem]{Remark}

\newcommand{\eps}{\varepsilon}
\newcommand{\E}{\mathbb{E}}
\newcommand{\R}{\mathbb{R}}
\renewcommand{\P}{\mathbb{P}}

\newcommand{\mfc}{\mathfrak{c}}
\newcommand{\mft}{\mathfrak{t}}

\renewcommand{\phi}{\varphi}

\newcommand{\tv}{\mathrm{TV}}
\newcommand{\var}{\mathrm{Var}}
\newcommand{\full}{\mathrm{full}}
\newcommand{\init}{\mathrm{init}}
\newcommand{\ind}{\perp\!\!\!\!\perp} 
\newcommand{\Var}{\operatorname{Var}}
\newcommand{\logit}{\operatorname{logit}}

\usepackage[
backend=biber,
style=numeric,
giveninits=true,
url = false,
doi = false,
isbn =false,
maxbibnames=99,
maxalphanames=99,
sortcites=true
]{biblatex}
\renewbibmacro{in:}{}
\DeclareFieldFormat{title}{#1} 
\DeclareFieldFormat[article]{title}{#1}
\DeclareFieldFormat[inbook]{title}{#1}
\DeclareFieldFormat[incollection]{title}{#1}
\AtBeginBibliography{\small}

\title[One-Step TMLE for WATEs]{One-step TMLE for weighted average treatment effects}
\date{\today}
\author{Yang Liu}
\author{Patrick Lopatto}
\author{Ivana Malenica}

\begin{document}
\begin{abstract}
We consider Targeted Maximum Likelihood Estimation (TMLE) of weighted average treatment effects (WATEs), a class of causal estimands that reweight the covariate distribution using a specified function of the propensity score. This class includes the average treatment effect and average treatment effect on the treated, as well as various overlap-based targets. 
We provide a comprehensive analysis of the one-step TMLE along the universal least favorable path for such parameters. Under explicit regularity conditions on the weight function and initialization, we show that the targeting procedure is  well-defined, reaches a solution of the estimating equation in finite time, and yields an asymptotically efficient estimator. In particular, convergence of the targeting dynamics and control of the second-order remainder are derived from these conditions rather than imposed as separate assumptions on the output of the algorithm.
\end{abstract}

\maketitle
\setcounter{tocdepth}{1}
\tableofcontents

\section{Introduction}

Targeted Maximum Likelihood Estimation (TMLE) is a general framework for constructing semiparametrically efficient estimators in complex statistical models \cite{vanderLaanRubin2006TMLE,book2011,book2018}. It combines flexible, data-adaptive estimation of nuisance components with a targeted fluctuation along a least favorable submodel aligned with the efficient influence function (EIF). Starting from an initial estimate of the data-generating distribution, the targeting step updates the model so that the resulting plug-in estimator solves the empirical EIF equation. Under suitable regularity conditions, this yields an asymptotically linear and semiparametrically efficient estimator. TMLE has been especially successful in causal inference, where nuisance functions are often high-dimensional and are naturally estimated using modern machine learning methods \cite{schuler2017targeted,SMITH202334}.

In parallel, weighted average treatment effects (WATEs) have emerged as a flexible and practically important class of causal estimands, indexed by a specified function of the propensity score that reweights the covariate distribution \cite{wang2025rate}. This class contains familiar targets such as the average treatment effect (ATE), the average treatment effect on the treated (ATT), and various overlap-based estimands as special cases, while also allowing investigators to focus on scientifically relevant subpopulations. This flexibility is particularly important in settings with limited overlap or heterogeneous covariate distributions, where standard estimands may be unstable or poorly aligned with the scientific question \cite{wang2025rate}. From a methodological perspective, the broader WATE class also presents additional challenges, because these estimands depend directly on the propensity score and thereby induce a more intricate coupling between nuisance components in the EIF. Although TMLE methodology is well developed for familiar special cases, extensions to more general weighting schemes have received less attention.

Despite its conceptual appeal and widespread use, the theoretical justification of TMLE is not fully self-contained. In the classical implementation, one iteratively updates an initial estimate along a fluctuation submodel until the empirical EIF equation is approximately satisfied \cite{vanderLaanRubin2006TMLE}. Standard efficiency arguments for this procedure are post hoc. Rather than deriving asymptotic efficiency directly from the targeting dynamics, they proceed by isolating two properties of the final estimator that would suffice for efficiency: approximate solution of the EIF equation, and negligibility of the associated second-order remainder in the von Mises expansion. These properties are then verified or assumed separately. In particular, convergence of the targeting algorithm is typically assumed except in special settings where one-step convergence can be established directly \cite{vanderLaanRubin2006TMLE,vanderlaan2006tmle}; see also \cite{li2025targeted} for a recent convergence analysis under additional strong regularity conditions. Likewise, control of the second-order remainder is usually imposed as a separate condition, or is addressed by higher-order targeting procedures that add further iterative structure \cite{carone2014higher,vanderlaan2021higherordertargetedmaximum}. As a result, existing theory does not provide an end-to-end account of why the targeting procedure itself should produce an efficient estimator.

A natural route to such an account is through the universal least favorable path (ULFP). Unlike a classical least favorable submodel, which is defined locally at a single distribution, a ULFP defines a trajectory whose score agrees with the EIF at every point along the path \cite{one_step_tmle}. This viewpoint gives rise to a one-step TMLE: rather than repeatedly re-linearizing around successive updates, one follows a single EIF-driven path and stops when the empirical EIF equation is solved. However, prior work does not establish the main analytic properties needed to justify this construction in general, such as local existence and uniqueness of the induced trajectory, control of its behavior on a nontrivial time interval, or finite-time attainment of the empirical EIF root.

The central contribution of this paper is an end-to-end analysis of one-step TMLE for WATEs. Under explicit local assumptions on the initial nuisance estimate and weight function, we show that the targeting dynamics themselves are sufficient to justify the estimator: they generate a well-defined universal least favorable path, drive the dynamics to a solution of the empirical EIF equation in finite time with high probability, and yield an asymptotically efficient plug-in estimator. In particular, these conclusions are derived from the dynamics rather than obtained by separately postulating convergence of the targeting algorithm or negligibility of the second-order remainder. To our knowledge, this is the first work to establish these three properties together within a single framework for the WATE class.

The paper is organized as follows. \Cref{s:notation} introduces the statistical model, the WATE estimand, the relevant full and restricted efficient influence functions, and the EIF-driven differential equation that defines the universal least favorable path. \Cref{s:main} then develops the main theory in three parts: first, a deterministic local well-posedness result for the ULFP (\Cref{mainthm:ODE-C1}); second, a finite-time targeting result in the empirical-marginal setting relevant for cross-fitting, showing that the score along this path reaches a solution of the empirical EIF equation with high probability (\Cref{mainthm:det-bracket-one}); and third, an asymptotic linearity and efficiency theorem for the resulting one-step TMLE, together with consistency of the plug-in variance estimator and asymptotically valid Wald inference (\Cref{mainthm:AN_cross}). As an illustrative example, we also verify that the assumptions of \Cref{mainthm:AN_cross} hold for spline-based nuisance estimators under suitable smoothness conditions. Technical proofs are deferred to the appendices.

\subsection*{Acknowledgments} 

This manuscript was written by the authors with the assistance of large language models. This assistance included
suggesting arguments, contributing to drafting and revision, and writing code
for computational checks. P.\ L.\ was partially supported by NSF grant DMS-2450004.

\section{Notation and Preliminaries}\label{s:notation}

Let $d$ be a fixed positive integer. For simplicity, we take the covariate space to be $\mathcal X = [0,1]^d$ and suppress the dependence of $\mathcal X$ on $d$ in the notation. We set $\mathcal A = \{0,1\}$ and $\mathcal Y = \{0,1\}$, and define the sample space as $\mathcal O = \mathcal X \times \mathcal A \times \mathcal Y$. Let $(\mathcal O,\mathcal F)$ denote the measurable space, where $\mathcal F$ is the product $\sigma$-algebra on $\mathcal O$. 

Let $O=(X,A,Y)$ denote a generic $\mathcal O$–valued random variable with law $P$. We write $P_X$ for the marginal distribution of $X$ under $P$, and similarly for the marginals of $A$ and $Y$. Let $O_1,O_2,\dots$ be independent and identically distributed (i.i.d.) $\mathcal O$–valued random variables with common law $P^*$, where $P^*$ denotes the true data-generating distribution on $\mathcal O$. For each $i\ge1$, write $O_i=(X_i,A_i,Y_i)$.

Let $n$ denote the sample size. We emphasize that the dimension $d$ is fixed and does not depend on $n$. Let $\P=(P^*)^{\otimes\mathbb N}$ denote the infinite product measure on $\mathcal O^{\mathbb N}$ governing the sequence $(O_i)_{i\ge1}$. All stochastic convergence statements are understood with respect to $\P$. For measurable functions $f\colon\mathcal O\to\mathbb R$, define the empirical measure $\P_n = \frac{1}{n}\sum_{i=1}^n \delta_{O_i}$, so that 
\[
\P_n[f] = \frac{1}{n}\sum_{i=1}^n f(O_i).
\]
We also write 
\[
P^*[f] = \int f(o)\, dP^*(o) = \E^*[f(O)], 
\]
where $\E^*[\cdot]$ denotes expectation under $P^*$. Unless otherwise specified, $\E[\cdot]$ denotes expectation under the distribution $P$.

\subsection{Statistical Model}\label{s:stat_model}
Let $\mu$ be a fixed $\sigma$-finite dominating measure on $(\mathcal O,\mathcal F)$. In our setup, $\mu
= \lambda_{\mathcal X} \otimes \nu_{\mathcal A} \otimes \nu_{\mathcal Y}$, where $\lambda_{\mathcal X}$ denotes Lebesgue measure on $\mathcal X$, and $\nu_{\mathcal A},\nu_{\mathcal Y}$ denote the counting measures on $\mathcal A$ and $\mathcal Y$, respectively. We begin by defining a full nonparametric model, which serves as the basis for our ultimate estimation result. We also define a  fixed-marginal model, which plays an important role in our results on the existence and convergence of TMLE.

\subsubsection{Full Model.}
We define the full nonparametric model as
\[
\mathcal M_{\full} = \left\{ P \text{ a probability measure on } (\mathcal O,\mathcal F) : P \ll \mu \right\}.
\]
For each $P\in\mathcal M_{\full}$, let 
$p=dP/d\mu$ denote its density with respect to $\mu$. Further, we define measurable versions of the conditional expectations
\[
q_a(x) = \E[Y \mid A=a,X=x], 
\qquad a\in\{0,1\},
\]
and the propensity score
\[
e(x) = P(A=1 \mid X=x),
\]
defined $P_X$–almost everywhere.

We refer to $q_1(\cdot), q_0(\cdot)$ and $e(\cdot)$ as the \emph{nuisance functions}. While they are not themselves parameters of primary inferential interest, they are essential for identification and efficient estimation of the target parameter defined in \Cref{s:wate_notation}. We collect the nuisance functions into the vector $U=(q_1,q_0,e)$,
and write $U^*=(q_1^*,q_0^*,e^*)$ for the nuisance functions induced by $P^*$. 
For the ODE-based least favorable flow developed below, we restrict attention to distributions $P$ for which the nuisance functions $U=(q_1,q_0,e)$ admit continuous versions on $\mathcal X$, and we identify these versions with elements of $C(\mathcal X;[0,1])^3$ equipped with the norm
\[
\|U\|_\infty = \max\{\|q_1\|_\infty,\|q_0\|_\infty,\|e\|_\infty\}.
\]

Let $p_X$ denote the marginal density of $X$ with respect to
$\lambda_{\mathcal X}$. For $P\in\mathcal M_{\full}$, the density admits the factorization
\begin{equation}\label{e:factorization}
p(x,a,y) = p_X(x)\, g(a\mid x)\, f(y\mid a,x),
\end{equation}
where
\[
g(a\mid x) = P(A=a\mid X=x),
\qquad
f(y\mid a,x) = P(Y=y\mid A=a,X=x).
\]
Any $P\in\mathcal M_{\full}$ is uniquely determined by the marginal $P_X$ together with the conditional laws of $A|X$ and $Y|A,X$, which, in the binary case, are encoded by the nuisance collection $U=(q_1,q_0,e)$ through chosen measurable versions.

\subsubsection{Fixed-Marginal Model.}\label{s:restricted}
Let $Q$ be a probability measure on $\mathcal X$.
For any nuisance collection $U=(q_1,q_0,e)\in C(\mathcal X;[0,1])^3$,
define the conditional law of $(A,Y)$ given $X=x$ by
\[
p_U(a,y\mid x)
=
e(x)^a\bigl(1-e(x)\bigr)^{1-a}
\bigl(q_a(x)\bigr)^y\bigl(1-q_a(x)\bigr)^{1-y},
\qquad (a,y)\in\{0,1\}^2.
\]
We then define $\widetilde P(Q,U)$ to be the unique probability measure on
$\mathcal O$ such that $X\sim Q$ and, conditional on $X=x$, the pair $(A,Y)$
has law $p_U(\cdot,\cdot\mid x)$.
Equivalently, for every bounded measurable $h\colon \mathcal O\to\mathbb R$,
\[
\widetilde P(Q,U)[h]
=
\int_{\mathcal X}\sum_{a\in\{0,1\}}\sum_{y\in\{0,1\}}
h(x,a,y)\,p_U(a,y\mid x)\,dQ(x).
\]

We then define the fixed-marginal model by
\[
\mathcal M(Q)=\{\widetilde P(Q,U):U\in C(\mathcal X;(0,1))^3\}.
\]
When $Q$ is absolutely continuous, this is a restricted version of the full model
$\mathcal M_{\full}$.

We often abbreviate $\mathcal M = \mathcal M(Q)$ when discussing a generic restricted model.

\begin{remark}\label{r:pair}
With a slight abuse of notation, we parametrize generic elements of the model by pairs $(Q,U)$, where $Q$ denotes the marginal law of $X$ and $U=(q_1,q_0,e)$ denotes the nuisance collection. 
At the true marginal, we write $Q=P_X^*$.
\end{remark}

\subsection{Weighted Average Treatment Effect}\label{s:wate_notation}

We adopt the potential outcomes framework and state the standard causal identification assumptions
\cite{Rubin1974,RosenbaumRubin1983,Neyman1990,Rosenbaum2002,ImbensRubin2015}. 

\begin{assumption}[Causal Identification]\label{a:identify}
Let $\{Y(a):a\in\mathcal A\}$ denote the potential outcomes defined on the same probability space as $(X,A,Y)$. The following conditions hold:
\begin{enumerate}
\item (Consistency/SUTVA): $Y = Y(A)$ almost surely.
\item (Conditional exchangeability): For all $a\in\mathcal A$, $Y(a) \ind A \mid X$.
\item (Strong Positivity): There exists $\eta\in(0,1/2)$ such that
\[
\eta \leq e^*(x) \leq 1-\eta
\quad \text{for $P_X^*$-almost every } x\in\mathcal X,
\]
where $e^*(x)=P^*(A=1\mid X=x)$ denotes the true propensity score.
\end{enumerate}
\end{assumption}

To define the target parameter, we first introduce the conditional average treatment effect (CATE) under the true law,
\[
\tau^*(x) = \E^*\!\left[Y(1)-Y(0)\mid X=x\right].
\]
Under \Cref{a:identify}, the CATE is identified from the observed data via
\[
\tau^*(x)=q_1^*(x)-q_0^*(x)
\qquad \text{for }P_X^*\text{-almost every }x\in\mathcal X,
\]
where $q_a^*(x)=\E^*[Y\mid A=a,X=x]$.

For a generic observed-data law $P\in\mathcal M_{\full}$, define the observed-data blip
\[
b_P(x)=q_1(x)-q_0(x).
\]
Throughout the paper, $\lambda\colon[0,1]\to\mathbb R_{\ge0}$
is a known, pre-specified Borel-measurable weight function satisfying
\[
\sup_{u\in[0,1]}\lambda(u)<\infty.
\]
We write $\dot\lambda(t)=d\lambda(t)/dt$ wherever the derivative
exists; the required local smoothness assumptions are stated below.
Further, let 
\[
\Omega(P)=\E\!\left[\lambda\bigl(e(X)\bigr)\right].
\]
The weighted average treatment effect (WATE) corresponding to fixed $\lambda$ is defined as
\begin{equation}\label{eq:wate_def}
\psi(P)
=
\frac{\E\!\left[\lambda\bigl(e(X)\bigr)b_P(X)\right]}
{\E\!\left[\lambda\bigl(e(X)\bigr)\right]}
=
\frac{\E\!\left[\lambda\bigl(e(X)\bigr)b_P(X)\right]}
{\Omega(P)}.
\end{equation}
We emphasize that $\psi(P)$ is defined, as a statistical functional of the observed-data law, for laws $P$ such that $\Omega(P)>0$ and $0<e(X)<1$ for $P_X$-almost every $x$ with $\lambda(e(x))>0$. Under Assumption~\ref{a:identify}, $\psi(P^*)$ coincides with the corresponding causal estimand defined via potential outcomes; the identification result is standard and hence omitted. Dependence of $\psi$ on $\lambda$ is suppressed in the notation.

We further note that \eqref{eq:wate_def} may equivalently be written as
\[
\psi(P) = \E_{P_{\lambda}}[b_P(X)],
\]
where we define $dP_\lambda(x)$ as the reweighted covariate distribution
\[
dP_\lambda(x) = \frac{\lambda(e(x))}{\Omega(P)}\, dP_X(x).
\]
Hence, $\psi(P)$ is the average observed-data blip under the distribution $dP_\lambda$. Under Assumption~\ref{a:identify}, the corresponding quantity at $P^*$ is the average treatment effect in that target population.
Different choices of $\lambda$ generate different target populations, including the ATE, ATT, average treatment effect on the controls (ATC), overlap, trimmed, entropy, and beta–weighted estimands summarized in Table~\ref{table::WATEs}. Our results apply to WATEs with fixed weight functions $\lambda$
that are bounded on $[0,1]$ and, on the propensity-score interval
$[\eta,1-\eta]$, are three times continuously differentiable,
bounded away from zero, and have bounded first, second, and third
derivatives (see \Cref{as::positivity} below). 

This includes ATE, ATT, ATC, average treatment effect on the overlap population (ATO), average treatment effect on the entropy population (ATEN), and average treatment effect on the beta-weighted population (ATB) with $\nu_1,\nu_2\ge1$, but excludes non-smooth choices such as the trimmed average treatment effect (TrATE$(\alpha)$), average treatment effect on the matching population (ATM), and average treatment effect for the population induced by trapezoidal weighting (ATTZ$(K)$) unless replaced by smooth approximations.

One particularly natural case is the ATO, corresponding to $\lambda(t)=t(1-t)$. To see this, consider a stochastic intervention that perturbs treatment assignment by a common shift on the log-odds scale:
\[
\logit\bigl(e_\delta(x)\bigr)=\logit\bigl(e(x)\bigr)+\delta,
\qquad \delta\approx 0.
\]
Then
\[
\left.\frac{d}{d\delta}e_\delta(x)\right|_{\delta=0}=e(x)\bigl(1-e(x)\bigr),
\]
so units with propensity scores near $1/2$ are the most affected by a small log-odds perturbation, while units with propensity scores near $0$ or $1$ are affected the least. Consequently,
\[
\frac{\left.\frac{d}{d\delta}\E\!\left[q_0(X)+e_\delta(X)\{q_1(X)-q_0(X)\}\right]\right|_{\delta=0}}
{\left.\frac{d}{d\delta}\E[e_\delta(X)]\right|_{\delta=0}}
=
\frac{\E\!\left[e(X)(1-e(X))\{q_1(X)-q_0(X)\}\right]}
{\E[e(X)(1-e(X))]},
\]
which is exactly the ATO. Thus, the ATO has a marginal effect interpretation: it is a weighted average of treatment effects, where the largest weight is placed on individuals whose treatment probabilities are most responsive to a small uniform shift in treatment log-odds. 

\begin{table}[ht]
\caption{Common weighted average treatment effects (WATEs) generated by  fixed weight functions $\lambda(t)$ applied to the propensity score, following \cite[Table 1]{wang2025rate}. The third column reports the names of the corresponding weighting schemes. The associated treated and control weights are $w_1(t)=\lambda(t)/t$ and $w_0(t)=\lambda(t)/(1-t)$. Applied at $t=e(X)$, they reweight treated and control units to target the same covariate distribution, with population weight $\lambda(e(X))$. The final column indicates whether the regularity conditions imposed in this paper are satisfied, in which case the corresponding one-step TMLE has theoretical support.}
\label{table::WATEs}
\begin{tabular}{lllll}
\hline
Estimand & $\lambda(t)$ & Weighting scheme & Ref. & Support? \\
\hline
ATE 
& $1$ 
& IPW
& \cite{Rubin1974} 
& Yes \\

ATT 
& $t$ 
& IPW (treated) 
& \cite{Hirano2003} 
& Yes \\

ATC 
& $1-t$ 
& IPW (controls) 
& \cite{TaoFu2019} 
& Yes \\

TrATE$(\alpha)$ 
& $\mathbb{I}(\alpha\le t\le 1-\alpha)$ 
& IPW trimming 
& \cite{Crump2009} 
& No \\

TrATE$_2(\alpha)$ 
& $\Phi_\varepsilon(1-\alpha-t)\,\Phi_\varepsilon(t-\alpha)$ 
& Smooth trimming 
& \cite{YangDing2018} 
& Yes \\

ATM 
& $\min\{t,1-t\}$ 
& Matching 
& \cite{LiGreene2013} 
& No \\

ATTZ$(K)$
& $\min\{1,\,K\min(t,1-t)\},\ K>2$
& Trapezoidal 
& \cite{Mao2019}
& No \\

ATO 
& $t(1-t)$ 
& Overlap 
& \cite{Li2018Overlap} 
& Yes \\

ATEN 
& $-t \log t - (1-t) \log (1-t)$ 
& Entropy 
& \cite{zhou2020propensity} 
& Yes \\

ATB 
& $t^{\nu_1-1}(1-t)^{\nu_2-1}$
& Beta 
& \cite{matsouaka2024causal} 
& $\nu_1,\nu_2\ge1$ \\

\hline
\end{tabular}

\vspace{0.2cm}
\parbox{\linewidth}{\raggedright
\footnotesize
\textit{Notes:} IPW stands for inverse propensity weights. ATE, ATT, ATC and ATO are special cases of ATB with $(\nu_1,\nu_2)=(1,1)$, $(2,1)$, $(1,2)$ and $(2,2)$, respectively. The function  
$\Phi_\varepsilon$ denotes the Gaussian CDF with mean $0$ and variance $\varepsilon^2$. 
As $\varepsilon\to0$,  the weight for TrATE$_2(\alpha)$  converges to that of TrATE$(\alpha)$.}
\end{table}

\begin{remark}
Let $Q=P_X$, and let $U$ denote chosen continuous versions of $(q_1,q_0,e)$ whenever such versions exist. Then $P=\widetilde P(Q,U)$, and we may equivalently regard $\psi$ as a functional of $(Q,U)$. In the restricted model, $Q$ is fixed and only $U$ varies.
\end{remark}

\subsection{Efficient Influence Functions} 
For the reader's convenience, we review efficient influence functions (EIFs) and some related concepts in \Cref{s:influence_appendix}. 
In what follows, we write the EIF of the WATE estimand under the full and restricted statistical models. For a generic law $P$ with nuisances $(q_1, q_0, e)$, recall that
\[
b_P(x)=q_1(x)-q_0(x).
\]
We first define the augmented inverse probability weighted pseudo-outcome for $b_P(x)$ as the function
$
\phi:\mathcal O\to\mathbb R
$
given by
\begin{equation}\label{e:aipw}
\phi(o) = \frac{a}{e(x)} \big( y - q_1(x)\big) - \frac{1-a}{1-e(x)} \big( y - q_0(x)\big) + \big( q_1(x)-q_0(x) \big).
\end{equation}
If $P$ satisfies $0<e(X)<1$ almost surely, then
\[
\E\big[\phi(O) \mid X\big] = b_P(X).
\]
The following expression is standard and can be verified by direct pathwise differentiation; see also \cite{wang2025rate}.

\begin{proposition}[Full-model EIF]
\label{prop:EIF_full}
Let $P\in\mathcal M_{\full}$ be such that there exists $\eta\in(0,1/2)$ with
\[
\eta \le e(x) \le 1-\eta
\qquad\text{for }P_X\text{-almost every }x\in\mathcal X,
\]
and $\Omega(P)>0$. Assume that $\lambda$ is bounded on $[0,1]$
and continuously differentiable on an open neighborhood of the
essential range of $e(X)$. Define $\phi$ as in \eqref{e:aipw}. The EIF of $\psi$ at $P$ relative to $\mathcal M_{\full}$ is
\begin{equation}\label{eqn::fullEIF}
D_{\full}^*(o; P) = \frac{\lambda(e(x))}{\Omega(P)} \big( \phi(o) - \psi(P) \big) + \frac{\dot{\lambda}(e(x))}{\Omega(P)} \big( b_P(x) - \psi(P)\big)\big(a - e(x)\big),
\end{equation}
where $\Omega(P)=\E[\lambda(e(X))]$. Moreover,
$D_{\full}^* = D_q^* + D_e^* + D_X^*$, with
\begin{align}
\begin{split}\label{e:EIF}
D_q^*(o;P) &= \frac{\lambda(e(x))}{\Omega(P)} \left[ \frac{a}{e(x)}(y-q_1(x)) - \frac{1-a}{1-e(x)}(y-q_0(x)) \right], \\[6pt]
D_e^*(o;P) &= \frac{\dot{\lambda}(e(x))}{\Omega(P)} \bigl(b_P(x)-\psi(P)\bigr)(a-e(x)), \\[6pt]
D_X^*(o;P) &= \frac{\lambda(e(x))}{\Omega(P)} \bigl(b_P(x)-\psi(P)\bigr).
\end{split}
\end{align}
The displayed formula also defines an algebraic extension of $D_{\full}^*(o;P)$ to distributions $P\in\mathcal M(Q)$, including the case where $Q$ is discrete and $P\notin\mathcal M_{\full}$.
\end{proposition}
\begin{remark}
Although the full-model EIF contains an $X$--component reflecting perturbations of the marginal law $P_X$, that  first-order contribution can be accounted for when estimating $\psi$ by replacing $P_X$  by its empirical distribution (as we will prove below). Accordingly, our targeting step defined below fluctuates only the conditional law of $(A,Y)\mid X$ while holding the marginal law of $X$ fixed. When the marginal law is written explicitly, we view $\psi$ as a functional of $(Q,U)$; when $Q$ is fixed and understood, we abbreviate $\psi(Q,U)$ by $\psi(U)$.
\end{remark}

When $Q$ is absolutely continuous, the restricted tangent space excludes
variations in $P_X$, so the restricted-model EIF is obtained by projecting
$D_{\full}^*$ onto the tangent space of the restricted model; we recall the definition of tangent spaces in \Cref{s:tangent}. The result of \Cref{prop:EIF_restricted} then follows
immediately from \Cref{prop:EIF_full}, and the derivation is thus omitted. When $Q$ is discrete, the tangent-space
interpretation is not available within $\mathcal M_{\full}$, because
$\widetilde P(Q,U)\notin\mathcal M_{\full}$ in general. In that case, we still
use the expression \eqref{e:dstar} below as an algebraic extension of the restricted-model EIF
formula, and all subsequent uses of $D^*(o;P)$ for discrete $Q$ are to be
understood in this sense.

\begin{proposition}[Restricted-model EIF]
\label{prop:EIF_restricted}
Let $P\in\mathcal M$ be such that there exists $\eta\in(0,1/2)$ with
\[
\eta \le e(x) \le 1-\eta
\qquad\text{for }P_X\text{-almost every }x\in\mathcal X,
\]
and $\Omega(P)>0$. Assume that $\lambda$ is bounded on $[0,1]$
and continuously differentiable on an open neighborhood of the
essential range of $e(X)$. Define
\begin{equation*}
D^*(o;P) = D_{\full}^*(o;P) - E_P\!\big[D_{\full}^*(O;P)\mid X\big].
\end{equation*}
Equivalently,
\begin{align}\label{e:dstar}
D^*(o;P) &= D_q^*(o;P)+D_e^*(o;P) \\
&= \frac{\lambda(e(x))}{\Omega(P)}\Big(\phi(o)-b_P(x)\Big)
 + \frac{\dot{\lambda}(e(x))}{\Omega(P)}\bigl(b_P(x)-\psi(P)\bigr)(a-e(x)). \nonumber
\end{align}
When $Q$ is absolutely continuous with respect to Lebesgue measure, $D^*(\cdot;P)$ is the efficient influence function of $\psi$ at $P$ relative to $\mathcal M$.
\end{proposition}

Let $U_t=(q_{1,t},q_{0,t},e_t)$ be a family of nuisance triples indexed by $t$
in an interval $I\subset\mathbb R$. We write
\[
P_t=\widetilde P(Q,U_t)
\]
for the induced distribution on $\mathcal O$, and we let $\E_t[\cdot]$ denote
expectation with respect to $P_t$. We also abbreviate
\[
p_t(a,y\mid x)=p_{U_t}(a,y\mid x),
\qquad (a,y)\in\{0,1\}^2.
\]

For each $t$, we write 
\[
D^*(o;U_t) = \frac{\lambda(e_t(x))}{\Omega_t}\bigl(\phi_t(o)-\tau_t(x)\bigr) + \frac{\dot\lambda(e_t(x))}{\Omega_t}\bigl(\tau_t(x) -\psi_t\bigr)\bigl(a - e_t(x)\bigr),
\]
where
\[
\phi_t(o) = \frac{a}{e_t(x)}\big( y-q_{1,t}(x)\big) - \frac{1-a}{1-e_t(x)}\big( y-q_{0,t}(x)\big) + \tau_t(x),
\]
and
\[
\Omega_t=\E_Q\!\left[\lambda\bigl(e_t(X)\bigr)\right],\qquad
\tau_t=q_{1,t}-q_{0,t},\qquad
\psi_t=\frac{\E_Q\!\left[\lambda\bigl(e_t(X)\bigr)\tau_t(X)\right]}{\Omega_t}.
\]

In the fixed-marginal model, the one-step TMLE is constructed by evolving an initial nuisance estimate $U_0$ along a path $t\mapsto U_t$ for which the conditional score coincides with the restricted-model EIF at every point on the path. This is the universal least favorable path (ULFP) viewpoint of \cite{one_step_tmle}. In the present setting, because the marginal law $Q$ of $X$ is held fixed, only the conditional law of $(A,Y)\mid X$ is updated, and the corresponding path is most naturally described through the nuisance functions $q_{1,t},q_{0,t},e_t$.

The next lemma records the resulting differential equations and the key score identities they imply. The ODE system in \eqref{eq:ATT-ODE} is the evolution equation obtained by requiring that the derivative of the conditional log-density along the path agree with the EIF $D^*(\cdot;U_t)$. Consequently, along this path the derivative of the empirical conditional log-likelihood is exactly the empirical EIF mean. Thus, if the path can be followed to a time $\hat t$ at which this derivative vanishes, then the resulting updated distribution $P_{\hat t}$ solves the empirical EIF equation, which is the defining targeting property of TMLE. The mean-zero identities in \eqref{eq:path-score} show that $D^*(\cdot;U_t)$ is indeed a valid score in the restricted model at each time $t$, both conditionally on $X$ and marginally. These facts provide the basic link between the dynamical system studied below and the statistical targeting step carried out by one-step TMLE.

\begin{lemma}\label{lm::ODE_flow_derivation}
Fix a probability measure $Q$ on $\mathcal X$, let $I\subset \mathbb R$ be an
open interval containing $0$, and let
$U_0=(q_1,q_0,e)\in C(\mathcal X;(0,1))^3$.
Consider the system of ordinary differential equations
\begin{align}
\begin{split}
\frac{d}{dt}q_{1,t}(x)
&= \frac{q_{1,t}(x)\bigl(1-q_{1,t}(x)\bigr)}{\Omega_t}\cdot
\frac{\lambda\bigl(e_t(x)\bigr)}{e_t(x)},\\
\frac{d}{dt}q_{0,t}(x)
&= - \frac{q_{0,t}(x)\bigl(1-q_{0,t}(x)\bigr)}{\Omega_t}\cdot
\frac{\lambda\bigl(e_t(x)\bigr)}{1-e_t(x)},\\
\frac{d}{dt}e_t(x)
&= \frac{e_t(x)\bigl(1-e_t(x)\bigr)\dot\lambda\bigl(e_t(x)\bigr)}{\Omega_t}
\bigl(\tau_t(x)-\psi_t\bigr),
\end{split}
\label{eq:ATT-ODE}
\end{align}
with initial condition $U_{t=0}=U_0$.
Suppose that $t\mapsto U_t$ is a map from $I$ into $C(\mathcal X;(0,1))^3$
such that $\Omega_t>0$ for all $t\in I$, and for every $x\in\mathcal X$ the map
\[
t\mapsto \bigl(q_{1,t}(x),q_{0,t}(x),e_t(x)\bigr)
\]
is $C^1$ on $I$ and solves \eqref{eq:ATT-ODE}.
Then for all $t \in I$ and $o=(x,a,y) \in \mathcal O$, 
\begin{equation}
\frac{d}{dt}\log p_t(a,y\mid x)=D^*(o;U_t), \quad 
\E_t\!\big[D^*(O;U_t)\mid X\big]=0,\quad
\E_t\!\big[D^*(O;U_t)\big]=0,
\label{eq:path-score}
\end{equation}
where $D^*(\cdot;U_t)$ is shorthand for
$D^*(\cdot;\widetilde P(Q,U_t))$.
\end{lemma}

For future use, we define the following notation. Let $Q$ be a probability measure on $\mathcal X$, and let
\[
U=(q_1,q_0,e)\in C(\mathcal X;(0,1))^3.
\]
Define
\[
\tau(U)(x)=q_1(x)-q_0(x), \qquad 
\Omega(Q,U)=\E_Q\!\left[\lambda\bigl(e(X)\bigr)\right].
\]
Whenever $\Omega(Q,U)>0$, we set
\[
\psi(Q,U)=
\frac{\E_Q\!\left[\lambda\bigl(e(X)\bigr)\tau(U)(X)\right]}{\Omega(Q,U)}.
\]
For such $U$, we then define
\[
F(Q,U)=\big(F_1(Q,U),F_0(Q,U),F_e(Q,U)\big),
\]
where
\begin{align}
\begin{split}
F_1(Q,U)(x)
&=
\frac{q_1(x)\bigl(1-q_1(x)\bigr)\lambda\bigl(e(x)\bigr)}
{\Omega(Q,U)\,e(x)},\\
F_0(Q,U)(x)
&=
-\frac{q_0(x)\bigl(1-q_0(x)\bigr)\lambda\bigl(e(x)\bigr)}
{\Omega(Q,U)\,\bigl(1-e(x)\bigr)},\\
F_e(Q,U)(x)
&=
\frac{e(x)\bigl(1-e(x)\bigr)\dot\lambda\bigl(e(x)\bigr)}
{\Omega(Q,U)}
\bigl(\tau(U)(x)-\psi(Q,U)\bigr).
\end{split}
\end{align}
When $Q$ is fixed, we often abbreviate $\Omega(Q,U)$ and $\psi(Q,U)$ by
$\Omega(U)$ and $\psi(U)$, respectively, and likewise write $F(U)$ for
$F(Q,U)$.

\section{Main Results}\label{s:main}

\subsection{Notation}\label{s:assumption_notation}
We begin by introducing some necessary notation for our assumptions. 
Let $r>0$ be a parameter. Given $U_0=(q_1,q_0,e) \in C(\mathcal X)^3$, we define the closed ball 
\[ \mathcal B_r (U_0)=\{U\in C(\mathcal X)^3: \|U-U_0\|_\infty\le r\}\]
and the open ball
\[
\mathcal{B}_r^\circ(U_0) = \{U \in C(\mathcal{X})^3 : \|U - U_0\|_\infty < r\}.
\]
When the choice of $U_0$ is clear, we sometimes abbreviate $\mathcal B_r = \mathcal B_r (U_0)$ and $\mathcal{B}_r^\circ=\mathcal{B}_r^\circ(U_0)$. We also set $M_0=\|U_0\|_\infty$.

Further, given  $\eta \in (0, 1/4)$ and a weight function $\lambda\colon [0,1] \to \mathbb{R}_{\ge 0}$, we define \begin{align*}
&\qquad\qquad\qquad\Lambda_{\min}=\inf_{u\in[\eta,1-\eta]}\lambda(u),\qquad
\Lambda_{\max}=\sup_{u\in[\eta,1-\eta]}\lambda(u),\\
&\dot\Lambda_{\max}=\sup_{u\in[\eta,1-\eta]}|\dot\lambda(u)|, \quad \Lambda_{2,\max}=\sup_{u\in[\eta,1-\eta]}|\lambda''(u)|,\quad \Lambda_{3,\max}=\sup_{u\in[\eta,1-\eta]}|\lambda'''(u)|.
\end{align*}
For brevity, we let $\mfc$ denote the largest positive real number such that
\[
\mfc \le
\min
\{
\Lambda_{\min}, 
\eta 
\}
\]
and 
\[
\max
\{
\Lambda_{\max},
\dot\Lambda_{\max},
\Lambda_{2,\max}, \Lambda_{3,\max}, \eta 
\}
\le \mfc^{-1}.
\]

\subsection{Assumptions}\label{s:assumptions} 
The analysis below is local in nature. Toward this end,  the assumptions play three conceptually distinct roles.

First, Assumption~\ref{as::positivity} is a regularity condition. It keeps the nuisance functions away from the boundary and requires the weight function \(\lambda\) to be sufficiently smooth on the relevant propensity-score range. These are the conditions that make the ordinary differential equation defining the universal least favorable path locally well behaved.

Second, Assumptions~\ref{as::square_bound}--\ref{ass:mu0-small} are initialization conditions for the empirical targeting problem. Roughly speaking, they require the score at \(t=0\) to be nondegenerate, the initial conditional law to be locally close to the truth, and the initial score bias to be small enough that a nearby zero of the empirical score can be found along the path.

Third, Assumption~\ref{as::initial_convergence} is the usual nuisance function rate condition needed to turn the targeted construction into asymptotic linearity and efficiency under cross-fitting.

The constants \(\mfc\), \(c_\init\), and \(\delta_\init\) are bookkeeping quantities that summarize these local regularity and initialization requirements on a common scale. Their exact powers are used only in the proofs. For purposes of reading the main results, it is enough to remember that \(\mfc\) encodes uniform positivity and smoothness, while \(c_\init\) encodes nondegeneracy of the initial score.

\begin{assumption}\label{as::positivity}
Let $U_0=(q_{1,0},q_{0,0},e_0)\in C(\mathcal X;[0,1])^3$ and
$\lambda\colon [0,1] \to \mathbb{R}_{\ge 0}$ be given.
There exists $\eta\in(0,1/4)$ such that for all
$U=(q_1,q_0,e)\in\mathcal B_{2\eta}(U_0)$, we have
\[
q_1(x),\,q_0(x),\,e(x)\in[\eta,1-\eta]
\qquad\text{for all }x\in\mathcal X.
\]
Further, $\lambda$ is bounded on $[0,1]$, is $C^3$ on
$[\eta,1-\eta]$, and
\[
\Lambda_{\min} >0, \qquad
\Lambda_{\max} + \dot\Lambda_{\max} + \Lambda_{2,\max} + \Lambda_{3,\max}< \infty.
\]
\end{assumption}

\begin{assumption}\label{a:pstar}
Let $P^*$ be a given probability measure on $\mathcal O$.
Assume that $P^*\ll \mu$, and let
\[
p^*=\frac{dP^*}{d\mu}
\]
denote its density with respect to $\mu$.
\end{assumption}

When the previous assumption holds and the nuisance functions corresponding to $P^*$ admit continuous versions on $\mathcal X$, we fix one such triple and denote it by
\[
U^*=(q_1^*,q_0^*,e^*).
\]
Given $U_0 \in C (\mathcal X;[0,1])^3$ and a probability measure $Q$ on $\mathcal X$, we recall that the induced measure $\widetilde P(Q, U_0)$ was defined in \Cref{s:stat_model}. 

\begin{assumption}\label{as::square_bound}
Let a probability measure $Q$ on $\mathcal X$ 
be given, and suppose that $Q$ is either absolutely  continuous with respect to Lebesgue measure or pure point. 
There exists a constant $c_\init \in (0,1)$ such that 
\[
\E_0\big[\big(D^*(O;\widetilde{P}(Q, U_0))\big)^2\big] \ge c_\init,
\]
where  $\E_0$ denotes expectation with respect to $\widetilde P (Q, U_0)$. 
\end{assumption}

For the next assumption, we write $\tv(\nu_1,\nu_2)$ for the total variation distance between two probability measures $\nu_1$ and $\nu_2$ on a common measurable space. When $\nu_1$ and $\nu_2$ are both absolutely continuous with respect to a common dominating measure, with densities $p$ and $q$, this distance is given by
\[
\tv(\nu_1,\nu_2)
=
\frac12 \|p-q\|_1,
\]
where $\|\cdot\|_1$ denotes the $L^1$ norm with respect to that dominating measure. In this case, we sometimes write $\tv(p,q)$ as shorthand for $\tv(\nu_1,\nu_2)$. 
\begin{assumption}\label{as::TV}
Assume $U^*\in \mathcal B_\eta$ and
\[
\E_Q \big[\tv\big(p^*(\cdot\mid X),p_{U_0}(\cdot\mid X)\big)\big]
\leq \frac{\mfc^{10} c_\init}{600},
\]
where $p^*(\cdot\mid x)$ is the conditional law of $(A,Y)$ given $X=x$ under $P^*$, and
for $U_0=(q_{1,0},q_{0,0},e_0)$,
\[
p_{U_0}(a,y\mid x)
=
\bigl(e_0(x)\bigr)^a\bigl(1-e_0(x)\bigr)^{1-a}
\bigl(q_{a,0}(x)\bigr)^y\bigl(1-q_{a,0}(x)\bigr)^{1-y}.
\]
\end{assumption}

The next assumption controls the \emph{initial score bias}. In the empirical-marginal setting considered below, where $Q$ is taken to be the empirical distribution of the evaluation-fold covariates, $\mu_0$ is the conditional mean of the score  given those covariates and the initial nuisance estimate $U_0$. Thus, the assumption requires the score at $t=0$ to start sufficiently close to zero on the local scale $\delta_\init$. We now define 
\begin{equation}\label{e:Tdef}
\delta_\init=\frac{c_\init\mfc^{20}}{10^6 }.
\end{equation}
\begin{assumption}\label{ass:mu0-small}
Let
\[
\mu_0
=
\E_Q\!\left[
\sum_{a,y\in\{0,1\}}
D^*\!\big((X,a,y);\widetilde P(Q,U_0)\big)\,p^*(a,y\mid X)
\right].
\]
We assume
\[
|\mu_0| \le \frac{c_{\mathrm{init}}\delta_\init}{8},
\]
where $\delta_\init$ is defined in \eqref{e:Tdef}.
\end{assumption}
\begin{definition}
Fix a probability measure $P^*$ on $\mathcal O$ and a function
$\lambda\colon [0,1]\to\mathbb R_{\ge 0}$.
We will say that a quadruple $(P^*,\lambda,Q,U_0)$ satisfies the local bracketing conditions with constants $(\mfc,c_\init)$ if Assumptions \ref{as::positivity}, \ref{as::square_bound}, \ref{as::TV}, and \ref{ass:mu0-small} hold with that choice of $P^*$, $\lambda$, $Q$, and $U_0$.
When $P^*$ and $\lambda$ are fixed by context, we suppress them and simply say that the pair $(Q,U_0)$ satisfies the local bracketing conditions with constants $(\mfc,c_\init)$.
\end{definition}

For $s\in[1,\infty)$, we use $\|\cdot\|_{s,*}$ to denote the $L^s(P_X^*)$ norm on $\mathcal X$. 
When $f$ is random, this norm is still taken only over the $x$-variable, and convergence statements such as $o_p(\cdot)$ refer to the randomness of the sample.

We let $(q_{1,0,n})_{n=1}^\infty$, $(q_{0,0,n})_{n=1}^\infty$, and $(e_{0,n})_{n=1}^\infty$ denote sequences of maps such that, for every $n$,
\[
q_{1,0,n},\,q_{0,0,n},\,e_{0,n}\colon \mathcal O^n\to C(\mathcal X;[0,1]).
\]
We sometimes abbreviate $q_{1,0}=q_{1,0,n}$, $q_{0,0}=q_{0,0,n}$, and $e_0=e_{0,n}$ when the value of $n$ is clear from context. 
\begin{assumption}\label{as::initial_convergence}
Let a probability measure $P^*$ on $\mathcal O$ and estimator sequences
$(q_{1,0,n})_{n=1}^\infty$, $(q_{0,0,n})_{n=1}^\infty$, and $(e_{0,n})_{n=1}^\infty$
be given. As $n\to\infty$, we have
\[
\|q_{1,0,n}-q_1^*\|_{2,*}=o_p(n^{-1/4}),\ 
\|q_{0,0,n}-q_0^*\|_{2,*}=o_p(n^{-1/4}),\ 
\|e_{0,n}-e^*\|_{2,*}=o_p(n^{-1/4}).
\]
\end{assumption}

\subsection{Results}\label{s:results}

We now state our three main results. The first is a deterministic well-posedness result for the universal least favorable path. Informally, it says that once the nuisance functions are uniformly bounded away from $0$ and $1$ and the weight function is smooth on the relevant range, the efficient-influence-function-driven ODE is not merely formal but defines a genuine local trajectory in the restricted model.

\begin{theorem}\label{mainthm:ODE-C1}
Fix a probability measure $Q$ on $\mathcal X$ that is either absolutely continuous with respect to Lebesgue measure or pure point,
$U_0 \in C(\mathcal X;[0,1])^3$, and
$\lambda\colon [0,1] \to \mathbb{R}_{\ge 0}$, and suppose that
\Cref{as::positivity} holds.
Let $\mft_1  = \mfc^6/8$ and $I_1 = ( - \mft_1, \mft_1)$. There exists a unique solution $U\in C^1(I_1; C(\mathcal X)^3)$ to the initial-value problem
\begin{equation}\label{e:main_ode}
U(0)=U_0,
\qquad
\frac{d}{dt}U_t =  F(Q, U_t),
\quad t\in I_1,
\end{equation}
such that for all $t\in I_1$, we have $U_t\in \mathcal B_\eta$.
Additionally, for every $x\in\mathcal X$,
the function $t\mapsto U_t(x)\in\mathbb R^3$ is an element of $C^1(I_1;\mathbb R^3)$ and satisfies
\[
\frac{d}{dt}U_t(x)=  F(Q, U_t)(x)
\]
for all $t \in I_1$.
\end{theorem}

The next theorem concerns the empirical-marginal setting used in cross-fitting. One subsample is used both to define
\begin{equation}\label{e:qempirical}
Q_n=\frac1n\sum_{i=1}^n \delta_{X_i},
\end{equation}
the empirical distribution of its covariates, and to evaluate the empirical score, while the initial nuisance estimate \(U_0\) is constructed from the complementary subsample. Conditional on \(U_0\) and on the covariates entering \(Q_n\), the targeting path is fixed, so only the treatment and outcome variables in that same subsample remain random. Under the local bracketing conditions, the empirical score along this path has, with exponentially high conditional probability, a unique zero near the initial value.

\begin{theorem}\label{mainthm:det-bracket-one}
Fix a probability measure $P^*$ on $\mathcal O$, a sub-$\sigma$-field $\mathcal G_n$, a $\mathcal G_n$-measurable random element $U_0\in C(\mathcal X;[0,1])^3$, and $\lambda\colon [0,1]\to\mathbb R_{\ge 0}$.
Assume that \Cref{a:pstar} holds, that $X_1,\dots,X_n$ are $\mathcal G_n$-measurable, and that conditional on $\mathcal G_n$, the pairs $(A_i,Y_i)$ are independent with
\[
\P\big((A_i,Y_i)=(a,y)\mid \mathcal G_n\big)=p^*(a,y\mid X_i),
\qquad i=1,\dots,n,\ (a,y)\in\{0,1\}^2.
\]
Define $Q_n$ through \eqref{e:qempirical}. 
Let $\mathcal A_n\in\mathcal G_n$ be an event on which the pair $(Q_n,U_0)$ satisfies the local bracketing conditions with constants $(\mfc,c_\init)$.

On $\mathcal A_n$, let $U_t$ be the solution to \eqref{e:main_ode}
provided by \Cref{mainthm:ODE-C1} with the pair $(Q_n,U_0)$, define
\[
p_t(a,y\mid x)=p_{U_t}(a,y\mid x),
\qquad (a,y)\in\{0,1\}^2,
\]
and recall that $\mft_1=\mfc^6/8$.
Define the random function $L_n\colon(-\mft_1,\mft_1)\to\mathbb R$ by
\[
L_n(t)
=
\begin{cases}
\displaystyle \frac{1}{n}\sum_{i=1}^n \log p_t(A_i,Y_i\mid X_i),
& \text{on }\mathcal A_n,\\[1ex]
0,
& \text{on }\mathcal A_n^{\mathrm c},
\end{cases}
\qquad t\in(-\mft_1,\mft_1).
\]

Then on $\mathcal A_n$, the map $t\mapsto L_n(t)$ belongs to
$C^2((-\mft_1,\mft_1))$, and there exist constants 
$\mft_2\in(0,\mft_1)$ and $m_0>0$ depending only on $(\mfc,c_\init)$ such that the following holds. 
Let $\mathcal E_n$ denote the subset of $\mathcal A_n$ on which there exists
a unique $\hat t\in[-\mft_2,\mft_2]$ such that $L_n'(\hat t)=0$.
Then
\[
\mathbf{1}_{\mathcal A_n}\,
\P\!\left(\mathcal E_n^{\mathrm c}\mid \mathcal G_n\right)
\le \frac{1}{m_0}\exp(-m_0 n).
\]
\end{theorem}

Now, we consider the cross-fitted one-step TMLE estimator. For simplicity, assume \(n\) is even and write \(n=2m\). 
Let \(O_1,\dots,O_n\) be i.i.d.\ from \(P^*\).
Define the index sets \(I_0=\{1,\dots,m\}\) and \(I_1=\{m+1,\dots,2m\}\).
For \(k\in\{0,1\}\), define the fold-specific empirical measure
\[
\mathbb P_m^{(k)}[f]=\frac{1}{m}\sum_{i\in I_k} f(O_i),
\]
and note that
\[
\P_n[f]=\frac{1}{2}\mathbb P_m^{(0)}[f]+\frac{1}{2}\mathbb P_m^{(1)}[f].
\]
For later use, define
\[
\mathcal G_m^{(k)}=\sigma\!\big((O_i)_{i\in I_{1-k}},(X_i)_{i\in I_k}\big).
\]

For each \(k\in\{0,1\}\), let \(\hat U_{0}^{(k)}=(\hat q_{1,0}^{(k)},\hat q_{0,0}^{(k)},\hat e_0^{(k)})\) be an initial estimator of the nuisance parameters constructed as
\[
\hat q_{1,0}^{(k)}
= 
 q_{1,0,m}\left( (O_i)_{i \in I_{k}} \right),
\quad
\hat q_{0,0}^{(k)}
= 
 q_{0,0,m}\left( (O_i)_{i \in I_{k}} \right), 
\quad
\hat e_{0}^{(k)}
= 
 e_{0,m}\left( (O_i)_{i \in I_{k}} \right).
\]
Let \(\hat{Q}^{(k)}\) be the estimator of the marginal of \(X\) given by the fold-specific marginal empirical measure,
\[
\hat{Q}^{(k)}
 = \frac{1}{m}\sum_{i \in I_k} \delta_{X_i}.
\]

Given positive constants \(\mfc\) and \(c_\init\),  let \(\mft_2\in(0,\mfc^6/8)\) denote the constant furnished by \Cref{mainthm:det-bracket-one} for the pair \((\mfc,c_\init)\).
For each \(k\in\{0,1\}\), define the event
\[
\mathcal A_m^{(k)}
=
\left\{
(\hat Q^{(k)},\hat U_0^{(1-k)}) \text{ satisfies the local bracketing conditions with constants } (\mfc,c_\init)
\right\}.
\]
Since \(\hat Q^{(k)}\) and \(\hat U_0^{(1-k)}\) are \(\mathcal G_m^{(k)}\)-measurable, the event \(\mathcal A_m^{(k)}\) is also \(\mathcal G_m^{(k)}\)-measurable.

On \(\mathcal A_m^{(k)}\), let \(t\mapsto \hat U_t^{(k,1-k)}\) denote the solution of the ODE
\eqref{e:main_ode} with initial condition \(\hat U_0^{(1-k)}\) and fixed marginal law \(\hat Q^{(k)}\).
On \((\mathcal A_m^{(k)})^c\), set
\[
\hat U_t^{(k,1-k)} \equiv \hat U_0^{(1-k)}
\qquad\text{for all }t\in\mathbb R.
\]
Abbreviate
\[
\hat p_t^{(k,1-k)}(a,y\mid x)=p_{\hat U_t^{(k,1-k)}}(a,y\mid x),
\qquad (a,y)\in\{0,1\}^2.
\]

On \(\mathcal A_m^{(k)}\), define the fold-specific empirical loss function
\[
L_m^{(k)}(t)=\mathbb P_m^{(k)}\big[\log \hat p_t^{(k,1-k)}(A,Y\mid X)\big].
\]
Define the event
\[
\mathcal E_m^{(k)}
=
\mathcal A_m^{(k)}\cap
\left\{
\text{there exists a unique } t\in [-\mft_2,\mft_2]
\text{ such that }
\left(L_m^{(k)}\right)'(t)=0
\right\}.
\]
On the event \(\mathcal E_m^{(k)}\), define \(\hat t_k\) to be the unique \(t\in[-\mft_2,\mft_2]\) such that
\[
\left(L_m^{(k)}\right)'(t)
=
\mathbb P_m^{(k)}\!\left[
D^*\!\left(\cdot;\widetilde P\big(\hat Q^{(k)},\hat U_t^{(k,1-k)}\big)\right)
\right]
=0.
\]
On \(\big(\mathcal E_m^{(k)}\big)^c\), set \(\hat t_k=0\).
Set
\[
\hat U_k^\dagger=\hat U_{\hat t_k}^{(k,1-k)}.
\]
On $\mathcal A_m^{(k)}\cap \mathcal E_m^{(k)}$, define
\[
\hat\psi_k=\psi\big(\widetilde P(\hat Q^{(k)},\hat U_k^\dagger)\big),
\]
and for $i\in I_k$ define
\[
\hat D_{\mathrm{CF},i}
=
D_{\full}^*(O_i;\widetilde P(\hat Q^{(k)},\hat U_k^\dagger)).
\]
On $(\mathcal A_m^{(k)}\cap \mathcal E_m^{(k)})^c$, set
\[
\hat\psi_k=0,
\qquad
\hat D_{\mathrm{CF},i}=0
\quad\text{for all }i\in I_k.
\]
Finally, set
\[
\hat\psi_{\mathrm{CF}}=\frac{1}{2}\hat\psi_0+\frac{1}{2}\hat\psi_1.
\]

The following theorem gives our main asymptotic result for the cross-fitted one-step TMLE \(\hat\psi_{\mathrm{CF}}\), including semiparametric efficiency and Wald inference. For \(\gamma\in[0,1]\), let \(z_\gamma\) denote the \(\gamma\)-quantile of a standard normal random variable.

\begin{theorem}
\label{mainthm:AN_cross}
Assume that \(n\) is even and write \(n=2m\).
Fix a probability measure $P^*$ on $\mathcal O$,
$\lambda\colon [0,1] \to \mathbb{R}_{\ge 0}$,
nuisance estimator sequences $( q_{1,0,s})_{s=1}^\infty$, $( q_{0,0,s})_{s=1}^\infty$, and $( e_{0,s})_{s=1}^\infty$,
and constants $\mfc, c_\init>0$. 
Suppose that the following three conditions hold:
\begin{enumerate}
\item Assumption~\ref{a:pstar} holds.
\item  With probability tending to one, for each $k\in\{0,1\}$, the pair $(\hat Q^{(k)},\hat U_0^{(1-k)})$ satisfies the local bracketing conditions with constants $(\mfc,c_\init)$.
\item \Cref{as::initial_convergence} holds for $(q_{1,0,s})_{s=1}^\infty$, $(q_{0,0,s})_{s=1}^\infty$, and $(e_{0,s})_{s=1}^\infty$.
\end{enumerate}
Then
\[
\sqrt{n}\left(\hat\psi_{\mathrm{CF}}-\psi(P^*)\right)
=\frac{1}{\sqrt{n}}\sum_{i=1}^n D_{\full}^*(O_i;P^*)+o_p(1).
\]
Further, the plug-in variance estimator based on the fold-specific estimated EIF values,
\[
\hat\sigma_{\mathrm{CF}}^2
=
\frac{1}{n}\sum_{i=1}^n \hat D_{\mathrm{CF},i}^2,
\]
converges in probability to 
\[\sigma^2 =\mathrm{Var}_{P^*}\big(D_{\full}^*(O;P^*)\big).
\]

Additionally, $\sigma^2 >0$, 
\[
\sqrt{n}\left(\hat\psi_{\mathrm{CF}}-\psi(P^*)\right)\xrightarrow{d}N(0,\sigma^2),
\]
and for every $\alpha \in (0,1)$, the Wald interval
\[
\left[
\hat\psi_{\mathrm{CF}} -  z_{1-\alpha/2}\frac{\hat\sigma_{\mathrm{CF}}}{\sqrt{n}}, \hat\psi_{\mathrm{CF}} + z_{1-\alpha/2}\frac{\hat\sigma_{\mathrm{CF}}}{\sqrt{n}}\right]
\]
is an asymptotically valid level-$(1-\alpha)$ confidence interval.
\end{theorem}

\subsection{Spline Estimation}
To apply \Cref{mainthm:AN_cross}, we must verify two things. First, for each $k\in\{0,1\}$, the pair $(\hat Q^{(k)},\hat U_0^{(1-k)})$ must satisfy the local bracketing conditions with probability tending to one. Second, the nuisance estimators used to construct $\hat U_0^{(k)}$ must converge sufficiently quickly for \Cref{as::initial_convergence} to hold. These requirements are standard and can be verified for many classical nonparametric estimators under smoothness and positivity assumptions on the nuisance functions. As an example, we verify them for regression spline estimators. We assume some familiarity with the definition of B-splines; for details, see \cite[Section 6]{chen2015optimal}. Essential details are recalled in \Cref{a:spline}.

We consider splines constructed from an isotropic tensor-product B-spline basis on $[0,1]^d$. Given a spline degree $r \ge 1$ and an integer $J\ge r+1$, we consider a univariate degree-$r$ B-spline basis on $[0,1]$ with uniformly spaced interior knots. The associated tensor-product basis on $[0,1]^d$ has a total of $K=J^d$ elements. The corresponding spline estimators for the nuisance functions are then constructed by least squares regression onto these tensor-product basis functions. Finally, the regression estimates are  truncated so that they take values in the interval $[\eta^\circ, 1 - \eta^\circ]$, where $\eta^\circ \in (0, 1/2)$ is a deterministic parameter. Full details of this construction are given in \Cref{def:spline}.

To state the following assumption on the data-generating process, we use the notion of a H\"older ball, recalled in \Cref{d:holderball}.

\begin{assumption}\label{ass:spline}
Let $P^*$ denote the true law of $O=(X,A,Y)$, satisfying \Cref{a:pstar}, with nuisance functions
$U^*=(q_1^*,q_0^*,e^*)$. 

\begin{enumerate}
\item \textbf{Design.}
The support of $X$ is $[0,1]^d$, and the marginal law $P_X^*$ admits a density
$f_X$ with respect to Lebesgue measure. Further, there exist constants $C_X, c_X>0$ such that
\[
0<c_X\le f_X(x)\le C_X<\infty
\qquad\text{for all }x\in[0,1]^d.
\]

\item \textbf{H\"older Smoothness.}
There exist $\beta>d/2$ and $L>0$ such that each of $q_1^*$, $q_0^*$, and $e^*$ belong to
the H\"older ball $\mathcal{C}^\beta([0,1]^d,L)$.

\item \textbf{Positivity.}
There exists $\eta^*\in(0,1/4)$ such that
\[
e^*(x)\in[2\eta^*,1-2\eta^*]
\qquad\text{for all }x\in[0,1]^d.
\]
\item \textbf{Outcome-Regression Boundedness.}
There exists $\kappa^*\in(0,1/4)$ such that for each $a\in\{0,1\}$,
\[
q_a^*(x)\in[2\kappa^*,1-2\kappa^*]
\qquad\text{for all }x\in[0,1]^d.
\]
\end{enumerate}
\end{assumption}

\begin{remark}
In order for the truncated spline estimator in \Cref{def:spline} to be consistent for the nuisance functions, it suffices to choose the truncation parameter so that 
\[
\eta^\circ \le 2\min\{\eta^*,\kappa^*\}.
\]
Generically, the values of \(\eta^*\) and \(\kappa^*\) are unknown, so one must assume lower bounds on both in order to select \(\eta^\circ\). This is consistent with common TMLE implementations, which often impose a small positive truncation level on estimated nuisance functions for numerical stability \cite{cai2020}.
We also require some smoothness information about \(\beta\) to select \(r\) and \(J\): an upper bound on \(\beta\) suffices for choosing \(r\), while a lower bound \(\underline\beta>d/2\) suffices for a conservative choice of \(J\).
It would be possible to weaken the first of these assumptions by making the truncation level shrink to zero as the sample size increases. Similarly, the second can be loosened using adaptive estimation. For brevity, we do not take this up here.
\end{remark}

We now show that, under the previous assumption, spline estimators of the nuisance functions satisfy the hypotheses of \Cref{mainthm:AN_cross}. The proof is given in \Cref{a:spline}.

\begin{lemma}
\label{lem:spline}
Let $\{(X_i,A_i,Y_i)\}_{i=1}^{2m}$, with $X_i\in[0,1]^d$, $A_i\in\{0,1\}$, and $Y_i\in\{0,1\}$, be independent and identically distributed observations from a distribution $P^*$ satisfying \Cref{ass:spline} with constants $L, C_X, c_X>0$, $\beta>d/2$, and $\eta^*,\kappa^*\in(0,1/4)$.

Fix $\eta^\circ \in \bigl(0,2\min\{\eta^*,\kappa^*\}\bigr)$, $r > \max\{\beta,1\}$,
\[
J_m=\max\left\{r+1,\left\lfloor (m/\log(m\vee 3))^{1/(2\beta+d)} \right\rfloor\right\},
\]
and construct $U_0=(\tilde q_1,\tilde q_0,\tilde e)$ using the tensor-product construction of Definition~\ref{def:spline}
with $J=J_m$ applied to the observations $\{(X_i,A_i,Y_i)\}_{i=1}^{m}$.
Let
\[
Q
=
\frac1m\sum_{i=m+1}^{2m}\delta_{X_i}
\]
denote the empirical distribution of $X$ computed from the observations $\{(X_i,A_i,Y_i)\}_{i=m+1}^{2m}$.
Let $\lambda\colon [0,1] \to \mathbb{R}_{\ge 0}$ be a bounded function that is $C^3$ on $[\eta^\circ/4,1-\eta^\circ/4]$ satisfying
\[
\Lambda_{\min} >0, \qquad
\Lambda_{\max} + \dot\Lambda_{\max} + \Lambda_{2,\max} + \Lambda_{3,\max}< \infty,
\]
where these quantities are defined relative to the domain $[\eta^\circ/4,1-\eta^\circ/4]$.
Set $\eta=\eta^\circ/4$, and let $\mfc$ be the corresponding constant from \Cref{s:assumption_notation}.

Then, under Assumption~\ref{ass:spline},
\[
\|\tilde q_1-q_1^*\|_{2,*}=o_p(m^{-1/4}),\qquad
\|\tilde q_0-q_0^*\|_{2,*}=o_p(m^{-1/4}),\qquad
\|\tilde e-e^*\|_{2,*}=o_p(m^{-1/4}).
\]
Moreover, there exists a constant $c_\init>0$ such that, with probability tending to one as $m\to\infty$, the pair $(Q,U_0)$ satisfies the local bracketing conditions with constants $(\mfc,c_\init)$.
\end{lemma}

Applying \Cref{lem:spline} first with the first half-sample used to construct $U_0$ and the second half-sample used to define $Q$, and then again with the roles of the two halves reversed, yields the desired conclusion for the two-fold construction. In the notation of \Cref{s:results}, with probability tending to one, for each $k\in\{0,1\}$, the pair $(\hat Q^{(k)},\hat U_0^{(1-k)})$ satisfies the local bracketing conditions with constants $(\mfc,c_\init)$. Moreover, \Cref{as::initial_convergence} holds for the corresponding spline nuisance estimator sequences. Hence the hypotheses of \Cref{mainthm:AN_cross} are satisfied for the two-fold spline construction above, and the resulting cross-fitted one-step TMLE is asymptotically linear, efficient, and asymptotically normal under Assumption \ref{ass:spline}.

\appendix 
\section{Background on TMLE}

\subsection{Semiparametric Preliminaries}\label{s:influence_appendix}

We briefly review regular parametric submodels, tangent spaces, pathwise differentiability, asymptotic linearity, and regularity.

Let \(P\in\mathcal M_{\full}\). We write \(L_0^2(P)\) for the Hilbert space of square–integrable mean-zero functions under \(P\), equipped with inner product
\[
\langle f,g\rangle_P = \E_P[fg].
\]

A one-dimensional regular parametric submodel through \(P\) is a family
\[
\{P_\epsilon:\epsilon\in(-\ell,\ell)\}\subset\mathcal M_{\full}
\]
such that \(P_0=P\) and, for some common dominating measure \(\nu\), the densities
\[
p_\epsilon=\frac{dP_\epsilon}{d\nu}
\]
are differentiable in quadratic mean at \(\epsilon=0\). The corresponding score is the element
\[
h
=
\left.\frac{d}{d\epsilon}\log p_\epsilon\right|_{\epsilon=0}
\in L_0^2(P).
\]

The tangent space \(T_{\full}(P)\) of the full model at \(P\) is defined as the closure in \(L_0^2(P)\) of the set of scores of all one-dimensional regular parametric submodels through \(P\).

Let \(\psi:\mathcal M_{\full}\to\mathbb R\) be a target parameter and write
\[
\psi^*=\psi(P^*).
\]
Let \(\psi_n\) be an estimator of \(\psi^*\) based on \(n\) i.i.d.\ observations. We now formalize pathwise differentiability, asymptotic linearity, and regularity.

\begin{definition}[Pathwise Differentiability]\label{def::PD}
A parameter \(\psi:\mathcal M_{\full}\to\mathbb R\) is said to be pathwise differentiable at \(P\in\mathcal M_{\full}\) if there exists a function
\[
D^*(\cdot;P)\in \overline{T_{\full}(P)}\subset L_0^2(P)
\]
such that, for every one-dimensional regular parametric submodel
\[
\{P_\epsilon:\epsilon\in(-\ell,\ell)\}\subset\mathcal M_{\full}
\]
through \(P\) with score \(h\in T_{\full}(P)\), the map \(\epsilon\mapsto\psi(P_\epsilon)\) is differentiable at \(\epsilon=0\) and
\[
\left.\frac{d}{d\epsilon}\psi(P_\epsilon)\right|_{\epsilon=0}
=
E_P\!\big[D^*(O;P)\,h(O)\big].
\]
The function \(D^*(\cdot;P)\) is called the efficient influence function (EIF), or canonical gradient, of \(\psi\) at \(P\) relative to \(\mathcal M_{\full}\).
\end{definition}

\begin{definition}[Asymptotic Linearity]\label{def::AL}
An estimator \(\psi_n\) of \(\psi^*=\psi(P^*)\) is said to be asymptotically linear at \(P^*\) if there exists a function
\[
D(\cdot;P^*)\in L_0^2(P^*)
\]
such that
\[
\sqrt{n}(\psi_n-\psi^*)
=
\frac{1}{\sqrt{n}}\sum_{i=1}^n D(O_i;P^*) + o_{P^{*\otimes n}}(1).
\]
The function \(D(\cdot;P^*)\) is called the influence function of the estimator at \(P^*\).
\end{definition}

\begin{definition}[Regularity]\label{def::reg}
An estimator \(\psi_n\) of \(\psi^*=\psi(P^*)\) is said to be regular at \(P^*\in\mathcal M_{\full}\) if, for every one-dimensional regular parametric submodel
\[
\{P_\epsilon:\epsilon\in(-\ell,\ell)\}\subset\mathcal M_{\full}
\]
with \(P_0=P^*\), and for every fixed \(t\in\mathbb R\), the sequence
\[
\sqrt{n}\bigl(\psi_n-\psi(P_{\epsilon_n})\bigr),
\qquad
\epsilon_n=t/\sqrt{n},
\]
converges in distribution under \(P_{\epsilon_n}^{\otimes n}\) to a limit that does not depend on \(t\) or on the particular choice of submodel.
\end{definition}

If \(\psi\) is pathwise differentiable at \(P^*\), then the EIF \(D^*(\cdot;P^*)\) determines the semiparametric efficiency bound. In particular, among regular asymptotically linear estimators, efficiency at \(P^*\) is equivalent to having influence function equal to \(D^*(\cdot;P^*)\), or equivalently asymptotic variance
\[
\E_{P^*}\!\big[D^*(O;P^*)^2\big].
\]
See \cite{bickel1998} for details.

\subsection{Tangent Space Identifications}\label{s:tangent}

We now specialize the preceding definitions to the statistical models considered in the paper.

Under the factorization \eqref{e:factorization}, the tangent space of the full model admits the orthogonal decomposition
\[
T_{\full}(P)=T_X(P)\oplus T_A(P)\oplus T_Y(P),
\]
where
\[
T_X(P)=\{h(X):\E_P[h(X)]=0\},
\]
\[
T_A(P)=\{h(A,X):\E_P[h(A,X)\mid X]=0\},
\]
and
\[
T_Y(P)=\{h(Y,A,X):\E_P[h(Y,A,X)\mid A,X]=0\}.
\]

When \(Q\) is absolutely continuous with respect to Lebesgue measure and
\[
P=\widetilde P(Q,U)\in\mathcal M(Q),
\]
the tangent space of the restricted model \(\mathcal M(Q)\) at \(P\) is
\[
T(P)
=
\{h\in T_{\full}(P): h\perp T_X(P)\}
=
T_A(P)\oplus T_Y(P).
\]
Thus, in the restricted model, first-order perturbations of the marginal law of \(X\) are excluded, and only the conditional law of \((A,Y)\mid X\) is allowed to vary.

\subsection{Derivation of the WATE Dynamics}
\begin{proof}[Proof of \Cref{lm::ODE_flow_derivation}]
For each $t\in I$, the assumption $U_t\in C(\mathcal X;(0,1))^3$ implies that for all $(a,y)\in\{0,1\}^2$,
\[
p_t(a,y\mid x)
=
e_t(x)^a\big(1-e_t(x)\big)^{1-a}\,q_{a,t}(x)^y\big(1-q_{a,t}(x)\big)^{1-y}
\]
is a well-defined and strictly positive conditional probability mass function of
$(A,Y)$ given $X=x$. Hence $P_t=\widetilde P(Q,U_t)$, the logarithm
$\log p_t(a,y\mid x)$, and the shorthand
$D^*(\cdot;U_t)=D^*(\cdot;\widetilde P(Q,U_t))$ are all well defined.
Since $\Omega_t>0$ for all $t\in I$, the differential equations in
\eqref{eq:ATT-ODE} and the expression for $D^*(\cdot;U_t)$ are also well defined.

Fix $t\in I$, $x\in\mathcal X$, and $(a,y)\in\{0,1\}^2$. By definition, 
\[
\log p_t(a,y\mid x)
=
a\log e_t(x)+(1-a)\log\big(1-e_t(x)\big)
+y\log q_{a,t}(x)+(1-y)\log\big(1-q_{a,t}(x)\big),
\]
where $q_{a,t}(x)=q_{1,t}(x)$ if $a=1$ and $q_{a,t}(x)=q_{0,t}(x)$ if $a=0$.
Since $t\mapsto (q_{1,t}(x),q_{0,t}(x),e_t(x))$ is $C^1$, differentiating with
respect to $t$ yields
\begin{align}
\frac{d}{dt}\log p_t(a,y\mid x)
&= \frac{\frac{d}{dt}e_t(x)}{e_t(x)\big(1-e_t(x)\big)}\bigl(a-e_t(x)\bigr)\notag\\
&\qquad
+\mathbf 1(a=1)\,
\frac{\frac{d}{dt}q_{1,t}(x)}{q_{1,t}(x)\big(1-q_{1,t}(x)\big)}
\bigl(y-q_{1,t}(x)\bigr)\notag\\
&\qquad
+\mathbf 1(a=0)\,
\frac{\frac{d}{dt}q_{0,t}(x)}{q_{0,t}(x)\big(1-q_{0,t}(x)\big)}
\bigl(y-q_{0,t}(x)\bigr).
\label{eq:appendix-score-general}
\end{align}
By definition,
\[
\phi_t(o)-\tau_t(x)
=
\frac{a}{e_t(x)}\big(y-q_{1,t}(x)\big)
-\frac{1-a}{1-e_t(x)}\big(y-q_{0,t}(x)\big).
\]
Hence, using the formula for the restricted-model EIF,
\begin{align}
D^*(o;U_t)
&=
\frac{\lambda(e_t(x))}{\Omega_t}
\left[
\frac{a}{e_t(x)}\big(y-q_{1,t}(x)\big)
-\frac{1-a}{1-e_t(x)}\big(y-q_{0,t}(x)\big)
\right]\notag\\
&\qquad
+\frac{\dot\lambda(e_t(x))}{\Omega_t}
\bigl(\tau_t(x)-\psi_t\bigr)\bigl(a-e_t(x)\bigr).
\label{eq:appendix-eif-expanded}
\end{align}
Now suppose that $U_t$ solves \eqref{eq:ATT-ODE}. Substituting the three
differential equations into \eqref{eq:appendix-score-general} gives
\begin{align*}
\frac{d}{dt}\log p_t(a,y\mid x)
&=
\frac{1}{e_t(x)\big(1-e_t(x)\big)}
\left[
\frac{e_t(x)\big(1-e_t(x)\big)\dot\lambda(e_t(x))}{\Omega_t}
\bigl(\tau_t(x)-\psi_t\bigr)
\right]
\bigl(a-e_t(x)\bigr)\\
&\qquad
+\mathbf 1(a=1)\,
\frac{1}{q_{1,t}(x)\big(1-q_{1,t}(x)\big)}
\left[
\frac{q_{1,t}(x)\big(1-q_{1,t}(x)\big)}{\Omega_t}
\frac{\lambda(e_t(x))}{e_t(x)}
\right]
\bigl(y-q_{1,t}(x)\bigr)\\
&\qquad
+\mathbf 1(a=0)\,
\frac{1}{q_{0,t}(x)\big(1-q_{0,t}(x)\big)}
\left[
-\frac{q_{0,t}(x)\big(1-q_{0,t}(x)\big)}{\Omega_t}
\frac{\lambda(e_t(x))}{1-e_t(x)}
\right]
\bigl(y-q_{0,t}(x)\bigr).
\end{align*}
After cancellation,
\begin{align*}
\frac{d}{dt}\log p_t(a,y\mid x)
&=
\frac{\dot\lambda(e_t(x))}{\Omega_t}
\bigl(\tau_t(x)-\psi_t\bigr)\bigl(a-e_t(x)\bigr)\\
&\qquad
+\frac{\lambda(e_t(x))}{\Omega_t}\frac{a}{e_t(x)}
\bigl(y-q_{1,t}(x)\bigr)
-\frac{\lambda(e_t(x))}{\Omega_t}\frac{1-a}{1-e_t(x)}
\bigl(y-q_{0,t}(x)\bigr),
\end{align*}
which is exactly \eqref{eq:appendix-eif-expanded}. Therefore,
\[
\frac{d}{dt}\log p_t(a,y\mid x)=D^*(o;U_t),
\qquad o=(x,a,y).
\]
Since $t\in I$, $x\in\mathcal X$, and $a,y\in\{0,1\}$ were arbitrary, this
proves the first identity in \eqref{eq:path-score}.

It remains to verify the mean-zero properties. Fix $t\in I$. Under $P_t$, we have
\[
\E_t[A-e_t(X)\mid X]=0.
\]
Moreover, for each $x\in\mathcal X$,
\begin{align*}
\E_t\!\left[\frac{A}{e_t(X)}\big(Y-q_{1,t}(X)\big)\Bigm|X=x\right]
&=
\frac{1}{e_t(x)}
\E_t\!\left[A\big(Y-q_{1,t}(x)\big)\mid X=x\right]\\
&=
\frac{P_t(A=1\mid X=x)}{e_t(x)}
\E_t\!\left[Y-q_{1,t}(x)\mid A=1,X=x\right]\\
&=0,
\end{align*}
because $P_t(A=1\mid X=x)=e_t(x)$ and
\[
\E_t[Y\mid A=1,X=x]=q_{1,t}(x).
\]
Similarly, for each $x\in\mathcal X$,
\begin{align*}
\E_t\!\left[\frac{1-A}{1-e_t(X)}\big(Y-q_{0,t}(X)\big)\Bigm|X=x\right]
&=
\frac{1}{1-e_t(x)}
\E_t\!\left[(1-A)\big(Y-q_{0,t}(x)\big)\mid X=x\right]\\
&=
\frac{P_t(A=0\mid X=x)}{1-e_t(x)}
\E_t\!\left[Y-q_{0,t}(x)\mid A=0,X=x\right]\\
&=0,
\end{align*}
because $P_t(A=0\mid X=x)=1-e_t(x)$ and
\[
\E_t[Y\mid A=0,X=x]=q_{0,t}(x).
\]
Therefore,
\[
\E_t[\phi_t(O)-\tau_t(X)\mid X]=0.
\]
Combining this with \eqref{eq:appendix-eif-expanded} yields
\[
\E_t[D^*(O;U_t)\mid X]=0.
\]
Finally, by iterated expectation,
\[
\E_t[D^*(O;U_t)]
=
\E_t\big[\E_t[D^*(O;U_t)\mid X]\big]
=0.
\]
This proves the remaining identities in \eqref{eq:path-score}.
\end{proof}

\section{Proof of \texorpdfstring{\Cref{mainthm:ODE-C1}}{First Main Result}}
\label{sec:result1}

Throughout this section, we work entirely in the deterministic setting of \Cref{mainthm:ODE-C1}. In particular, the probability measure $Q$, the initial nuisance triple $U_0\in C(\mathcal X;[0,1])^3$, and the weight function $\lambda$ are all fixed, and no conditioning or empirical-process argument enters anywhere in the proof. The goal is to solve the ODE
\[
\frac{d}{dt}U_t=F(Q,U_t), \qquad U_{t=0}=U_0,
\]
locally in the Banach space $C(\mathcal X)^3$.

The proof has two steps. First, we show that the vector field $U\mapsto F(Q,U)$ is uniformly bounded and Lipschitz on the local ball $\mathcal B_\eta(U_0)$. These estimates imply that the integral operator
\[
(\mathcal T\mathsf U)(t)=U_0+\int_0^t F(Q,\mathsf U(s))\,ds
\]
maps a suitable closed ball in the path space $C([-\mft_1,\mft_1];C(\mathcal X)^3)$ into itself and is a contraction there. Banach's fixed-point theorem therefore yields a unique continuous fixed point.

Second, we upgrade this fixed point to a $C^1$ solution of the ODE by applying the Banach-space-valued fundamental theorem of calculus to the integral equation. Uniqueness of the fixed point gives uniqueness of the ODE solution among paths that remain in $\mathcal B_\eta(U_0)$, and the pointwise ODE follows by composing with the evaluation maps $\mathrm{ev}_x\colon C(\mathcal X)^3\to\mathbb R^3$. Thus the theorem reduces to a local contraction argument for a deterministic flow, with all constants controlled by the positivity neighborhood from \Cref{as::positivity}.

We first recall the Banach Contraction Theorem; see, for example \cite[Corollary 1.3]{pata2019fixed}. Recall that a function is called a contraction if it is Lipschitz continuous with Lipschitz constant strictly less than one.

\begin{theorem}\label{t:fixedpoint}
Let $S$ be a nonempty complete metric space, and let $f\colon S\to S$ be a contraction. Then $f$ has a unique fixed point.
\end{theorem}

We now derive various estimates that allow us to apply \Cref{t:fixedpoint} to prove \Cref{mainthm:ODE-C1}.

\begin{lemma}
We have 
\begin{equation}\label{eq:MF}
\sup_{U\in\mathcal B_\eta(U_0)}\|F(U)\|_\infty \le \frac{1}{2\mfc^4}.
\end{equation}
\end{lemma}
\begin{proof}
Under \Cref{as::positivity}, we have 
\begin{align*}
\|F_1(U)\|_{\infty} &= \left \|\frac{q_{1}(x)\big(1-q_{1}(x)\big)\lambda\big(e(x)\big)}{\Omega e(x)}\right \|_{\infty}\\
&\leq \frac{\Lambda_{\max}}{\Lambda_{\min}\eta}\|q_1(x)(1-q_1(x))\|_{\infty}\\
&\leq \frac{\Lambda_{\max}}{4\Lambda_{\min}\eta} \le \frac{1}{4 \mfc^3}.
\end{align*}
The first inequality uses that $\Omega(U)=\E_Q[\lambda(e(X))]\ge \Lambda_{\min}$ and that $e(x)\ge \eta$ for all $x\in\mathcal X$ under \Cref{as::positivity}. The second inequality uses that $q_1(x)(1-q_1(x))\le 1/4$ for all $x$.
Similarly, we have
\[
\|F_0(U)\|_{\infty}\le \frac{1}{4 \mfc^3}.
\]
For $F_e(U)$, we first note that by the definition of $\psi$,
\[
|\psi(U)| \le 1.\]
Then, we have 
\begin{align*}
\|F_e(U)\|_{\infty} &= \left\|\frac{e(x)(1-e(x))\dot\lambda(e(x))}{\Omega}\big (\tau(x)-\psi \big)\right \|_{\infty}\\
&\le \frac{\dot\Lambda_{\max}}{4\Lambda_{\min}}(1+|\psi|) \le \frac{1}{2 \mfc^4}
\end{align*}
Therefore,
\[
\sup_{U\in\mathcal B_\eta(U_0)}\|F(U)\|_\infty \le \max\left\{\frac{1}{4\mfc^3},\frac{1}{2\mfc^4}\right\}=\frac{1}{2\mfc^4}
\]
\end{proof}

\begin{lemma}
\label{l:Flips}
For all $U,V\in\mathcal B_\eta(U_0)$, 
\[ \|F(U)-F(V)\|_\infty \le 4\mfc^{-6}\|U-V\|_\infty.\]
\end{lemma}

\begin{proof}
Denote 
\[ U=(q_1,q_0,e), \quad V=(q_1',q_0',e'), \quad \Delta=\|U-V\|_\infty,\]
and 
\[
\tau=q_1-q_0, \quad \tau'=q_1'-q_0', \quad \Omega=\Omega(U), \quad \Omega'=\Omega(V), \quad 
\psi=\psi(U), \quad \psi'=\psi(V).
\]

We consider each coordinate of $F$ separately. 
For $F_1$, we write $F_1(U)=f(q_1)G(e,\Omega)$, where
\[ f(q)=q(1-q), \qquad G(e,\Omega) =\frac{\lambda(e)}{\Omega e }.\]
Then
\begin{align}\label{e:f1lipschitz}
\begin{split}
 \|F_1(U)-F_1(V)  \|_\infty \le&  \|f(q_1)-f(q_1') \|_\infty \|G(e,\Omega) \|_\infty \\ &+ \|f(q_1')\|_\infty  \|G(e,\Omega)-G(e',\Omega') \|_\infty.
\end{split}
\end{align}
Since $|f'(q)|=|1-2q|\le 1$, we have  that 
\[  \|f (q_1)  -f(q'_1)  \|_\infty \le \|q_1  -q_1' \|_\infty \le \Delta, \qquad  \|G(e,\Omega)  \|_\infty \le\frac{\Lambda_{\max}}{\Lambda_{\min}\eta},\] 
which implies
\begin{equation}\label{e:l1term}
\big \|f(q_1)-f(q_1')\big \|_{\infty} \big \|G(e,\Omega) \big\|_{\infty}
\le \frac{\Lambda_{\max}}{\Lambda_{\min}\eta} \Delta.
\end{equation}
To bound the second term in \eqref{e:f1lipschitz}, we use that 
\[
\left|\frac{\partial G}{\partial e}\right|
=\frac{|e\dot\lambda(e)-\lambda(e)|}{\Omega e^2}
\le \frac{\dot\Lambda_{\max}+\Lambda_{\max}}{\Lambda_{\min}\eta^2},\qquad
\left|\frac{\partial G}{\partial \Omega}\right|
=\frac{|\lambda(e)|}{\Omega^2 e}\le \frac{\Lambda_{\max}}{\Lambda_{\min}^2\eta}.
\]
These inequalities imply, together with the mean value theorem, that
\[
\|G(e,\Omega)-G(e',\Omega') \|_\infty \le
\frac{\dot\Lambda_{\max}+\Lambda_{\max}}{\Lambda_{\min} \eta^2}\|e-e'\|_\infty
+\frac{\Lambda_{\max}}{\Lambda_{\min}^2\eta} |\Omega-\Omega'|.
\]
Then using $\|f(q_1')\|_\infty\le 1/4$,  $\|e-e'\|_\infty \le \Delta$ and 
\[
|\Omega-\Omega'|\le \E_Q\Big[ \big|\lambda\bigl(e(X)\bigr)-\lambda\bigl(e'(X)\bigr)\big|\Big]\le \dot\Lambda_{\max}\Delta,\]
where the second inequality follows from the mean value theorem, we find 
\begin{equation}\label{e:l2term}
\|f(q_1')\|_\infty  \|G(e,\Omega)-G(e',\Omega') \|_\infty
 \le 
\frac{1}{4}\left(\frac{\dot\Lambda_{\max}+\Lambda_{\max}}{\Lambda_{\min}\eta^2}
+\frac{\Lambda_{\max}\dot\Lambda_{\max}}{\Lambda_{\min}^2\eta}\right)\Delta.
\end{equation}
Then \eqref{e:l1term} and \eqref{e:l2term} imply 
\begin{align*}
    \|F_1(U)-F_1(V)\|_\infty &\le\left(\frac{\Lambda_{\max}}{\Lambda_{\min}\eta}
 +\frac{1}{4}\left(\frac{\dot\Lambda_{\max}+\Lambda_{\max}}{\Lambda_{\min}\eta^2}
 +\frac{\Lambda_{\max}\dot\Lambda_{\max}}{\Lambda_{\min}^2\eta}\right)\right)\Delta\\
 &\le \left(\frac{1}{\mfc^3}+\frac{1}{4}\left(\frac{2}{\mfc^4}+\frac{1}{\mfc^5}\right)\right)\Delta\\
 &\le 2\mfc^{-5}\Delta.
\end{align*}

For $F_0$, the same argument with 
\[ G_0(e,\Omega) =\frac{\lambda(e)}{\Omega(1-e)}\]
in place of $G$, together with the bound $1-e(x)\ge \eta$ for all $x \in \mathcal X$ (which follows from \Cref{as::positivity}), yields
\[ \|F_0(U)-F_0(V)\|_\infty\le 2\mfc^{-5} \Delta.\]
Lastly, for $F_e$, we write $F_e(U)=H(e,\Omega) (\tau-\psi)$ with 
\[
H(e,\Omega)=\frac{e(1-e)\dot\lambda(e)}{\Omega}.
\]
Then
\begin{equation}\label{e:Fedecomp}
\|F_e(U)-F_e(V)\|_\infty \le \|H-H'\|_{\infty}\|\tau-\psi\|_{\infty} + \|H'\|_{\infty}\big(\|\tau-\tau'\|_{\infty}+|\psi-\psi'|\big),
\end{equation}
where we abbreviate $H'=H(e',\Omega')$.
Using \Cref{as::positivity}, we have
\begin{align}
\begin{split}\label{e:Feinitialest}
    &\|\tau-\tau'\|_{\infty}\le \|q_1-q_1'\|_{\infty}+\|q_0-q_0'\|_{\infty}\le 2\Delta,\\
    &\|\tau-\psi\|_{\infty}\le \|\tau\|_{\infty}+|\psi|
\le 1+\frac{\E_Q\big[|\lambda(e(X))\tau(X)|\big]}{\Omega}
\le 1+\frac{\Lambda_{\max}}{\Lambda_{\min}}
\le \frac{3}{2\mfc^2},\\
    &\|H'\|_{\infty}\le \frac{\big \|e'(1-e') |\dot\lambda(e')| \big \|_{\infty}}{\Omega'} \le \frac{\dot\Lambda_{\max}}{4\Lambda_{\min}}\leq\frac{1}{4\mfc^2}.
\end{split}
\end{align}
Let 
\[g(u)=u(1-u)\dot\lambda(u), \qquad
g'(u)=(1-2u)\dot\lambda(u)+u(1-u)\lambda''(u).\]
We have
$$\sup_{u\in[\eta,1-\eta]}|g'(u)|\le \dot\Lambda_{\max}+\frac{\Lambda_{2,\max}}{4}.$$
By \Cref{as::positivity},
$$\|g(e)-g(e')\|_{\infty}\le \left(\dot\Lambda_{\max}+\frac{\Lambda_{2,\max}}{4}\right) \|e-e'\|_{\infty}\le \left(\dot\Lambda_{\max}+\frac{\Lambda_{2,\max}}{4}\right) \Delta.$$
Moreover
\begin{equation}\label{e:omega_reciprocal}
\left|\frac{1}{\Omega}-\frac{1}{\Omega'}\right| \le \frac{|\Omega-\Omega'|}{\Omega\Omega'}
\le \frac{\dot\Lambda_{\max}}{\Lambda_{\min}^2}\Delta,
\qquad
\max\bigl\{\|g(e)\|_{\infty},\|g(e')\|_{\infty}\bigr\}\le \frac{\dot\Lambda_{\max}}{4}.
\end{equation}
Consequently,
\begin{align}
\begin{split}\label{e:qqprime}
  \|H-H'\|_{\infty}
  &=\left\|\frac{g(e)}{\Omega}-\frac{g(e')}{\Omega'}\right\|_{\infty}  \\
  &\le \frac{\|g(e)-g(e')\|_{\infty}}{\Omega}
  + \left|\frac{1}{\Omega}-\frac{1}{\Omega'}\right|\|g(e')\|_{\infty}\\
  &\le \left(\frac{4\dot\Lambda_{\max}+\Lambda_{2,\max}}{4\Lambda_{\min}}+\frac{\dot\Lambda_{\max}^2}{4\Lambda_{\min}^2}\right)\Delta
  \le \frac{3}{2\mfc^4}\Delta.
\end{split}
\end{align}

Denote 
\[
N(U)=\E_Q\Big[\lambda\bigl(e(X)\bigr)\tau(X)\Big],
\]
so that $\psi=N/\Omega$.
Then
\begin{equation}\label{e:t1psidiff}
|\psi-\psi'|
\le \left|\frac{1}{\Omega}-\frac{1}{\Omega'}\right||N|
+\frac{1}{\Omega'}|N-N'|.
\end{equation}
We bound $|N|\le \Lambda_{\max}$ and
\begin{align*}
  |N-N'|&=\left|\E_Q\big[\lambda\bigl(e(X)\bigr)\tau(X)-\lambda\bigl(e'(X)\bigr)\tau'(X)\big]\right|  \\
  &\le \E_Q\big[ \big|\lambda\bigl(e(X)\bigr)-\lambda\bigl(e'(X)\bigr)\big|\big]
   + \Lambda_{\max}\E_Q\big[ |\tau(X)-\tau'(X)|\big]\\
  &\le \dot\Lambda_{\max}\Delta + 2\Lambda_{\max}\Delta.
\end{align*}
Inserting this and  first bound of \eqref{e:omega_reciprocal} in \eqref{e:t1psidiff} yields
\[
|\psi-\psi'|\le
\left(\frac{\Lambda_{\max}\dot\Lambda_{\max}}{\Lambda_{\min}^2}
+\frac{\dot\Lambda_{\max}+2\Lambda_{\max}}{\Lambda_{\min}}\right)\Delta\le \frac{4}{\mfc^4}\Delta.
\]
Putting this estimate, \eqref{e:Feinitialest}, and \eqref{e:qqprime} in \eqref{e:Fedecomp}, we find 
\begin{equation*}
    \|F_e(U)-F_e(V)\|_{\infty}\le 4\mfc^{-6}\Delta.
\end{equation*}
Therefore, \[ \|F(U)-F(V)\|_\infty\le \max\{2\mfc^{-5},4\mfc^{-6}\}\Delta=4\mfc^{-6}\Delta,\]
concluding the proof. 
\end{proof}

Recall that $I_1$ was defined in \Cref{mainthm:ODE-C1}. 
Define 
\[
\mathsf X_{I_1}=\Big\{ \mathsf U \in C\big ([-\mft_1,\mft_1];C(\mathcal{X})^3\big): \| \mathsf U -U_0\|_{\sup}\le \eta\Big\},
\]
where $U_0$ represents the constant function on $[ - \mft_1, \mft_1]$ taking the value $U_0$, and we  define the norm
\[
\|\mathsf U\|_{\sup}=\sup_{|t|\le \mft_1}\|U_t\|_\infty,
\]
where we use the shorthand $U_t = \mathsf U(t)$. 
The induced metric is then
\[
d(\mathsf U, \mathsf V)=\| \mathsf U- \mathsf V\|_{\sup}.
\]

We now apply the Banach Contraction Theorem on $\mathsf X_{I_1}$.
\begin{theorem}
\label{thm:ODE}
Define the operator
\[
(\mathcal T \mathsf U)(t)
=
U_0+\int_0^t F(\mathsf U(s))\,ds,
\qquad t\in [-\mft_1,\mft_1],\ \mathsf U\in \mathsf X_{I_1}.
\]
Then $\mathcal T\colon \mathsf X_{I_1}\to \mathsf X_{I_1}$ is a contraction with Lipschitz constant at most $1/2$, and there exists a unique $\mathsf U\in\mathsf X_{I_1}$ such that
$\mathsf U=\mathcal{T} \mathsf U$.
\end{theorem}

\begin{proof}
Because $C(\mathcal X)^3$ is a Banach space under $\|\cdot\|_\infty$, the space
\[
C\big([-\mft_1,\mft_1];C(\mathcal X)^3\big)
\]
is a Banach space under $\|\cdot\|_{\sup}$. Since $\mathsf X_{I_1}$ is a closed ball in this Banach space, $(\mathsf X_{I_1},d)$ is complete.

Fix $\mathsf U\in\mathsf X_{I_1}$, and write $U_s=\mathsf U(s)$. Since $F$ is Lipschitz on $\mathcal B_\eta(U_0)$ by \Cref{l:Flips}, it is continuous there. Hence $s\mapsto F(U_s)$ is continuous, so $t\mapsto (\mathcal T\mathsf U)(t)$ is continuous.

Moreover, for every $t\in[-\mft_1,\mft_1]$,
\[
\|(\mathcal T\mathsf U)(t)-U_0\|_\infty
\le \int_{\min\{0,t\}}^{\max\{0,t\}} \|F(U_s)\|_\infty\,ds
\le \mft_1 \sup_{|s|\le \mft_1}\|F(U_s)\|_\infty
\le \frac{\mft_1}{2\mfc^4}
= \frac{\mfc^2}{16}.
\]
Since $\mfc\le \eta$ and $\eta<1/4$, we have $\mfc\le 1$, hence
\[
\frac{\mfc^2}{16}\le \frac{\mfc}{16}\le \frac{\eta}{16}<\eta.
\]
Therefore
\[
\|(\mathcal T\mathsf U)(t)-U_0\|_\infty \le \eta
\qquad\text{for all }t\in[-\mft_1,\mft_1],
\]
so $\mathcal T\mathsf U\in\mathsf X_{I_1}$.

Now let $\mathsf U,\mathsf V\in\mathsf X_{I_1}$, and write $U_s=\mathsf U(s)$ and $V_s=\mathsf V(s)$. Then for every $t\in[-\mft_1,\mft_1]$,
\begin{align*}
\|(\mathcal T\mathsf U)(t)-(\mathcal T\mathsf V)(t)\|_\infty
& \le \int_{\min\{0,t\}}^{\max\{0,t\}} \|F(U_s)-F(V_s)\|_\infty\,ds \\
&\le 4\mft_1\mfc^{-6}\|\mathsf U-\mathsf V\|_{\sup}
=\frac12 \|\mathsf U-\mathsf V\|_{\sup},
\end{align*}
by \Cref{l:Flips} and $\mft_1=\mfc^6/8$. Taking the supremum over $t$ shows that $\mathcal T$ is a contraction with Lipschitz constant at most $1/2$. The existence and uniqueness of the fixed point now follow from \Cref{t:fixedpoint}.
\end{proof}

We are now ready for the proof of \Cref{mainthm:ODE-C1}. 
\begin{proof}[Proof of \Cref{mainthm:ODE-C1}]
Let $\mathsf U\in\mathsf X_{I_1}$ be the unique fixed point from \Cref{thm:ODE}, and write $U_t=\mathsf U(t)$. Then
\[
U_t = U_0 + \int_{0}^{t} F(U_s)\,ds,
\qquad t\in[-\mft_1,\mft_1].
\]
Since $t\mapsto U_t$ is continuous and $F$ is Lipschitz on $\mathcal B_\eta(U_0)$ by \Cref{l:Flips}, the map
\[
f:[-\mft_1,\mft_1]\to C(\mathcal X)^3,
\qquad
f(t)=F(U_t),
\]
is continuous. Define
\[
\Phi(t)=\int_0^t f(s)\,ds,
\qquad t\in[-\mft_1,\mft_1].
\]
By the Banach-space-valued fundamental theorem of calculus, $\Phi\in C^1(I_1;C(\mathcal X)^3)$ and
\[
\Phi'(t)=f(t)=F(U_t)
\qquad\text{for all }t\in I_1.
\]
Here $\Phi'(t)=f(t)$ means that the derivative is taken in the Banach space $C(\mathcal X)^3$, i.e.,
\[
\lim_{h\to 0}\left\|\frac{\Phi(t+h)-\Phi(t)}{h}-f(t)\right\|_\infty=0.
\] 
Therefore $t\mapsto U_t=U_0+\Phi(t)$ belongs to $C^1(I_1;C(\mathcal X)^3)$ and satisfies
\[
\frac{d}{dt}U_t = F(U_t)=F(Q,U_t),
\qquad t\in I_1.
\]
Because $\mathsf U\in\mathsf X_{I_1}$, we have $U_t\in \mathcal B_\eta(U_0)$ for all $t\in I_1$.

To prove uniqueness, let $V\in C^1(I_1;C(\mathcal X)^3)$ be any other solution such that $V_t\in\mathcal B_\eta(U_0)$ for all $t\in I_1$ and $V_0=U_0$. Since
\[
\|V'(t)\|_\infty = \|F(V_t)\|_\infty \le \frac{1}{2\mfc^4}
\]
for all $t\in I_1$ by \eqref{eq:MF}, the map $t\mapsto V_t$ is Lipschitz on $I_1$ and therefore extends uniquely to a continuous path on $[-\mft_1,\mft_1]$, still denoted by $V$, with values in $\mathcal B_\eta(U_0)$. 
Since $\mathcal B_\eta(U_0)$ is closed in $C(\mathcal X)^3$ and $V_t\in\mathcal B_\eta(U_0)$ for all $t\in I_1$, the extended path also satisfies $V_t\in\mathcal B_\eta(U_0)$ for all $t\in[-\mft_1,\mft_1]$. 
Integrating the identity $V'(t)=F(V_t)$ from $0$ to $t$ and using continuity at the endpoints yields
\[
V_t = U_0+\int_0^t F(V_s)\,ds,
\qquad t\in[-\mft_1,\mft_1].
\]
Thus $V$ is a fixed point of $\mathcal T$ in $\mathsf X_{I_1}$. By \Cref{thm:ODE}, the fixed point is unique, so $V_t=U_t$ for all $t\in I_1$.

For the pointwise statement, fix $x\in\mathcal X$ and consider the bounded linear evaluation map $\mathrm{ev}_x\colon C(\mathcal X)^3\to\mathbb R^3$ given by $\mathrm{ev}_x(W)=W(x)$. By \Cref{lem:commute-der},
\[
\frac{d}{dt}U_t(x)=F(U_t)(x)=F(Q,U_t)(x),
\]
so $t\mapsto U_t(x)$ is $C^1$ on $I_1$ and solves the ODE pointwise.
\end{proof}
\begin{remark}
\label{rem:parameter_dependence_ODE}
The solution constructed in the proof of \Cref{mainthm:ODE-C1} depends continuously, and hence measurably, on the pair \((Q,U_0)\). While the details are routine, we include them for completeness.

Let \(\mathcal P(\mathcal X)\) denote the set of Borel probability measures on
\(\mathcal X=[0,1]^d\), equipped with the weak topology, and define
\[
\mathcal A_\eta
=
\left\{
U_0\in C(\mathcal X;[0,1])^3:
\mathcal B_{2\eta}(U_0)\subset C(\mathcal X;[\eta,1-\eta])^3
\right\}.
\]
Also write
\[
\overline I_1=[-\mft_1,\mft_1],
\qquad
\mathbb X=C(\overline I_1;C(\mathcal X)^3),
\qquad
\mathbb B_\eta=\{\mathsf V\in\mathbb X:\|\mathsf V\|_{\sup}\le \eta\}.
\]
For \((Q,U_0)\in \mathcal P(\mathcal X)\times \mathcal A_\eta\) and
\(\mathsf V\in \mathbb B_\eta\), define
\[
(\Phi_{Q,U_0}\mathsf V)(t)
=
\int_0^t F(Q,U_0+\mathsf V(s))\,ds,
\qquad t\in \overline I_1.
\]
Since \(\|\mathsf V\|_{\sup}\le \eta\), we have
\(U_0+\mathsf V(s)\in \mathcal B_\eta(U_0)\subset \mathcal B_{2\eta}(U_0)\) for
all \(s\in \overline I_1\), and therefore
\(U_0+\mathsf V(s)\in C(\mathcal X;[\eta,1-\eta])^3\). Thus the right-hand side
is well defined.

The estimates proved in \eqref{eq:MF} and \Cref{l:Flips} depend only on the
bounds \(q_1,q_0,e\in[\eta,1-\eta]\) and on the regularity constants of
\(\lambda\). Hence, for every \(Q\in\mathcal P(\mathcal X)\) and every
\(U,V\in C(\mathcal X;[\eta,1-\eta])^3\),
\[
\|F(Q,U)\|_\infty\le \frac{1}{2\mfc^4},
\qquad
\|F(Q,U)-F(Q,V)\|_\infty\le 4\mfc^{-6}\|U-V\|_\infty.
\]
Therefore, exactly as in the proof of \Cref{thm:ODE}, the map
\(\Phi_{Q,U_0}\) sends \(\mathbb B_\eta\) into itself and is a contraction with
Lipschitz constant at most \(1/2\), uniformly over
\((Q,U_0)\in \mathcal P(\mathcal X)\times \mathcal A_\eta\).
Let \(\mathsf V^{Q,U_0}\in\mathbb B_\eta\) denote its unique fixed point, and
define
\[
\mathsf U^{Q,U_0}=U_0+\mathsf V^{Q,U_0}\in \mathbb X.
\]
Then \(\mathsf U^{Q,U_0}\) is the unique solution path furnished by
\Cref{mainthm:ODE-C1}.

We next show that \((Q,U_0)\mapsto \mathsf U^{Q,U_0}\) is continuous. First,
\((Q,U)\mapsto F(Q,U)\) is continuous on
\(\mathcal P(\mathcal X)\times C(\mathcal X;[\eta,1-\eta])^3\). Indeed, if
\(Q_m\Rightarrow Q\) weakly and \(U_m=(q_{1,m},q_{0,m},e_m)\to U=(q_1,q_0,e)\)
in \(C(\mathcal X)^3\), then
\[
\lambda(e_m)\to \lambda(e),
\qquad
\lambda(e_m)(q_{1,m}-q_{0,m})\to \lambda(e)(q_1-q_0)
\]
uniformly on \(\mathcal X\). Hence
\[
Q_m[\lambda(e_m)]\to Q[\lambda(e)],
\qquad
Q_m[\lambda(e_m)(q_{1,m}-q_{0,m})]\to Q[\lambda(e)(q_1-q_0)],
\]
so \(\Omega(Q_m,U_m)\to \Omega(Q,U)\) and \(\psi(Q_m,U_m)\to \psi(Q,U)\), and
the explicit formula for \(F\) then yields \(F(Q_m,U_m)\to F(Q,U)\) in
\(C(\mathcal X)^3\).

Now let \((Q_m,U_{0,m},\mathsf V_m)\to (Q,U_0,\mathsf V)\) in
\(\mathcal P(\mathcal X)\times \mathcal A_\eta\times \mathbb B_\eta\). Then
\begin{align*}
\|\Phi_{Q_m,U_{0,m}}\mathsf V_m-\Phi_{Q,U_0}\mathsf V\|_{\sup}
&\le
2\mft_1\sup_{s\in\overline I_1}
\|F(Q_m,U_{0,m}+\mathsf V_m(s))-F(Q_m,U_0+\mathsf V(s))\|_\infty\\
&\quad+
2\mft_1\sup_{s\in\overline I_1}
\|F(Q_m,U_0+\mathsf V(s))-F(Q,U_0+\mathsf V(s))\|_\infty .
\end{align*}
The first term tends to \(0\) by the uniform Lipschitz bound above. For the
second term, the set
\[
K=\{U_0+\mathsf V(s):s\in\overline I_1\}\subset C(\mathcal X;[\eta,1-\eta])^3
\]
is compact, because \(s\mapsto U_0+\mathsf V(s)\) is continuous on the compact
interval \(\overline I_1\). Also, the set
\[
\{Q\}\cup\{Q_m:m\ge 1\}\subset \mathcal P(\mathcal X)
\]
is compact, since \(Q_m\to Q\). Therefore
\(F\) is uniformly continuous on the compact product of these two sets, and
hence
\[
\sup_{U\in K}\|F(Q_m,U)-F(Q,U)\|_\infty\to 0.
\]
Thus \((Q,U_0,\mathsf V)\mapsto \Phi_{Q,U_0}\mathsf V\) is continuous.

Finally, let \((Q_m,U_{0,m})\to(Q,U_0)\). Since
\(\mathsf V^{Q_m,U_{0,m}}=\Phi_{Q_m,U_{0,m}}(\mathsf V^{Q_m,U_{0,m}})\) and
\(\mathsf V^{Q,U_0}=\Phi_{Q,U_0}(\mathsf V^{Q,U_0})\), we have
\begin{align*}
\|\mathsf V^{Q_m,U_{0,m}}-\mathsf V^{Q,U_0}\|_{\sup}
&\le
\|\Phi_{Q_m,U_{0,m}}(\mathsf V^{Q_m,U_{0,m}})
-\Phi_{Q_m,U_{0,m}}(\mathsf V^{Q,U_0})\|_{\sup}\\
&\quad+
\|\Phi_{Q_m,U_{0,m}}(\mathsf V^{Q,U_0})
-\Phi_{Q,U_0}(\mathsf V^{Q,U_0})\|_{\sup}\\
&\le
\frac12\|\mathsf V^{Q_m,U_{0,m}}-\mathsf V^{Q,U_0}\|_{\sup}
+
\|\Phi_{Q_m,U_{0,m}}(\mathsf V^{Q,U_0})
-\Phi_{Q,U_0}(\mathsf V^{Q,U_0})\|_{\sup},
\end{align*}
and therefore
\[
\|\mathsf V^{Q_m,U_{0,m}}-\mathsf V^{Q,U_0}\|_{\sup}
\le
2\,
\|\Phi_{Q_m,U_{0,m}}(\mathsf V^{Q,U_0})
-\Phi_{Q,U_0}(\mathsf V^{Q,U_0})\|_{\sup}\to 0.
\]
Hence
\[
(Q,U_0)\longmapsto \mathsf U^{Q,U_0}
\]
is continuous from \(\mathcal P(\mathcal X)\times \mathcal A_\eta\) into
\(C(\overline I_1;C(\mathcal X)^3)\). In particular, for each fixed
\(t\in I_1\), the map \((Q,U_0)\mapsto U_t^{Q,U_0}\) is continuous.

Consequently, if \(A\) is an event and \((Q,U_0)\) is a measurable random
element of \(\mathcal P(\mathcal X)\times C(\mathcal X;[0,1])^3\) such that
\[
(Q(\omega),U_0(\omega))\in \mathcal P(\mathcal X)\times \mathcal A_\eta
\qquad\text{for all }\omega\in A,
\]
then
\[
\omega\longmapsto \mathsf U^{Q(\omega),U_0(\omega)}
\]
is a measurable \(C(\overline I_1;C(\mathcal X)^3)\)-valued random element on
\(A\). After any measurable extension on \(A^c\), the full random path is
globally measurable. In particular, for each fixed \(t\in I_1\), the map
\[
\omega\longmapsto U_t^{Q(\omega),U_0(\omega)}
\]
is measurable on \(A\). Whenever \Cref{as::positivity} holds for
a given pair \((Q,U_0)\), we have \(U_0\in \mathcal A_\eta\), so the preceding
conclusion applies on any event where \Cref{as::positivity} is assumed.
\end{remark}

\section{Proof of \texorpdfstring{\Cref{mainthm:det-bracket-one}}{Second Main Result}}

Throughout this section we work in the empirical-marginal setting of \Cref{mainthm:det-bracket-one}. The first main theorem supplies the deterministic path \(t\mapsto U_t\); the present theorem shows that, with high conditional probability, the empirical score along that path has a unique zero in a deterministic neighborhood of \(0\).

The key point is that the argument is \emph{conditional on} \(\mathcal G_n\). On the good event \(\mathcal A_n\), the empirical design law
\[
Q_n=\frac1n\sum_{i=1}^n\delta_{X_i},
\]
the initial nuisance estimate \(U_0\), the covariates \(X_1,\dots,X_n\), and therefore the path \(t\mapsto U_t\) from \Cref{mainthm:ODE-C1} are all fixed once we condition on \(\mathcal G_n\). Consequently, only the pairs \((A_i,Y_i)\) remain random, conditionally independent with conditional law \(p^*(\cdot,\cdot\mid X_i)\). Thus the proof is a conditional fixed-design argument: once \(\mathcal G_n\) is fixed, all analytic properties of the path and score process are deterministic, while the concentration bounds apply only to the remaining treatment/outcome randomness.

The proof has three steps. First, we show that along the ODE path the map
\[
t\longmapsto L_n(t)=\frac1n\sum_{i=1}^n\log p_t(A_i,Y_i\mid X_i)
\]
is twice continuously differentiable, with
\[
L_n'(t)=\P_n[D^*(\cdot;U_t)],
\qquad
L_n''(t)=\P_n[\partial_t D^*(\cdot;U_t)].
\]
This reduces the theorem to controlling the empirical score and empirical curvature along a deterministic path. Second, we identify the conditional mean of \(L_n''(t)\) and show that it has strictly negative curvature on a deterministic interval: the population curvature identity gives
\[
\bar P_n^*[\partial_t D^*(\cdot;U_t)]
=
-\E_t[(D^*(O;U_t))^2]+R_t,
\]
where the remainder \(R_t\) is small because the initial conditional law is locally close to the truth. Combined with the square-bound assumption, this yields a uniform negative upper bound for the conditional mean of \(L_n''(t)\). Third, we use conditional concentration inequalities to show that \(L_n''(t)\) stays close to its conditional mean uniformly over \(t\), and that \(L_n'(0)\) is small. On the resulting high-probability event, \(L_n'\) is strictly decreasing on \([-\mft_2,\mft_2]\) and changes sign there, so it has a unique zero \(\hat t\). This gives the desired targeted update and the exponential conditional probability bound in the theorem.

We now begin the technical estimates needed for this argument. The first group of lemmas establishes regularity of the score map \(U\mapsto D^*(\cdot;U)\), which in turn yields the differentiability of \(t\mapsto L_n(t)\). The second group establishes the deterministic curvature bound and the conditional empirical-process bounds that force \(L_n'\) to have a unique root.

\begin{lemma}\label{l:Lip}
For all $U,V\in\mathcal{B}_{2\eta}$, we have 
\[
\|D^*(\cdot;U)-D^*(\cdot;V)\|_\infty\le 18 \mfc^{-6}\|U-V\|_\infty
\]
and 
\[
\sup_{U\in\mathcal B_{2\eta}}\|D^*(\cdot;U)\|_\infty\le 18 \mfc^{-6}.
\]
\end{lemma}

\begin{proof}
Denote 
\[ U=(q_1,q_0,e), \quad V=(q_1',q_0',e'), \quad \Delta=\|U-V\|_\infty,\]
and 
\[
\tau=q_1-q_0, \quad \tau'=q_1'-q_0', \quad \Omega=\Omega(U), \quad \Omega'=\Omega(V), \quad 
\psi=\psi(U), \quad \psi'=\psi(V).
\]

By \Cref{as::positivity}, for all $x\in \mathcal X$,
\begin{equation}\label{e:bounds1}
e(x)\in[\eta,1-\eta],\quad \lambda\big(e(x) \big) \in[\Lambda_{\min},\Lambda_{\max}], \quad \Omega\in[\Lambda_{\min},\Lambda_{\max}].
\end{equation}
Further, for all $y \in \{0,1\}$,  $x \in \mathcal X$, and $a\in\{0,1\}$, we have
\begin{equation}\label{e:bounds2}
|y-q_a(x)|\le 1, \qquad \tau(x) \in [-1,1],
\end{equation}
and 
\begin{equation}\label{e:bounds3}
\frac{a}{e(x)}\le \frac{1}{\eta}, \qquad \frac{1-a}{1-e(x)}\le \frac{1}{\eta}.\end{equation}
Using \eqref{e:bounds1} and the bound on $\tau$ in  \eqref{e:bounds2}, we have 
\begin{equation}\label{e:bounds4}
\big |\psi(U) \big|
=\left|\frac{\E_Q\Big[\lambda\big(e(X)\big)\tau(X)\Big]}{\E_Q\Big[\lambda\big(e(X)\big)\Big]}\right|
\le \frac{\Lambda_{\max}\E_Q\Big[ \big|\tau(X)\big| \Big] }{\Lambda_{\min}}
\le \frac{\Lambda_{\max}}{\Lambda_{\min}}.
\end{equation}
Recalling the definitions of $D^*_q$, and $D^*_e$ in \eqref{e:EIF}, and using \eqref{e:bounds1}, \eqref{e:bounds2}, \eqref{e:bounds3}, and \eqref{e:bounds4}, we have 
\[
\|D^*_q(\cdot;U)\|_\infty
\le \frac{2\Lambda_{\max}}{\Lambda_{\min}\eta},
\quad
\|D^*_e(\cdot;U)\|_\infty
\le \frac{\dot\Lambda_{\max}}{\Lambda_{\min}}\left(1+\frac{\Lambda_{\max}}{\Lambda_{\min}}\right).
\]
By the triangle inequality,

\begin{equation}
\label{eq:CR}
\sup_{U\in\mathcal B_{2\eta}}\|D^*(\cdot;U)\|_\infty
\le 4 \mfc^{-4}.
\end{equation}

Next, for the Lipschitz bound, we will show that the functions $U\mapsto D^*_q(\cdot;U)$, and $U\mapsto D^*_e(\cdot;U)$ are Lipschitz, then sum the resulting Lipschitz constants.

First, for the $D^*_q$ term,
write $$D^*_q(o ; U)=\frac{\lambda(e)}{\Omega}\left[ a b_1-(1-a)b_0\right]$$ with 
$$b_1=e^{-1}(y-q_1),\quad b_0 =(1-e)^{-1}(y-q_0),$$
and define $b_0'$ and $b_1'$ analogously. 
Thus, for every $o \in \mathcal O$, 
\begin{equation}\label{e:twotbound}
\big |D^*_q(o; U)-D^*_q(o; V) \big| \le T_{q,1}(o)+T_{q,2}(o)
\end{equation}
where
\begin{align}
\begin{split}
T_{q,1}(o) &=\left|\frac{\lambda(e)}{\Omega}-\frac{\lambda(e')}{\Omega'}\right|\big|ab_1-(1-a)b_0\big|,\\
T_{q,2}(o) &= \frac{\lambda(e')}{\Omega'}\big|a (b_1-b_1')-(1-a)(b_0-b_0')\big|.
\end{split}
\end{align}
We have 
\begin{equation}\label{e:alphabd}
\left |\frac{\lambda(e)}{\Omega}-\frac{\lambda(e')}{\Omega'}\right|
\le \frac{|\lambda(e)-\lambda(e')|}{\Omega}
+\lambda(e')\left|\frac{1}{\Omega}-\frac{1}{\Omega'}\right|
\le \frac{\dot\Lambda_{\max}}{\Lambda_{\min}}\Delta
+\frac{\Lambda_{\max}\dot\Lambda_{\max}}{\Lambda_{\min}^2}\Delta,
\end{equation}
where the last inequality follows from the mean value theorem,  \Cref{as::positivity}, and \eqref{e:bounds1}. Also, since $|a| \le 1$, $$|a b_1-(1-a)b_0 |\le |b_1|+|b_0|\le \frac{2}{\eta}.$$ 
The previous bounds imply $$T_{q,1}\le \frac{2}{\eta}\left(\frac{\dot\Lambda_{\max}}{\Lambda_{\min}}+\frac{\Lambda_{\max}\dot\Lambda_{\max}}{\Lambda_{\min}^2}\right)\Delta.$$
Next, since \[ \frac{\lambda(e')}{\Omega'}\le \frac{ \Lambda_{\max}}{\Lambda_{\min}},\] we can bound
\begin{align*}
    |b_1-b_1'|
&=\left|\frac{y-q_1}{e}-\frac{y-q_1'}{e'}\right|\\
&\le |y-q_1|\left|\frac{1}{e}-\frac{1}{e'}\right|+\frac{|q_1-q_1'|}{e'}\\
&\le \frac{|e-e'|}{\eta^2}+\frac{|q_1-q_1'|}{\eta}\\
&\le \left(\frac{1}{\eta^2}+\frac{1}{\eta}\right)\Delta,
\end{align*}
and similarly we have 
\[
|b_0-b_0'|
\le \left(\frac{1}{\eta^2}+\frac{1}{\eta}\right)\Delta.
\]
Therefore
\[
T_{q,2}\le \frac{2\Lambda_{\max}}{\Lambda_{\min}}
\left(\frac{1}{\eta^2}+\frac{1}{\eta}\right)\Delta.
\]
Putting our estimates on $T_{q,1}$ and $T_{q,2}$ into \eqref{e:twotbound}, we obtain 
\[
\|D^*_q(\cdot;U)-D^*_q(\cdot;V)\|_\infty
\le 8 \mfc^{-6}\Delta.
\]

Then, for the $D^*_e$ term,
write $D^*_e = \rho \beta \zeta$ with $\rho=\dot\lambda(e)/\Omega$, $\beta=\tau-\psi$, and $\zeta=a-e$. 
Then
\[
\big|D^*_e(U)-D^*_e(V)\big|
\le |\rho -\rho'||\beta||\zeta|
+ |\rho'||\beta -\beta'||\zeta|
+ |\rho'||\beta'||\zeta-\zeta'|.
\]
Using the definitions of $\tau$ and $\psi$, 
\begin{equation}\label{e:betabound}
|\beta |\le |\tau|+|\psi|\le 1+\frac{\Lambda_{\max}}{\Lambda_{\min}},
\end{equation}
and
\begin{equation}\label{e:betadiff}
|\beta -\beta'|\le |\tau-\tau'|+|\psi-\psi'|\ \le \left (2+\frac{\Lambda_{\max}\dot\Lambda_{\max}}{\Lambda_{\min}^2}
+\frac{\dot\Lambda_{\max}+2\Lambda_{\max}}{\Lambda_{\min}}\right)\Delta.
\end{equation}
where in the last inequality, we used the bound 
\begin{align*}
|\psi-\psi'|
&\le \left|\frac{1}{\Omega}-\frac{1}{\Omega'}\right|\left|\E[\lambda(e)\tau]\right|+\frac{1}{\Omega'}\E[\left|\lambda(e)\tau-\lambda(e')\tau'\right|]\\
&\le \frac{\Lambda_{\max}\dot\Lambda_{\max}}{\Lambda_{\min}^2}\Delta
+\frac{\dot\Lambda_{\max}+2\Lambda_{\max}}{\Lambda_{\min}}\Delta,
\end{align*}
which follows from the mean value theorem.
Using \eqref{e:betabound} and the estimates 
\begin{align*}
    |\zeta |,|\zeta '|\le 1,\quad 
    |\rho'|\le \frac{\dot\Lambda_{\max}}{\Lambda_{\min}},\quad|\zeta -\zeta '|=|e-e'|\le \Delta,
\end{align*}
it suffices to bound $|\rho -\rho '|$ and $|\beta - \beta'|$.
We already gave a bound for $|\beta -\beta'|$ in \eqref{e:betadiff}. 
Further, 
\[
|\rho -\rho '|\le \frac{|\dot\lambda(e)-\dot\lambda(e')|}{\Omega}
+\big| \dot\lambda(e')\big| \left|\frac{1}{\Omega}-\frac{1}{\Omega'}\right|
\le \frac{\Lambda_{2,\max}}{\Lambda_{\min}}\Delta
+\frac{\dot\Lambda_{\max}^2}{\Lambda_{\min}^2}\Delta.
\]
Hence
\[
|D^*_e(o; U)-D^*_e(o; V)|
\le 10 \mfc^{-6}\Delta.
\]

By the triangle inequality, 
\[
\|D^*(\cdot;U)-D^*(\cdot;V)\|_\infty
\le 18 \mfc^{-6}\|U-V\|_\infty.
\]
Together with \eqref{eq:CR}, this proves the lemma.
\end{proof}

\begin{lemma}
\label{lemma:wJ}
Under \Cref{as::positivity}, the map $U\mapsto D^*(\cdot;U)$ lies in $C^1(\mathcal B^{\circ}_{2\eta};C(\mathcal O))$. Let $J(U)$ denote its Fr\'echet derivative at $U$ and set
\[
\|J(U)\|_{\mathrm{op}}=\sup_{\|h\|_{\infty}\leq 1}\|J(U)[h]\|_{\infty}.
\]
Then, for all $U,V\in\mathcal{B}_{2\eta}^{\circ}$,
\[
\|J(U)-J(V)\|_{\rm op}\leq 104\mfc^{-10}\|U-V\|_{\infty}.
\]
\end{lemma}

\begin{proof}
Let $U = (q_1, q_0, e)$ and $V = (q_1', q_0', e')$ be elements of $\mathcal{B}_{2\eta}^{\circ}(U_0)$, and let $\Delta = \|U - V\|_\infty$. 
Let $h = (h_1, h_0, h_e)$ be a direction in $C(\mathcal{X})^3$ with $\|h\|_\infty \le 1$. Fix $o=(x,a,y)\in\mathcal O$. All identities and inequalities below are understood pointwise in $o$; at the end we take the supremum over $o\in\mathcal O$ to obtain the claimed $\|\cdot\|_\infty$ bounds.

The Fréchet derivative of the efficient influence function is
\[
J(U)[h] = J_q(U)[h] + J_e(U)[h],
\]
where $J_q(U)[h]$ and $J_e(U)[h]$ are the directional derivatives of $D_q^*(\cdot; U)$ and $D_e^*(\cdot; U)$, respectively. To justify the existence of these Fréchet derivatives, note first that on the open set $\mathcal B^\circ_{2\eta}(U_0)$, Assumption~\ref{as::positivity} implies that the coordinates $q_1,q_0,e$ take values in $[\eta,1-\eta]$, so the maps
\[
f\mapsto \lambda\circ f,\qquad f\mapsto \dot\lambda\circ f,\qquad f\mapsto \frac1f,\qquad f\mapsto \frac1{1-f}
\]
are $C^1$ as maps on the corresponding open subset of $C(\mathcal X)$, with derivatives given by the usual pointwise formulas. Moreover, pointwise multiplication on $C(\mathcal X)$ is continuous bilinear, and $f\mapsto \E[f(X)]$ is bounded linear. It follows, by repeated applications of \Cref{l:chain_rule} together with the product and quotient rules, that the maps $U\mapsto \Omega(U)$, $U\mapsto N(U)$, $U\mapsto \psi(U)$, and hence $U\mapsto D_q^*(\cdot;U)$, $U\mapsto D_e^*(\cdot;U)$, are Fréchet differentiable on $\mathcal B^\circ_{2\eta}(U_0)$. The formulas displayed below are precisely these Fréchet derivatives. We will systematically bound the magnitude of each derivative component and the difference $|J(U)[h] - J(V)[h]|$. 

Define
\[
N(U)=\E_Q[\lambda(e(X))(q_1-q_0)(X)],\]
which allows us to write
\[
\psi(U)=\frac{N(U)}{\Omega(U)}.
\]
The expectation functional $L(W)=\E_Q[W(X)]$ is a bounded linear function $C(\mathcal X)\to\mathbb R$.
By \Cref{l:chain_rule}, for any direction $h=(h_1,h_0,h_e)\in C(\mathcal{X})^3$,
\[
D\Omega(U)[h]=\E[(\dot\lambda\circ e)(X)h_e(X)],
\]
and by linearity plus \Cref{l:chain_rule},
\[
DN(U)[h]=\E\left[(\dot\lambda\circ e)(X)\,h_e(X)\,(q_1-q_0)(X)+(\lambda\circ e)(X)\,(h_1-h_0)(X)\right].
\]
We then know
$$|D\Omega(U)[h]| \le \dot{\Lambda}_{\max}\|h\|_\infty \le \mfc^{-1}\|h\|_\infty,$$
and we know $$|D\Omega(U)[h] - D\Omega(V)[h]| \le \mathbb{E}[|\dot{\lambda}(e) - \dot{\lambda}(e')| |h_e|] \le \Lambda_{2,\max}\Delta\|h\|_\infty \le \mfc^{-1}\Delta\|h\|_\infty.$$
Similarly, we have $$|DN(U)[h]| \le (\dot{\Lambda}_{\max} + 2\Lambda_{\max})\|h\|_\infty \le 3\mfc^{-1}\|h\|_\infty,$$
and
\begin{align*}
|DN(U)[h] - DN(V)[h]| &\le \mathbb{E}[|\dot{\lambda}(e)\tau - \dot{\lambda}(e')\tau'||h_e| + |\lambda(e) - \lambda(e')||h_1 - h_0|]\\
&\le  3\mfc^{-1}\Delta\|h\|_{\infty} + 2\mfc^{-1}\Delta\|h\|_{\infty}\\
&= 5\mfc^{-1}\Delta\|h\|_{\infty}
\end{align*}
From the proof of \Cref{l:Lip}, we know $|\psi(U) - \psi(V)| \le 4\mfc^{-4}\Delta$. The derivative is $$D\psi(U)[h] = \frac{DN(U)[h]}{\Omega} - \frac{N(U)D\Omega(U)[h]}{\Omega^2}.$$
Then, we have
$$|D\psi(U)[h]| \le \frac{3\mfc^{-1}\|h\|_\infty}{\mfc} + \frac{\mfc^{-1}(\mfc^{-1}\|h\|_\infty)}{\mfc^2} \le 4\mfc^{-4}\|h\|_\infty,$$
and
$$|D\psi(U)[h] - D\psi(V)[h]| \le  14\mfc^{-6}\Delta\|h\|_\infty.$$

Next, we bound the derivative component $J_q(U)[h]$. Let $D_{q1}^*(U) = c_1(U) A(Y-q_1)$, where \[ c_1(U) = \frac{\lambda\circ e}{\Omega(U) e}.\] Define \[f_1(e) = \frac{\lambda\circ e}{e}.\] By the quotient rule,$$Df_1(e)[h] = \frac{(\dot\lambda\circ e)e - (\lambda\circ e)}{e^2}h_e.$$
So, we have$$|Df_1(e)[h]| \le \frac{2\mfc^{-1}}{\mfc^2}\|h\|_\infty = 2\mfc^{-3}\|h\|_\infty.$$
Using the quotient difference rules and the bound $|\lambda''(e(X))| \le \mfc^{-1}$,
\begin{equation*}|Df_1(e)[h] - Df_1(e')[h]| \le 5\mfc^{-5}\Delta\|h\|_\infty.\end{equation*}
For $c_1(U)$, the derivative is$$Dc_1(U)[h] = \frac{Df_1(e)[h]}{\Omega(U)} - \frac{f_1(e)D\Omega(U)[h]}{\Omega(U)^2}.$$
We then know,
$$|Dc_1(U)[h]| \le \frac{2\mfc^{-3}\|h\|_\infty}{\mfc} + \frac{\mfc^{-2}(\mfc^{-1}\|h\|_\infty)}{\mfc^2} \le 3\mfc^{-5}\|h\|_\infty.$$
By splitting the difference into two terms,
\begin{align*}
|Dc_1(U)[h] - Dc_1(V)[h]| &\le \left| \frac{Df_1}{\Omega} - \frac{Df_1'}{\Omega'} \right| + \left| \frac{f_1 D\Omega}{\Omega^2} - \frac{f_1' D\Omega'}{(\Omega')^2} \right|\\  &\le 7\mfc^{-9}\Delta\|h\|_\infty + 5\mfc^{-9}\Delta\|h\|_\infty\\ &\le 12\mfc^{-9}\Delta\|h\|_\infty.
\end{align*}
The directional derivative of the treated component is$$J_{q1}(U)[h] = Dc_1(U)[h]A(Y-q_1) - c_1(U)Ah_1.$$Similarly, we have
\begin{align*}
|J_{q1}(U)[h] - J_{q1}(V)[h]| &\le |Dc_1 - Dc_1'||A(Y-q_1)| + |Dc_1'||A(q_1' - q_1)| + |c_1 - c_1'||Ah_1|\\
&\le 12\mfc^{-9}\Delta\|h\|_\infty + 3\mfc^{-5}\|h\|_\infty\Delta + 3\mfc^{-7}\Delta\|h\|_\infty \\
&\le 18\mfc^{-9}\Delta\|h\|_\infty.
\end{align*}
Summing the symmetric treated and control components gives
$$|J_q(U)[h] - J_q(V)[h]| \le 36\mfc^{-9}\Delta\|h\|_\infty.$$

Now, we bound the derivative component $J_e(U)[h]$. We write $D_e^*(U) = \rho(U)\beta(U)\zeta(U)$, where$$\rho(U) = \frac{\dot\lambda\circ e}{\Omega(U)}, \quad \beta(U) = (q_1 - q_0) - \psi(U), \quad \zeta(U) = A - e.$$
The derivative of $\rho(U)$ is$$D\rho(U)[h] = \frac{(\lambda''\circ e)h_e}{\Omega(U)} - \frac{(\dot\lambda\circ e)D\Omega(U)[h]}{\Omega(U)^2}.$$
So, we have$$|D\rho(U)[h]| \le \frac{\mfc^{-1}\|h\|_\infty}{\mfc} + \frac{\mfc^{-1}(\mfc^{-1}\|h\|_\infty)}{\mfc^2} \le 2\mfc^{-4}\|h\|_\infty.$$
Using the assumption $\Lambda_{3,\max} \le \mfc^{-1}$, we bound the difference:
\begin{align*}|D\rho(U)[h] - D\rho(V)[h]| &\le \left| \frac{(\lambda''\circ e)h_e}{\Omega} - \frac{(\lambda''\circ e')h_e}{\Omega'} \right| + \left| \frac{(\dot\lambda\circ e)D\Omega}{\Omega^2} - \frac{(\dot\lambda\circ e')D\Omega'}{(\Omega')^2} \right|\\  
&\le 2\mfc^{-4}\Delta\|h\|_\infty + 3\mfc^{-8}\Delta\|h\|_\infty\\ 
&\le 5\mfc^{-8}\Delta\|h\|_\infty.
\end{align*}
For $\beta(U)$, the derivative is $D\beta(U)[h] = h_1 - h_0 - D\psi(U)[h]$, which implies $|D\beta(U)[h]| \le 6\mfc^{-4}\|h\|_\infty$, and
$$|D\beta(U)[h] - D\beta(V)[h]| = |D\psi(V)[h] - D\psi(U)[h]| \le 14\mfc^{-6}\Delta\|h\|_\infty.$$
For $\zeta(U)$, the derivative is $D\zeta(U)[h] = -h_e$, which implies $|D\zeta(U)[h]| \le \|h\|_\infty$, and $|D\zeta(U)[h] - D\zeta(V)[h]| = 0$. By the product rule, the derivative is$$J_e(U)[h] = D\rho(U)[h]\beta(U)\zeta(U) + \rho(U)D\beta(U)[h]\zeta(U) + \rho(U)\beta(U)D\zeta(U)[h].$$
We bound the difference of each term. From the proof of \Cref{l:Lip}, we know $|\beta(U)| \le 2\mfc^{-2}$, $|\beta(U) - \beta(V)| \le 6\mfc^{-4}\Delta$, $|\rho(U)| \le \mfc^{-2}$, $|\rho(U) - \rho(V)| \le 2\mfc^{-4}\Delta$, $|\zeta(U)|\le 1$, and $|\zeta(U)-\zeta(V)|\le\Delta$.
We then know,
\begin{align*}|D\rho\cdot\beta\cdot\zeta - D\rho'\cdot\beta'\cdot\zeta'| &\le |D\rho - D\rho'||\beta||\zeta| + |D\rho'||\beta - \beta'||\zeta| + |D\rho'||\beta'||\zeta - \zeta'|\\ &\le 10\mfc^{-10}\Delta\|h\|_\infty + 12\mfc^{-8}\Delta\|h\|_\infty + 4\mfc^{-6}\Delta\|h\|_\infty\\ 
&\le 26\mfc^{-10}\Delta\|h\|_\infty.
\end{align*}
Similarly, we have
\begin{align*}|\rho\cdot D\beta\cdot\zeta - \rho'\cdot D\beta'\cdot\zeta'| &\le |\rho - \rho'||D\beta||\zeta| + |\rho'||D\beta - D\beta'||\zeta| + |\rho'||D\beta'||\zeta - \zeta'|\\
&\le 12\mfc^{-8}\Delta\|h\|_\infty + 14\mfc^{-8}\Delta\|h\|_\infty + 6\mfc^{-6}\Delta\|h\|_\infty\\ 
&\le 32\mfc^{-8}\Delta\|h\|_\infty.
\end{align*}
and
\begin{align*}|\rho\cdot\beta\cdot D\zeta - \rho'\cdot\beta'\cdot D\zeta'| &\le |\rho - \rho'|| \beta || D\zeta | + |\rho'|| \beta - \beta'|| D\zeta | + 0\\
&\le (2\mfc^{-4}\Delta)(2\mfc^{-2})\|h\|_\infty + (\mfc^{-2})(6\mfc^{-4}\Delta)\|h\|_\infty\\
&\le  10\mfc^{-6}\Delta\|h\|_\infty.
\end{align*}
Summing these gives$$|J_e(U)[h] - J_e(V)[h]| \le 68\mfc^{-10}\Delta\|h\|_\infty.$$
Finally, combining the bounds for the two components of the Fréchet derivative gives
\begin{align*}
\|J(U)[h] - J(V)[h]\|_\infty
&\le \|J_q(U)[h] - J_q(V)[h]\|_\infty + \|J_e(U)[h] - J_e(V)[h]\|_\infty\\ 
&\le 36\mfc^{-9} \Delta \|h\|_\infty + 68\mfc^{-10} \Delta \|h\|_\infty\\ 
&\le 104\mfc^{-10} \Delta \|h\|_\infty.
\end{align*}
Taking the supremum over all $\|h\|_\infty \le 1$ yields
\[
\|J(U)-J(V)\|_{\mathrm{op}}
\le 104\mfc^{-10}\|U-V\|_\infty.
\]
This completes the proof.
\end{proof}

Define the modulus of continuity of $J(\cdot)$ on $\mathcal B_\eta$ by
\begin{equation}\label{e:mod_con_J}
\omega_J(\delta)
=
\sup \left\{
\|J(U)-J(V)\|_{\mathrm{op}}
:
U,V\in \mathcal B_\eta,\ \|U-V\|_\infty\le \delta
\right\},
\end{equation}
and define the generalized inverse by
\[
\omega_J^{-1}(y)
=
\sup \left\{
\delta\ge 0 : \omega_J(\delta)\le y
\right\}.
\]
Then, by \Cref{lemma:wJ}, for every $y>0$,
\begin{equation}\label{e:modulus_bound}
\omega_J(y)\le \frac{104y}{\mfc^{10}},
\qquad
\omega_J^{-1}(y)\ge \frac{y\mfc^{10}}{104}>0.
\end{equation}

We now set
\begin{equation}\label{e:Tstar-def}
T_\star
=
\min\left\{
\frac{\mft_1}{2},
\frac{c_\init}{12312\mfc^{-18}e^{9/4}}
\right\},
\qquad
\mft_2=\delta_\init=\frac{c_\init\mfc^{20}}{10^6}.
\end{equation}
Since $c_\init\in(0,1)$ and $\mfc\le 1$, we have $\mft_2<T_\star<\mft_1$.

The following lemma shows that along the one-step TMLE path,
\[
L_n'(t)=\P_n \left[D^*(\cdot;U_t)\right].
\]

\begin{lemma}\label{lem:Ln-2nd-deriv}
On $\mathcal A_n$, we have
\[
t\mapsto L_n(t)\in C^2((-\mft_1,\mft_1)),
\qquad
L_n'(t)=\P_n \left[D^*(\cdot;U_t)\right],
\qquad
L_n''(t)=\P_n \left[\partial_t D^*(\cdot;U_t)\right].
\]
\end{lemma}

\begin{proof}
Write
\[
S_t(o)=D^*(o;U_t).
\]
Because $U_t$ solves \eqref{e:main_ode}, the score identity \eqref{eq:path-score} from \Cref{lm::ODE_flow_derivation} applies with $Q=Q_n$. Thus, for each $o=(x,a,y)\in\mathcal O$ and $t\in(-\mft_1,\mft_1)$,
\[
\frac{d}{dt}\log p_t(a,y\mid x)=S_t(o),
\]
so $t\mapsto \log p_t(a,y\mid x)$ is $C^1$ once $t\mapsto S_t(o)$ is continuous, and is $C^2$ once $t\mapsto S_t(o)$ is $C^1$.

Consider the map
\[
U\mapsto D^*(\cdot;U)
\quad\text{from}\quad
\mathcal B_{2\eta}\subset C(\mathcal X)^3
\quad\text{into}\quad
C(\mathcal O).
\]
By \Cref{lemma:wJ}, this map is Fr\'echet differentiable on $\mathcal B_{2\eta}^\circ$. 
Moreover, by \Cref{l:Lip},
\[
\|D^*(\cdot;U)-D^*(\cdot;V)\|_\infty
\le 18\mfc^{-6}\|U-V\|_\infty,
\qquad U,V\in\mathcal B_{2\eta}.
\]
Fix $U\in\mathcal B_\eta$ and $h\in C(\mathcal X)^3$ with $\|h\|_\infty\le 1$. If $|s|<\eta$, then
\[
\|U+sh-U_0\|_\infty
\le \|U-U_0\|_\infty+|s|\|h\|_\infty
< \eta+\eta
=
2\eta,
\]
so $U+sh\in\mathcal B_{2\eta}^\circ$. Therefore
\[
\left\|
\frac{D^*(\cdot;U+sh)-D^*(\cdot;U)}{s}
\right\|_\infty
\le
18\mfc^{-6}\|h\|_\infty
\le
18\mfc^{-6}.
\]
Letting $s\to 0$ and using Fr\'echet differentiability at $U$ gives
\[
\|J(U)[h]\|_\infty\le 18\mfc^{-6}.
\]
Taking the supremum over $\|h\|_\infty\le 1$ yields
\begin{equation}\label{e:Juniform}
\|J(U)\|_{\mathrm{op}}\le 18\mfc^{-6},
\qquad U\in\mathcal B_\eta.
\end{equation}

By \Cref{mainthm:ODE-C1}, the path $t\mapsto U_t$ lies in
$C^1((-\mft_1,\mft_1);C(\mathcal X)^3)$ and satisfies $U_t'=F(U_t)$.
Hence, by the Banach-space chain rule,
\[
t\mapsto S_t=D^*(\cdot;U_t)
\]
belongs to $C^1((-\mft_1,\mft_1);C(\mathcal O))$ and
\[
\partial_t S_t=J(U_t)[U_t'].
\]
Evaluating pointwise in $o\in\mathcal O$, we obtain
\[
\partial_t S_t(o)=J(U_t)[U_t'](o).
\]

By definition,
\[
L_n(t)=\frac{1}{n}\sum_{i=1}^n \log p_t(A_i,Y_i\mid X_i).
\]
Each summand is therefore $C^2$ in $t$, and finite sums preserve differentiability. Thus
\[
L_n'(t)
=
\frac{1}{n}\sum_{i=1}^n \partial_t\log p_t(A_i,Y_i\mid X_i)
=
\frac{1}{n}\sum_{i=1}^n S_t(O_i)
=
\P_n \left[D^*(\cdot;U_t)\right],
\]
and
\[
L_n''(t)
=
\frac{1}{n}\sum_{i=1}^n \partial_t S_t(O_i)
=
\P_n \left[\partial_t D^*(\cdot;U_t)\right].
\]
This proves the lemma.
\end{proof}

\begin{lemma}\label{lem:TV}
For every $T<\mft_1$,
\[
\sup_{|t|\le T}\|\partial_t D^*(\cdot;U_t)\|_\infty\le 9\mfc^{-10}.
\]
\end{lemma}

\begin{proof}
Recall from \eqref{eq:MF} that
\[
\sup_{U\in\mathcal B_\eta}\|F(U)\|_\infty\le \frac{1}{2}\mfc^{-4}.
\]
Since $U_t$ satisfies $U_t'=F(U_t)$ and $U_t\in\mathcal B_\eta$ for $|t|<\mft_1$, we have
\[
\sup_{|t|<\mft_1}\|U_t'\|_\infty\le \frac{1}{2}\mfc^{-4}.
\]

Fix $t\in(-\mft_1,\mft_1)$ and $h$ such that $t+h\in(-\mft_1,\mft_1)$. By \Cref{l:Lip},
\begin{align*}
\left\|
\frac{D^*(\cdot;U_{t+h})-D^*(\cdot;U_t)}{h}
\right\|_\infty
&\le
18\mfc^{-6}
\frac{\|U_{t+h}-U_t\|_\infty}{|h|}\\
&\le
18\mfc^{-6}
\frac{1}{|h|}
\int_{\min\{t,t+h\}}^{\max\{t,t+h\}}\|U_s'\|_\infty \, ds\\
&\le 9\mfc^{-10}.
\end{align*}
Taking $h\to 0$ gives the claim.
\end{proof}

For the next proof, we need the following standard fact about exchanging derivatives and integrals (see, e.g., \cite[Theorem 2.27(b)]{folland1999real}).

\begin{lemma}\label{l:folland}
Let $(X,\mu)$ be a measure space, and let $I\subset\mathbb R$ be a compact interval.
Let $f\colon X\times I\to\mathbb R$ be measurable such that $f(\cdot,t)\in L^1(\mu)$ for each $t\in I$.
Suppose that $\partial_t f(x,t)$ exists for all $x\in X$ and all $t$ in the interior of $I$, and that there exists $g\in L^1(\mu)$ such that
\[
|\partial_t f(x,t)|\le g(x)
\qquad\text{for all }x\in X\text{ and }t\in I.
\]
Then the function
\[
F(t)=\int_X f(x,t) \, d\mu(x)
\]
is differentiable on the interior of $I$ and satisfies
\[
F'(t)=\int_X \partial_t f(x,t) \, d\mu(x).
\]
\end{lemma}

\begin{lemma}\label{lem:pop-curv}
We have
\[
\E_t \left[\partial_t D^*(O;U_t)\right]
=
-\E_t \left[(D^*(O;U_t))^2\right].
\]
\end{lemma}

\begin{proof}
Let $\nu_n$ denote the product of $Q_n$ on $\mathcal X$ and counting
measure on $\{0,1\}^2$. Then, for each $t\in(-\mft_1,\mft_1)$,
the law
\[
P_t=\widetilde P(Q_n,U_t)
\]
has density $p_t(a,y\mid x)$ with respect to $\nu_n$.

Set
\[
S_t(x,a,y)=D^*((x,a,y);U_t),
\qquad
f((x,a,y),t)=S_t(x,a,y)p_t(a,y\mid x).
\]
By \Cref{lem:Ln-2nd-deriv}, for each $(x,a,y)$ the map
$t\mapsto S_t(x,a,y)$ is differentiable on $(-\mft_1,\mft_1)$.
Also, by \eqref{eq:path-score},
\[
\partial_t p_t(a,y\mid x)
=
p_t(a,y\mid x)\partial_t \log p_t(a,y\mid x)
=
p_t(a,y\mid x)S_t(x,a,y).
\]
Hence
\[
\partial_t f((x,a,y),t)
=
(\partial_t S_t(x,a,y))p_t(a,y\mid x)
+
S_t(x,a,y)^2 p_t(a,y\mid x).
\]

Fix a compact interval $I\subset(-\mft_1,\mft_1)$. For $t\in I$, we have
\[
\left|
\log \frac{p_t(a,y\mid x)}{p_0(a,y\mid x)}
\right|
=
\left|
\int_0^t S_s(x,a,y)\,ds
\right|
\le
18\mfc^{-6}|t|
\le
18\mfc^{-6}\mft_1
=
\frac{9}{4},
\]
where we used \Cref{l:Lip} and $\mft_1=\mfc^6/8$. Therefore
\[
p_t(a,y\mid x)\le e^{9/4}p_0(a,y\mid x).
\]
By \Cref{lem:TV} and \Cref{l:Lip},
\[
|\partial_t S_t(x,a,y)|\le 9\mfc^{-10},
\qquad
|S_t(x,a,y)|\le 18\mfc^{-6}.
\]
Thus, for all $(x,a,y)\in\mathcal X\times\{0,1\}^2$ and all $t\in I$,
\[
|\partial_t f((x,a,y),t)|
\le
\left(9\mfc^{-10}+324\mfc^{-12}\right)e^{9/4}p_0(a,y\mid x).
\]
The right-hand side is $\nu_n$-integrable because
\[
\int p_0(a,y\mid x)\,d\nu_n(x,a,y)=1.
\]
Hence \Cref{l:folland} applies and yields
\[
\frac{d}{dt}\E_t[S_t(O)]
=
\E_t[\partial_t S_t(O)]
+
\E_t[S_t(O)^2]
\qquad\text{for all }t\in I^\circ.
\]

Finally, by \eqref{eq:path-score},
\[
\E_t[S_t(O)]
=
\E_t[D^*(O;U_t)]
=
0
\]
for all $t\in(-\mft_1,\mft_1)$. Differentiating this identity gives
\[
0
=
\E_t[\partial_t D^*(O;U_t)]
+
\E_t[(D^*(O;U_t))^2].
\]
Since $I\subset(-\mft_1,\mft_1)$ was arbitrary, the identity holds for all
$t\in(-\mft_1,\mft_1)$.
\end{proof}

\begin{lemma}\label{l:curv}
We have
\[
\inf_{|t|\le \mft_2}\E_t \left[(D^*(O;U_t))^2\right]\ge \frac{c_\init}{2}.
\]
\end{lemma}

\begin{proof}
Let
\[
f_t(o)=\left(D^*(o;U_t)\right)^2.
\]
We first bound $\|f_t-f_0\|_\infty$. By the identity
\[
|a^2-b^2|\le (|a|+|b|)|a-b|
\]
and \Cref{l:Lip},
\begin{align*}
\|f_t-f_0\|_\infty
&\le
\left(\|D^*(\cdot;U_t)\|_\infty+\|D^*(\cdot;U_0)\|_\infty\right)
\|D^*(\cdot;U_t)-D^*(\cdot;U_0)\|_\infty\\
&\le
(18\mfc^{-6}+18\mfc^{-6})18\mfc^{-6}\|U_t-U_0\|_\infty\\
&=
648\mfc^{-12}\|U_t-U_0\|_\infty.
\end{align*}
By \eqref{eq:MF},
\[
\|U_t-U_0\|_\infty
\le
\int_{\min\{0,t\}}^{\max\{0,t\}}\|F(U_s)\|_\infty \, ds
\le
\frac{1}{2}\mfc^{-4}|t|.
\]
Hence
\begin{equation}\label{eq:ftLipschitz}
\|f_t-f_0\|_\infty\le 324\mfc^{-16}|t|.
\end{equation}

Let $\nu_n$ be the finite measure from the proof of \Cref{lem:pop-curv},
so that $P_t=\widetilde P(Q_n,U_t)$ has density $p_t(a,y\mid x)$ with
respect to $\nu_n$. Define
\[
R_t(x,a,y)=\frac{p_t(a,y\mid x)}{p_0(a,y\mid x)}.
\]
Then, for every bounded measurable function $h$,
\[
\E_t[h(O)]=\E_0[h(O)R_t(O)].
\]

As in the proof of \Cref{lem:pop-curv},
\[
\|\log R_t\|_\infty\le 18\mfc^{-6}|t|,
\]
so
\begin{equation}\label{eq:LR-bdd}
\|R_t\|_\infty\le e^{18\mfc^{-6}|t|},
\qquad
\|R_t-1\|_\infty\le e^{18\mfc^{-6}|t|}18\mfc^{-6}|t|,
\end{equation}
where we used $|e^x - 1| \le e^{|x|} | x |$ in the last inequality. 
Since $|t|\le \mft_2<\mft_1$, we have $e^{18\mfc^{-6}|t|}\le e^{9/4}$. 
Because $\E_t[f_t]=\E_0[f_tR_t]$, we have
\begin{equation}\label{eq:two-terms}
\left|\E_t[f_t]-\E_0[f_0]\right|
\le
(\mathrm I)+(\mathrm{II}),
\end{equation}
where
\[
(\mathrm I)=\E_0[|f_t-f_0|R_t],
\qquad
(\mathrm{II})=\E_0[f_0|R_t-1|].
\]
Using \eqref{eq:ftLipschitz} and the first bound in \eqref{eq:LR-bdd},
\[
(\mathrm I)
\le
\|f_t-f_0\|_\infty\|R_t\|_\infty
\le
324\mfc^{-16}e^{9/4}|t|
\le
324\mfc^{-18}e^{9/4}|t|,
\]
since $\mfc\le 1$.
Also, by \Cref{l:Lip},
\[
\|f_0\|_\infty
\le
\|D^*(\cdot;U_0)\|_\infty^2
\le 324\mfc^{-12},
\]
so the second bound in \eqref{eq:LR-bdd} gives
\[
(\mathrm{II})
\le
324\mfc^{-12}18\mfc^{-6}e^{9/4}|t|
=
5832\mfc^{-18}e^{9/4}|t|.
\]
Combining these estimates with \eqref{eq:two-terms}, we obtain
\begin{equation}\label{eq:cont-bound}
\left|\E_t[f_t]-\E_0[f_0]\right|
\le
6156\mfc^{-18}e^{9/4}|t|.
\end{equation}

Since $\mft_2\le T_\star$ by \eqref{e:Tstar-def},
\[
6156\mfc^{-18}e^{9/4}\mft_2\le \frac{c_\init}{2}.
\]
Therefore, for all $|t|\le \mft_2$,
\[
\E_t[f_t]
\ge
\E_0[f_0]-6156\mfc^{-18}e^{9/4}|t|
\ge
c_\init-\frac{c_\init}{2}
=
\frac{c_\init}{2},
\]
where we used \Cref{as::square_bound} in the second step. Recalling the definition of $f_t$ gives the claim.
\end{proof}

From this point onward, we work on the event $\mathcal A_n$ from the statement of \Cref{mainthm:det-bracket-one} and condition on $\mathcal G_n$. Thus $Q_n$, $U_0$, and $X_1,\dots,X_n$ are fixed. For any bounded measurable function $f\colon\mathcal O\to\mathbb R$, define
\[
\bar P_n^*[f]
=
\frac{1}{n}\sum_{i=1}^n \sum_{a,y\in\{0,1\}} f(X_i,a,y)p^*(a,y\mid X_i).
\]
This is the conditional expectation of $\P_n[f]$ given $\mathcal G_n$.

\begin{lemma}\label{lemma::emprical_score}
For every $T\in(0,\mft_1)$ and every $\eps>0$, on $\mathcal A_n$ we have
\[
\P\left(
\sup_{|t|\le T}\left|(\P_n-\bar P_n^*)[\partial_t D^*(\cdot;U_t)]\right|
\ge 2\eps
\mid
\mathcal G_n
\right)
\le
2m_\star(\eps,\mfc,T)\exp(-C_\star(\eps,\mfc)n),
\]
where $m_\star(\eps,\mfc,T)>0$ and $C_\star(\eps,\mfc)>0$ depend only on $(\eps,\mfc,T)$ and $(\eps,\mfc)$, respectively.
\end{lemma}

\begin{proof}
Let
\[
g_t=\partial_t D^*(\cdot;U_t),
\qquad |t|\le T.
\]
By \Cref{mainthm:ODE-C1} and \eqref{eq:MF},
\begin{equation}\label{e:c7continuitybounds}
\|U_t-U_s\|_\infty\le \frac{1}{2}\mfc^{-4}|t-s|,
\qquad
\|U_t'\|_\infty\le \frac{1}{2}\mfc^{-4}
\end{equation}
for all $s,t\in[-T,T]$.
By \Cref{lem:TV},
\[
\|g_t\|_\infty\le 9\mfc^{-10},
\qquad |t|\le T.
\]

By \Cref{lemma:wJ} and \eqref{e:Juniform},
\[
g_t=J(U_t)[U_t'].
\]
Therefore, for all $s,t\in[-T,T]$,
\begin{align*}
\|g_t-g_s\|_\infty
&\le
\|(J(U_t)-J(U_s))[U_t']\|_\infty
+
\|J(U_s)[U_t'-U_s']\|_\infty\\
&\le
\|J(U_t)-J(U_s)\|_{\mathrm{op}}\|U_t'\|_\infty
+
\|J(U_s)\|_{\mathrm{op}}\|U_t'-U_s'\|_\infty\\
&\le
104\mfc^{-10}\|U_t-U_s\|_\infty \cdot \frac{1}{2}\mfc^{-4}
+
18\mfc^{-6}\|F(U_t)-F(U_s)\|_\infty\\
&\le
52\mfc^{-14}\|U_t-U_s\|_\infty
+
72\mfc^{-12}\|U_t-U_s\|_\infty\\
&\le
26\mfc^{-18}|t-s|
+
36\mfc^{-16}|t-s|\\
&\le
62\mfc^{-18}|t-s|,
\end{align*}
where we used \Cref{l:Flips} in the fourth line and $\mfc\le 1$ in the last line.

Fix $\Delta>0$ and set
\[
N=\left\lceil \frac{2T}{\Delta}\right\rceil,
\qquad
\mathcal T
=
\left\{
-T+\frac{2Tj}{N}:j=0,\dots,N
\right\}.
\]
Then $|\mathcal T|=N+1\le 2+2T/\Delta$, and every $t\in[-T,T]$ lies within distance at most $\Delta$ of some $t'\in\mathcal T$.

For each fixed $t\in\mathcal T$, the random variables
\[
Z_{i,t}
=
g_t(O_i)-\E[g_t(O_i)\mid\mathcal G_n],
\qquad i=1,\dots,n,
\]
are conditionally independent given $\mathcal G_n$, mean zero, and satisfy
\[
|Z_{i,t}|\le 18\mfc^{-10},
\qquad
\var(Z_{i,t}\mid\mathcal G_n)\le \E[g_t(O_i)^2\mid\mathcal G_n]\le 81\mfc^{-20}.
\]
Hence, Bernstein's inequality yields
\[
\P\left(
\left|(\P_n-\bar P_n^*)[g_t]\right|\ge \eps
\mid
\mathcal G_n
\right)
\le
2\exp\left(
-\frac{n\eps^2}{162\mfc^{-20}+12\mfc^{-10}\eps}
\right).
\]
A union bound over $\mathcal T$ gives
\[
\P\left(
\max_{t\in\mathcal T}\left|(\P_n-\bar P_n^*)[g_t]\right|\ge \eps
\mid
\mathcal G_n
\right)
\le
2|\mathcal T|\exp\left(
-\frac{n\eps^2}{162\mfc^{-20}+12\mfc^{-10}\eps}
\right).
\]

Now let $t\in[-T,T]$ and choose $t'\in\mathcal T$ with $|t-t'|\le \Delta$.
Then
\[
\left|(\P_n-\bar P_n^*)[g_t]\right|
\le
\left|(\P_n-\bar P_n^*)[g_{t'}]\right|
+
2\|g_t-g_{t'}\|_\infty
\le
\left|(\P_n-\bar P_n^*)[g_{t'}]\right|
+
124\mfc^{-18}\Delta.
\]
Therefore,
\[
\sup_{|t|\le T}\left|(\P_n-\bar P_n^*)[g_t]\right|
\le
\max_{t\in\mathcal T}\left|(\P_n-\bar P_n^*)[g_t]\right|
+
124\mfc^{-18}\Delta.
\]

Choose
\[
\Delta_\eps(\mfc,T)=\min\left\{T,\frac{\eps\mfc^{18}}{124}\right\}.
\]
Then \(124\mfc^{-18}\Delta_\eps(\mfc,T)\le \eps\), so
\[
\P\left(
\sup_{|t|\le T}\left|(\P_n-\bar P_n^*)[g_t]\right|\ge 2\eps
\mid
\mathcal G_n
\right)
\le
\P\left(
\max_{t\in\mathcal T}\left|(\P_n-\bar P_n^*)[g_t]\right|\ge \eps
\mid
\mathcal G_n
\right).
\]
Applying the previous grid bound with \(\Delta=\Delta_\eps(\mfc,T)\) completes the proof, with
\[
m_\star(\eps,\mfc,T)
=
2+\frac{2T}{\Delta_\eps(\mfc,T)},
\qquad
C_\star(\eps,\mfc)
=
\frac{\eps^2}{162\mfc^{-20}+12\mfc^{-10}\eps}.
\]
\end{proof}

\begin{lemma}\label{lemma::concavity}
On $\mathcal A_n$, we have
\[
\P\left(
\sup_{|t|\le \mft_2}L_n''(t)>-\frac{c_\init}{4}
\mid
\mathcal G_n
\right)
\le
2\bar m_\star\exp(-C_{1,\star}n),
\]
where
\[
\bar m_\star=m_\star\left(\frac{c_\init}{16},\mfc,\mft_2\right),
\qquad
C_{1,\star}=C_\star\left(\frac{c_\init}{16},\mfc\right).
\]
\end{lemma}

\begin{proof}
Let
\[
g_t=\partial_t D^*(\cdot;U_t).
\]
By \Cref{lem:Ln-2nd-deriv},
\[
L_n''(t)=\P_n[g_t]
=
\bar P_n^*[g_t]
+
(\P_n-\bar P_n^*)[g_t].
\]

We first bound the conditional mean term $\bar P_n^*[g_t]$.
Since $Q_n$ is defined through \eqref{e:qempirical}, we may write
\[
\bar P_n^*[g_t]
=
P_t[g_t]+R_t,
\]
where
\[
P_t=\widetilde P(Q_n,U_t)
\]
and
\[
R_t
=
\E_{Q_n}\left[
\sum_{a,y\in\{0,1\}}
g_t(X,a,y)\left(p^*(a,y\mid X)-p_t(a,y\mid X)\right)
\right].
\]
By \Cref{lem:pop-curv} and \Cref{l:curv},
\[
P_t[g_t]
=
-\E_t \left[(D^*(O;U_t))^2\right]
\le -\frac{c_\init}{2}
\qquad\text{for all }|t|\le \mft_2.
\]

Next, since $\|g_t\|_\infty\le 9\mfc^{-10}$ by \Cref{lem:TV},
\begin{align*}
|R_t|
&\le
18\mfc^{-10}
\E_{Q_n}\left[\tv\left(p^*(\cdot\mid X),p_t(\cdot\mid X)\right)\right]\\
&\le
18\mfc^{-10}
\left(
\E_{Q_n}\left[\tv\left(p^*(\cdot\mid X),p_0(\cdot\mid X)\right)\right]
+
\E_{Q_n}\left[\tv\left(p_0(\cdot\mid X),p_t(\cdot\mid X)\right)\right]
\right).
\end{align*}
By \Cref{as::TV},
\[
\E_{Q_n}\left[\tv\left(p^*(\cdot\mid X),p_0(\cdot\mid X)\right)\right]
\le \frac{\mfc^{10}c_\init}{600}.
\]
Also, for each fixed $x\in\mathcal X$, \Cref{lm::ODE_flow_derivation} and \Cref{l:Lip} yield
\begin{align*}
\sum_{a,y}|p_t(a,y\mid x)-p_0(a,y\mid x)|
&\le
\int_{\min\{0,t\}}^{\max\{0,t\}}
\sum_{a,y}|\partial_s p_s(a,y\mid x)| \, ds\\
&=
\int_{\min\{0,t\}}^{\max\{0,t\}}
\sum_{a,y}p_s(a,y\mid x)|D^*(x,a,y;U_s)| \, ds\\
&\le
18\mfc^{-6}|t|,
\end{align*}
so
\[
\tv\left(p_0(\cdot\mid x),p_t(\cdot\mid x)\right)\le 9\mfc^{-6}|t|.
\]
Averaging with respect to $Q_n$ yields
\[
\E_{Q_n}\left[\tv\left(p_0(\cdot\mid X),p_t(\cdot\mid X)\right)\right]
\le 9\mfc^{-6}|t|.
\]
Therefore, for $|t|\le \mft_2$,
\begin{align*}
|R_t|
&\le
18\mfc^{-10}
\left(
\frac{\mfc^{10}c_\init}{600}
+
9\mfc^{-6}\mft_2
\right)\\
&=
\frac{3c_\init}{100}
+
162\mfc^{-16}\mft_2\\
&=
\frac{3c_\init}{100}
+
\frac{162c_\init\mfc^4}{10^6}
<
\frac{c_\init}{16}.
\end{align*}
Hence
\[
\sup_{|t|\le \mft_2}\bar P_n^*[g_t]
\le
-\frac{c_\init}{2}+\frac{c_\init}{16}
=
-\frac{7c_\init}{16}.
\]

It follows that
\[
\left\{
\sup_{|t|\le \mft_2}L_n''(t)>-\frac{c_\init}{4}
\right\}
\subseteq
\left\{
\sup_{|t|\le \mft_2}\left|(\P_n-\bar P_n^*)[g_t]\right|>\frac{c_\init}{8}
\right\}.
\]
Applying \Cref{lemma::emprical_score} with $T=\mft_2$ and $\eps=c_\init/16$ proves the claim.
\end{proof}

\begin{lemma}\label{lemma:first_order_loss}
Recall that $\mft_2=\delta_\init$. On $\mathcal A_n$, we have
\[
\P\left(
|L_n'(0)|\ge \frac{c_\init\mft_2}{4}
\mid
\mathcal G_n
\right)
\le
2\exp(-C_{2,\star}n),
\]
where $C_{2,\star}>0$ depends only on $(\mfc,c_\init)$.
\end{lemma}

\begin{proof}
By \Cref{lem:Ln-2nd-deriv},
\[
L_n'(0)=\P_n \left[D^*\left(\cdot;\widetilde P(Q_n,U_0)\right)\right]
=
\frac{1}{n}\sum_{i=1}^n \xi_i,
\]
where
\[
\xi_i
=
D^*\left(O_i;\widetilde P(Q_n,U_0)\right).
\]
Conditional on $\mathcal G_n$, the variables $\xi_1,\dots,\xi_n$ are independent and satisfy
\[
\E[\xi_i\mid\mathcal G_n]
=
\sum_{a,y}
D^*\left((X_i,a,y);\widetilde P(Q_n,U_0)\right)p^*(a,y\mid X_i).
\]
Since $Q_n$ is defined through \eqref{e:qempirical}, the average of these conditional means is exactly
\[
\mu_0
=
\E_{Q_n}\left[
\sum_{a,y\in\{0,1\}}
D^*\left((X,a,y);\widetilde P(Q_n,U_0)\right)p^*(a,y\mid X)
\right].
\]
Thus
\[
\mu_0=\frac{1}{n}\sum_{i=1}^n \E[\xi_i\mid\mathcal G_n].
\]

Set
\[
Z_i=\xi_i-\E[\xi_i\mid\mathcal G_n].
\]
Then conditional on $\mathcal G_n$, the variables $Z_1,\dots,Z_n$ are independent and mean zero. By \Cref{l:Lip},
\[
|\xi_i|\le 18\mfc^{-6},
\]
so
\[
|Z_i|
\le
|\xi_i|+\left|\E[\xi_i\mid\mathcal G_n]\right|
\le 36\mfc^{-6},
\qquad
\var(Z_i\mid\mathcal G_n)\le \E[\xi_i^2\mid\mathcal G_n]\le 324\mfc^{-12}.
\]
Bernstein's inequality therefore gives, for every $a>0$,
\[
\P\left(
|L_n'(0)-\mu_0|\ge a
\mid
\mathcal G_n
\right)
\le
2\exp\left(
-\frac{na^2}{648\mfc^{-12}+24\mfc^{-6}a}
\right).
\]
Taking
\[
a=\frac{c_\init\mft_2}{8}
\]
yields
\[
\P\left(
|L_n'(0)-\mu_0|\ge \frac{c_\init\mft_2}{8}
\mid
\mathcal G_n
\right)
\le
2\exp\left(
-\frac{n c_\init^2\mft_2^2}{41472\mfc^{-12}+192\mfc^{-6}c_\init\mft_2}
\right).
\]
By \Cref{ass:mu0-small},
\[
|\mu_0|\le \frac{c_\init\mft_2}{8},
\]
so the triangle inequality gives
\[
\left\{
|L_n'(0)|\ge \frac{c_\init\mft_2}{4}
\right\}
\subseteq
\left\{
|L_n'(0)-\mu_0|\ge \frac{c_\init\mft_2}{8}
\right\}.
\]
Thus the desired bound holds with
\[
C_{2,\star}
=
\frac{c_\init^2\mft_2^2}{41472\mfc^{-12}+192\mfc^{-6}c_\init\mft_2}.
\]
Since $\mft_2=\delta_\init$ depends only on $(\mfc,c_\init)$, so does $C_{2,\star}$.
\end{proof}

Let
\[
\mathcal B_n
=
\left\{
\sup_{|t|\le \mft_2}L_n''(t)\le -\frac{c_\init}{4}
\right\}
\cap
\left\{
|L_n'(0)|\le \frac{c_\init\mft_2}{4}
\right\}.
\]
Also define
\begin{equation}\label{eq:hatt}
t_-
=
-\frac{4[-L_n'(0)]_+}{c_\init},
\qquad
t_+
=
\frac{4[L_n'(0)]_+}{c_\init}.
\end{equation}

\begin{proof}[Proof of \Cref{mainthm:det-bracket-one}]
On $\mathcal A_n$, the path $t\mapsto U_t$ exists by \Cref{mainthm:ODE-C1}, and \Cref{lem:Ln-2nd-deriv} shows that the map $t\mapsto L_n(t)$ belongs to $C^2((-\mft_1,\mft_1))$.

By \Cref{lemma::concavity} and \Cref{lemma:first_order_loss}, on $\mathcal A_n$ we have
\[
\P(\mathcal B_n^{\mathrm c}\mid \mathcal G_n)
\le
2\bar m_\star e^{-C_{1,\star}n}
+
2e^{-C_{2,\star}n},
\]
where
\[
\bar m_\star=m_\star\left(\frac{c_\init}{16},\mfc,\mft_2\right),
\qquad
C_{1,\star}=C_\star\left(\frac{c_\init}{16},\mfc\right).
\]

We now show that
\[
\mathcal A_n\cap \mathcal B_n\subseteq \mathcal E_n.
\]
Fix an outcome in $\mathcal A_n\cap \mathcal B_n$.
Since
\[
\sup_{|t|\le \mft_2}L_n''(t)\le -\frac{c_\init}{4},
\]
the mean-value theorem implies that for any $s<t$ in $[-\mft_2,\mft_2]$,
\[
L_n'(t)-L_n'(s)
=
L_n''(\xi)(t-s)
\le
-\frac{c_\init}{4}(t-s)
\]
for some $\xi\in(s,t)$.
Hence $L_n'$ is continuous and strictly decreasing on $[-\mft_2,\mft_2]$.

If $L_n'(0)=0$, then $\hat t=0$ is a solution of $L_n'(\hat t)=0$. Since $L_n'$ is strictly decreasing on $[-\mft_2,\mft_2]$, this solution is unique.

If $L_n'(0)>0$, then by \eqref{eq:hatt},
\[
t_+=\frac{4L_n'(0)}{c_\init}\in(0,\mft_2],
\]
because $|L_n'(0)|\le c_\init\mft_2/4$ on $\mathcal B_n$.
Applying the mean-value theorem on $[0,t_+]$, we obtain
\[
L_n'(t_+)
=
L_n'(0)+L_n''(\xi)t_+
\le
L_n'(0)-\frac{c_\init}{4}t_+
=
0
\]
for some $\xi\in(0,t_+)$.
Since $L_n'(0)>0$, $L_n'(t_+)\le 0$, and $L_n'$ is continuous and strictly decreasing, there exists a unique
\[
\hat t\in[0,t_+]\subset[-\mft_2,\mft_2]
\]
such that $L_n'(\hat t)=0$.

If $L_n'(0)<0$, then
\[
t_-=\frac{4L_n'(0)}{c_\init}\in[-\mft_2,0),
\]
again because $|L_n'(0)|\le c_\init\mft_2/4$ on $\mathcal B_n$.
Applying the mean-value theorem on $[t_-,0]$, we get
\[
L_n'(0)-L_n'(t_-)
=
L_n''(\xi)(0-t_-)
\le
-\frac{c_\init}{4}(0-t_-)
\]
for some $\xi\in(t_-,0)$, hence
\[
L_n'(t_-)
\ge
L_n'(0)+\frac{c_\init}{4}(0-t_-)
=
0.
\]
Since $L_n'(t_-)\ge 0$, $L_n'(0)<0$, and $L_n'$ is continuous and strictly decreasing, there exists a unique
\[
\hat t\in[t_-,0]\subset[-\mft_2,\mft_2]
\]
such that $L_n'(\hat t)=0$.

Thus in all cases there exists a unique $\hat t\in[-\mft_2,\mft_2]$ with $L_n'(\hat t)=0$, so indeed
\[
\mathcal A_n\cap \mathcal B_n\subseteq \mathcal E_n.
\]
Therefore,
\[
\mathbf{1}_{\mathcal A_n}\P(\mathcal E_n^{\mathrm c}\mid \mathcal G_n)
\le
\mathbf{1}_{\mathcal A_n}\P(\mathcal B_n^{\mathrm c}\mid \mathcal G_n)
\le
2\bar m_\star e^{-C_{1,\star}n}+2e^{-C_{2,\star}n}.
\]

Finally, define
\[
m_0
=
\min\left\{
C_{1,\star},
C_{2,\star},
1/(2\bar m_\star+2)
\right\}.
\]
Then $m_0>0$ depends only on $(\mfc,c_\init)$, and
\[
2\bar m_\star e^{-C_{1,\star}n}+2e^{-C_{2,\star}n}
\le
(2\bar m_\star+2)e^{-m_0 n}
\le
\frac{1}{m_0}e^{-m_0 n}.
\]
This proves
\[
\mathbf 1_{\mathcal A_n}\P(\mathcal E_n^{\mathrm c}\mid \mathcal G_n)
\le \frac{1}{m_0}e^{-m_0 n},
\]
as desired.
\end{proof}

\section{Proof of \texorpdfstring{\Cref{mainthm:AN_cross}}{Third Main Result}}
\label{append::AN}

\subsection{Remainder Estimates}

Throughout this subsection we work in the foldwise cross-fitted setting used in the proof of \Cref{mainthm:AN_cross}. We suppress the fold index and write
\[
Q_n=\frac{1}{n}\sum_{i=1}^n \delta_{X_i},
\qquad
Q^*=P_X^*.
\]
Here \(Q_n\) is the empirical law of the evaluation-fold covariates, while \(U_0\) is constructed from the opposite fold. The \(\sigma\)-field \(\mathcal G_n\) contains the opposite-fold data together with the evaluation-fold covariates. Consequently, conditional on \(\mathcal G_n\), the empirical design law \(Q_n\), the initial nuisance estimate \(U_0\), and, on \(\mathcal A_n\), the entire path \(t\mapsto U_t\) are fixed; only the treatment and outcome variables in the evaluation fold remain random and conditionally independent.

The arguments below therefore have two components. First, we prove conditional fixed-design bounds for empirical fluctuations around
\[
\bar P_n^*[f]
=
\frac{1}{n}\sum_{i=1}^n \sum_{a,y\in\{0,1\}} f(X_i,a,y)\,p^*(a,y\mid X_i),
\]
which is the conditional expectation of \(\P_n[f]\) given \(\mathcal G_n\). In these arguments, the realized covariates and the path \(t\mapsto U_t\) are treated as deterministic. Second, we prove separate unconditional bounds that quantify the effect of replacing the realized design law \(Q_n\) by the population design law \(Q^*=P_X^*\).

All lemmas in this subsection are stated for this setting only. In the proof of \Cref{mainthm:AN_cross}, they are applied with
\[
n=m,
\qquad
Q_n=\hat Q^{(k)},
\qquad
U_0=\hat U_0^{(1-k)},
\qquad
\mathcal G_n=\mathcal G_m^{(k)}.
\]

From this point onward, all pathwise remainder objects are interpreted on the good event \(\mathcal A_n\). On \(\mathcal A_n^c\), we adopt the bookkeeping convention
\[
U_t \equiv U_0
\qquad\text{for all }t,
\]
and, whenever the quantities
\[
r_t,\ \Omega_t^n,\ \Omega_t^*,\ N_t^n,\ N_t^*,\ \psi_t^n,\ \psi_t^*,\ H_t,\ G_t
\]
are introduced below, we extend them to \(\mathcal A_n^c\) by setting
\[
r_t\equiv 0,\qquad
\Omega_t^n=\Omega_t^*=1,\qquad
N_t^n=N_t^*=0,\qquad
\psi_t^n=\psi_t^*=0,\qquad
H_t\equiv 0,\qquad
G_t\equiv 0.
\]
These extensions serve only to make the corresponding random objects globally measurable. Unless explicitly stated otherwise, all equalities and inequalities in the remainder proofs are asserted on \(\mathcal A_n\). When these lemmas are applied in the proof of \Cref{mainthm:AN_cross}, condition 2 implies \(\P(\mathcal A_n^c)\to 0\), so the above conventions do not affect any \(o_p(\cdot)\) conclusion.

\begin{lemma}\label{lem:score-small-root}
We have
\[
\mathbf 1_{\mathcal A_n} L_n'(0)=o_p(n^{-1/4}),
\qquad
\mathbf 1_{\mathcal A_n} \hat t=o_p(n^{-1/4}).
\]
\end{lemma}

\begin{proof}
We first prove that
\[
\mathbf 1_{\mathcal A_n}L_n'(0)=o_p(n^{-1/4}).
\]

Define, for a fixed probability measure $Q$ on $\mathcal X$,
\[
S_Q(U)(o)=D^*(o;\widetilde P(Q,U)).
\]
Then
\[
\mathbf 1_{\mathcal A_n}L_n'(0)
=
\mathbf 1_{\mathcal A_n}\P_n[S_{Q_n}(U_0)]
=
\mathbf 1_{\mathcal A_n}\P_n\!\left[S_{Q_n}(U_0)-S_{Q_n}(U^*)\right]
+
\mathbf 1_{\mathcal A_n}\P_n\!\left[S_{Q_n}(U^*)\right].
\]

We first treat the second term. On $\mathcal A_n$, the local bracketing conditions imply
$U^*\in\mathcal B_\eta(U_0)$, so \Cref{as::positivity} applies to both $U_0$ and $U^*$.
Hence $S_{Q_n}(U^*)$ is uniformly bounded by a constant depending only on $\mfc$.

Set
\[
\xi_i=S_{Q_n}(U^*)(O_i),\qquad i=1,\dots,n.
\]
Conditional on $\mathcal G_n$, the quantities $Q_n$ and $X_1,\dots,X_n$ are fixed, while
$(A_i,Y_i)$ are conditionally independent with conditional law
\[
\P\big((A_i,Y_i)=(a,y)\mid \mathcal G_n\big)=p^*(a,y\mid X_i).
\]
Moreover, because $U^*$ is the true nuisance triple, the restricted EIF has conditional
mean zero given $X$ for every fixed marginal law. Therefore, for each $i$,
\[
\E\!\left[\xi_i\mid \mathcal G_n\right]
=
\sum_{a,y\in\{0,1\}}
S_{Q_n}(U^*)((X_i,a,y))\,p^*(a,y\mid X_i)=0.
\]
Thus, conditional on $\mathcal G_n$, the variables $\xi_1,\dots,\xi_n$ are independent,
mean zero, and uniformly bounded by a constant depending only on $\mfc$. In particular,
\[
\E\!\left[\left(\sqrt n\,\P_n[S_{Q_n}(U^*)]\right)^2\Bigm|\mathcal G_n\right]
=
\E\!\left[\left(\frac{1}{\sqrt n}\sum_{i=1}^n \xi_i\right)^2\Bigm|\mathcal G_n\right]
=
\frac{1}{n}\sum_{i=1}^n \E\!\left[\xi_i^2\mid \mathcal G_n\right]
\le C(\mfc).
\]
Hence, by the conditional Chebyshev inequality,
for every $\eps>0$,
\[
\P\!\left(
\sqrt n\,\left|\P_n[S_{Q_n}(U^*)]\right|>\eps
\Bigm|\mathcal G_n
\right)\mathbf 1_{\mathcal A_n}
\le \frac{C(\mfc)}{\eps^2}.
\]
Therefore
\[
\mathbf 1_{\mathcal A_n}\sqrt n\,\P_n[S_{Q_n}(U^*)]=O_p(1),
\qquad\text{and hence}\qquad
\mathbf 1_{\mathcal A_n}\P_n[S_{Q_n}(U^*)]=O_p(n^{-1/2}).
\]

Next, consider the first term. Set
\[
r_0(x)
=
|q_{1,0}(x)-q_1^*(x)|
+
|q_{0,0}(x)-q_0^*(x)|
+
|e_0(x)-e^*(x)|.
\]
Define
\[
\Omega_n(U)=Q_n[\lambda(e)],
\qquad
\psi_n(U)=\psi(Q_n,U),
\]
and
\[
c_{1,n}(U)(x,a)=\frac{\lambda(e(x))}{\Omega_n(U)}\frac{a}{e(x)},
\qquad
c_{0,n}(U)(x,a)=\frac{\lambda(e(x))}{\Omega_n(U)}\frac{1-a}{1-e(x)},
\]
\[
b_n(U)(x,a)=\frac{\dot\lambda(e(x))}{\Omega_n(U)}(a-e(x)).
\]
Then
\[
S_{Q_n}(U)(x,a,y)
=
c_{1,n}(U)(x,a)(y-q_1(x))
-
c_{0,n}(U)(x,a)(y-q_0(x))
+
b_n(U)(x,a)\bigl((q_1-q_0)(x)-\psi_n(U)\bigr).
\]

On $\mathcal A_n$, positivity gives
\[
\Omega_n(U_0),\Omega_n(U^*)\in[\Lambda_{\min},\Lambda_{\max}],
\qquad
e_0(x),e^*(x)\in[\eta,1-\eta].
\]
Moreover,
\[
|\Omega_n(U_0)-\Omega_n(U^*)|
\le
\dot\Lambda_{\max}Q_n[r_0].
\]
Let
\[
\tau_0=q_{1,0}-q_{0,0},
\qquad
\tau^*=q_1^*-q_0^*,
\qquad
N_n(U_0)=Q_n[\lambda(e_0)\tau_0],
\qquad
N_n(U^*)=Q_n[\lambda(e^*)\tau^*].
\]
Then
\[
\psi_n(U_0)=\frac{N_n(U_0)}{\Omega_n(U_0)},
\qquad
\psi_n(U^*)=\frac{N_n(U^*)}{\Omega_n(U^*)},
\]
so
\[
\psi_n(U_0)-\psi_n(U^*)
=
\frac{N_n(U_0)-N_n(U^*)}{\Omega_n(U_0)}
+
N_n(U^*)\left(
\frac{1}{\Omega_n(U_0)}-\frac{1}{\Omega_n(U^*)}
\right).
\]
Since
\[
|\tau_0-\tau^*|
\le
|q_{1,0}-q_1^*|+|q_{0,0}-q_0^*|
\le r_0,
\]
the mean value theorem yields
\begin{align*}
|N_n(U_0)-N_n(U^*)|
&=
\left|
Q_n[\lambda(e_0)\tau_0-\lambda(e^*)\tau^*]
\right| \\
&\le
Q_n\left[|\lambda(e_0)-\lambda(e^*)||\tau_0|\right]
+
Q_n\left[\lambda(e^*)|\tau_0-\tau^*|\right] \\
&\le
\dot\Lambda_{\max}Q_n[|e_0-e^*|]
+
\Lambda_{\max}Q_n[|\tau_0-\tau^*|] \\
&\le
(\dot\Lambda_{\max}+\Lambda_{\max})Q_n[r_0].
\end{align*}
Also,
\[
|N_n(U^*)|\le \Lambda_{\max},
\qquad
\Omega_n(U_0),\Omega_n(U^*)\in[\Lambda_{\min},\Lambda_{\max}].
\]
Therefore, on $\mathcal A_n$,
\[
|\psi_n(U_0)-\psi_n(U^*)|
\le
\frac{|N_n(U_0)-N_n(U^*)|}{\Lambda_{\min}}
+
\frac{\Lambda_{\max}}{\Lambda_{\min}^2}
|\Omega_n(U_0)-\Omega_n(U^*)|
\le
C(\mfc)\,Q_n[r_0].
\]
Using the mean value theorem and the previous two displays, we obtain on $\mathcal A_n$
\[
|c_{1,n}(U_0)-c_{1,n}(U^*)|
+
|c_{0,n}(U_0)-c_{0,n}(U^*)|
+
|b_n(U_0)-b_n(U^*)|
\le
C(\mfc)\bigl(r_0(x)+Q_n[r_0]\bigr).
\]
Also,
\[
|q_{1,0}(x)-q_1^*(x)|+|q_{0,0}(x)-q_0^*(x)|\le r_0(x),
\qquad
|\tau_0(x)-\tau^*(x)|\le r_0(x).
\]
Since $|y-q_a^*(x)|\le 1$, $|a-e_0(x)|\le 1$, and
\[
|\tau^*(x)-\psi_n(U^*)|
\le 1+\frac{\Lambda_{\max}}{\Lambda_{\min}},
\]
it follows that for every $(x,a,y)$,
\[
\mathbf 1_{\mathcal A_n}
|S_{Q_n}(U_0)((x,a,y))-S_{Q_n}(U^*)((x,a,y))|
\le
C(\mfc)\bigl(r_0(x)+Q_n[r_0]\bigr).
\]
Therefore, conditional on $\mathcal G_n$,

\begin{align*}
\E\!\left[
\mathbf 1_{\mathcal A_n}
\P_n\!\left[\left(S_{Q_n}(U_0)-S_{Q_n}(U^*)\right)^2\right]
\Bigm| \mathcal G_n
\right]
&\le
C(\mfc)\bigl(Q_n[r_0^2]+Q_n[r_0]^2\bigr)\\
&\le
C(\mfc)\,Q_n[r_0^2].
\end{align*}

Because $U_0$ is constructed on the opposite fold, $r_0$ is independent of the
evaluation-fold covariates $X_1,\dots,X_n$. Therefore, conditional on $U_0$,
the variables $r_0(X_1)^2,\dots,r_0(X_n)^2$ are independent and identically
distributed with mean $P_X^*[r_0^2]$. Hence
\[
\E\!\left[Q_n[r_0^2]\mid U_0\right]=P_X^*[r_0^2].
\]
Also, since $0\le r_0\le 3$,
\[
\var\!\left(Q_n[r_0^2]\mid U_0\right)
=
\frac{1}{n}\var\!\left(r_0(X)^2\mid U_0\right)
\le \frac{1}{n}\E\!\left[r_0(X)^4\mid U_0\right]
\le \frac{9}{n}P_X^*[r_0^2].
\]

By \Cref{as::initial_convergence},
\[
P_X^*[r_0^2]
\lesssim
\|q_{1,0}-q_1^*\|_{2,*}^2
+
\|q_{0,0}-q_0^*\|_{2,*}^2
+
\|e_0-e^*\|_{2,*}^2
=
o_p(n^{-1/2}).
\]
We now show that this implies
\[
Q_n[r_0^2]=o_p(n^{-1/2}).
\]

Let $a_n=n^{-1/2}$. Fix $\eps,\eta>0$, and set
\[
\mu_n=P_X^*[r_0^2].
\]
Then
\[
\P\!\left(Q_n[r_0^2]>\eps a_n\right)
\le
\P\!\left(\mu_n>\frac{\eps}{2}a_n\right)
+
\P\!\left(Q_n[r_0^2]>\eps a_n,\ \mu_n\le \frac{\eps}{2}a_n\right).
\]
On the event $\{\mu_n\le \eps a_n/2\}$, conditional Chebyshev gives
\begin{align*}
\P\!\left(Q_n[r_0^2]>\eps a_n\mid U_0\right)
&\le
\P\!\left(|Q_n[r_0^2]-\mu_n|>\frac{\eps}{2}a_n\mid U_0\right) \\
&\le
\frac{4\,\var(Q_n[r_0^2]\mid U_0)}{\eps^2 a_n^2} \\
&\le
\frac{36}{\eps^2}\mu_n.
\end{align*}
Therefore
\begin{align*}
\P\!\left(Q_n[r_0^2]>\eps a_n,\ \mu_n\le \frac{\eps}{2}a_n\right)
&\le
\E\!\left[
\mathbf 1\!\left\{\mu_n\le \frac{\eps}{2}a_n\right\}
\P\!\left(Q_n[r_0^2]>\eps a_n\mid U_0\right)
\right] \\
&\le
\frac{36}{\eps^2}\E\!\left[
\mathbf 1\!\left\{\mu_n\le \frac{\eps}{2}a_n\right\}\mu_n
\right] \\
&\le
\frac{18}{\eps}a_n.
\end{align*}
Since $\mu_n=o_p(a_n)$, the first term also tends to zero, so
\[
Q_n[r_0^2]=o_p(n^{-1/2}).
\]

Next define
\[
Z_n
=
\mathbf 1_{\mathcal A_n}
\P_n\!\left[\left(S_{Q_n}(U_0)-S_{Q_n}(U^*)\right)^2\right].
\]
From the conditional second-moment bound proved above,
\[
\E[Z_n\mid \mathcal G_n]
\le
C(\mfc)\,Q_n[r_0^2].
\]
Since $Q_n[r_0^2]=o_p(n^{-1/2})$, we have
\[
\E[Z_n\mid \mathcal G_n]=o_p(n^{-1/2}).
\]
To conclude that $Z_n=o_p(n^{-1/2})$, fix $\eps,\eta>0$ and write
\begin{align*}
\P\!\left(Z_n>\eps n^{-1/2}\right)
&\le
\P\!\left(\E[Z_n\mid \mathcal G_n]>\eta n^{-1/2}\right) \\
&\quad+
\P\!\left(
Z_n>\eps n^{-1/2},\ 
\E[Z_n\mid \mathcal G_n]\le \eta n^{-1/2}
\right).
\end{align*}
On the event $\{\E[Z_n\mid \mathcal G_n]\le \eta n^{-1/2}\}$,
conditional Markov gives
\[
\P\!\left(
Z_n>\eps n^{-1/2}\mid \mathcal G_n
\right)
\le
\frac{\E[Z_n\mid \mathcal G_n]}{\eps n^{-1/2}}
\le
\frac{\eta}{\eps}.
\]
Hence
\[
\P\!\left(
Z_n>\eps n^{-1/2},\ 
\E[Z_n\mid \mathcal G_n]\le \eta n^{-1/2}
\right)
\le
\frac{\eta}{\eps}.
\]
Since $\E[Z_n\mid \mathcal G_n]=o_p(n^{-1/2})$, the first term tends to zero,
and then letting $\eta\downarrow 0$ yields
\[
\mathbf 1_{\mathcal A_n}
\P_n\!\left[\left(S_{Q_n}(U_0)-S_{Q_n}(U^*)\right)^2\right]
=o_p(n^{-1/2}).
\]

By Cauchy--Schwarz,
\[
\mathbf 1_{\mathcal A_n}
\left|
\P_n\!\left[S_{Q_n}(U_0)-S_{Q_n}(U^*)\right]
\right|
\le
\left(
\mathbf 1_{\mathcal A_n}
\P_n\!\left[\left(S_{Q_n}(U_0)-S_{Q_n}(U^*)\right)^2\right]
\right)^{1/2}
=
o_p(n^{-1/4}).
\]
Since also
\[
\mathbf 1_{\mathcal A_n}\P_n[S_{Q_n}(U^*)]
=
O_p(n^{-1/2})
=
o_p(n^{-1/4}),
\]
we conclude that
\[
\mathbf 1_{\mathcal A_n}L_n'(0)=o_p(n^{-1/4}).
\]

For the second claim, note that $\hat t=0$ on $\mathcal E_n^{\mathrm c}$. Moreover, on
$\mathcal A_n\cap \mathcal B_n$, the proof of \Cref{mainthm:det-bracket-one}
shows that the unique root $\hat t$ lies in $[t_-,t_+]$, where $t_-,t_+$ are
defined in \eqref{eq:hatt}. Hence, on $\mathcal A_n\cap \mathcal B_n$,
\[
|\hat t|
\le \max\{|t_-|,|t_+|\}
= \frac{4}{c_\init}|L_n'(0)|.
\]
Therefore
\[
\mathbf 1_{\mathcal A_n}|\hat t|
\le
\mathbf 1_{\mathcal A_n\cap\mathcal B_n}|\hat t|
+
\mft_2\,\mathbf 1_{\mathcal A_n\cap\mathcal B_n^{\mathrm c}}
\le
\frac{4}{c_\init}\mathbf 1_{\mathcal A_n}|L_n'(0)|
+
\mft_2\,\mathbf 1_{\mathcal A_n\cap\mathcal B_n^{\mathrm c}}.
\]
By the first part of the lemma,
\[
\mathbf 1_{\mathcal A_n}L_n'(0)=o_p(n^{-1/4}).
\]
Also, by \Cref{lemma::concavity,lemma:first_order_loss},
\[
\mathbf 1_{\mathcal A_n}\P(\mathcal B_n^{\mathrm c}\mid \mathcal G_n)
\le
2m_\star e^{-C_{1,\star}n}+2e^{-C_{2,\star}n},
\]
so
\[
\P(\mathcal A_n\cap\mathcal B_n^{\mathrm c})\to 0.
\]
Hence
\[
\mft_2\,\mathbf 1_{\mathcal A_n\cap\mathcal B_n^{\mathrm c}}=o_p(n^{-1/4}),
\]
and therefore
\[
\mathbf 1_{\mathcal A_n}\hat t=o_p(n^{-1/4}).
\]
This proves the lemma.
\end{proof}

For $|t|\le \mft_2$, define
\[
r_t(x)
=
|q_{1,t}(x)-q_1^*(x)|
+
|q_{0,t}(x)-q_0^*(x)|
+
|e_t(x)-e^*(x)|.
\]
Also define
\[
\Omega_t^n=\Omega(Q_n,U_t),
\qquad
\Omega_t^*=\Omega(Q^*,U_t),
\qquad
N_t^n=Q_n[\lambda(e_t)\tau_t],
\qquad
N_t^*=Q^*[\lambda(e_t)\tau_t],
\]
and
\[
\psi_t^n=\psi(Q_n,U_t),
\qquad
\psi_t^*=\psi(Q^*,U_t).
\]
Finally define
\[
H_t(\cdot)
=
D_{\full}^*(\cdot;\widetilde P(Q_n,U_t))
-
D_{\full}^*(\cdot;P^*),
\]
and
\[
G_t(\cdot)
=
D_{\full}^*(\cdot;\widetilde P(Q_n,U_t))^2
-
D_{\full}^*(\cdot;P^*)^2.
\]
\begin{lemma}\label{lem:full-eif-bdd}
Assume \Cref{as::positivity} and fix a probability measure $Q$ on $\mathcal X$.
There exists a constant $C_{\full}(\mfc)>0$ such that
\[
\sup_{U\in\mathcal B_{2\eta}(U_0)}
\|D_{\full}^*(\cdot;\widetilde P(Q,U))\|_\infty
\le C_{\full}(\mfc).
\]
Further, for all $U,V\in\mathcal B_{2\eta}(U_0)$,
\[
\|D_{\full}^*(\cdot;\widetilde P(Q,U))^2
-
D_{\full}^*(\cdot;\widetilde P(Q,V))^2\|_\infty
\le
2C_{\full}(\mfc)\,
\|D_{\full}^*(\cdot;\widetilde P(Q,U))
-
D_{\full}^*(\cdot;\widetilde P(Q,V))\|_\infty.
\]
If, in addition, $U^*\in\mathcal B_\eta(U_0)$, then
\[
\|D_{\full}^*(\cdot;P^*)\|_\infty\le C_{\full}(\mfc).
\]
\end{lemma}

\begin{proof}
By \Cref{l:Lip},
\[
\sup_{U\in\mathcal B_{2\eta}(U_0)}
\|D^*(\cdot;\widetilde P(Q,U))\|_\infty
\le 18\mfc^{-6}.
\]
Moreover, for $U=(q_1,q_0,e)\in\mathcal B_{2\eta}(U_0)$ and
$o=(x,a,y)\in\mathcal O$,
\[
D_X^*(o;\widetilde P(Q,U))
=
\frac{\lambda(e(x))}{\Omega(Q,U)}\bigl(\tau(U)(x)-\psi(Q,U)\bigr).
\]
By \Cref{as::positivity},
\[
\lambda(e(x))\le \Lambda_{\max},
\qquad
\Omega(Q,U)\ge \Lambda_{\min},
\qquad
|\tau(U)(x)|\le 1,
\]
and
\[
|\psi(Q,U)|
=
\left|
\frac{Q[\lambda(e)\tau(U)]}{\Omega(Q,U)}
\right|
\le
\frac{\Lambda_{\max}}{\Lambda_{\min}}.
\]
Hence
\[
\|D_X^*(\cdot;\widetilde P(Q,U))\|_\infty
\le
\frac{\Lambda_{\max}}{\Lambda_{\min}}
\left(1+\frac{\Lambda_{\max}}{\Lambda_{\min}}\right)
\le 2\mfc^{-4}.
\]
Therefore
\[
\sup_{U\in\mathcal B_{2\eta}(U_0)}
\|D_{\full}^*(\cdot;\widetilde P(Q,U))\|_\infty
\le
18\mfc^{-6}+2\mfc^{-4}
\le 20\mfc^{-6}.
\]
Thus one may take \(C_{\full}(\mfc)=20\mfc^{-6}\).

For the second claim, set
\[
A_U=D_{\full}^*(\cdot;\widetilde P(Q,U)),
\qquad
A_V=D_{\full}^*(\cdot;\widetilde P(Q,V)).
\]
Then
\[
A_U^2-A_V^2=(A_U+A_V)(A_U-A_V),
\]
so
\[
\|A_U^2-A_V^2\|_\infty
\le
(\|A_U\|_\infty+\|A_V\|_\infty)\|A_U-A_V\|_\infty
\le
2C_{\full}(\mfc)\|A_U-A_V\|_\infty.
\]

Finally, if \(U^*\in\mathcal B_\eta(U_0)\), then \(U^*\in\mathcal B_{2\eta}(U_0)\) and
\[
P^*=\widetilde P(Q^*,U^*),
\qquad Q^*=P_X^*,
\]
so the first bound applied with \(Q=Q^*\) and \(U=U^*\) gives
\[
\|D_{\full}^*(\cdot;P^*)\|_\infty\le C_{\full}(\mfc).
\]
\end{proof}

\begin{lemma}\label{lem:path-controls}
Let $(\delta_n)$ be any deterministic sequence with $\delta_n=o(n^{-1/4})$.
Then
\[
\mathbf 1_{\mathcal A_n}\sup_{|t|\le \delta_n}Q_n[r_t^2]=o_p(n^{-1/2}),
\]
and
\[
\mathbf 1_{\mathcal A_n}\sup_{|t|\le \delta_n}
\left(
|\Omega_t^n-\Omega_t^*|
+
|N_t^n-N_t^*|
+
|\psi_t^n-\psi_t^*|
\right)
=
o_p(n^{-1/4}).
\]
Moreover, there exists a constant $C_{\full,\infty}(\mfc)>0$ such that on $\mathcal A_n$,
for all $|s|,|t|\le \mft_2$,
\[
\|H_t-H_s\|_\infty+\|G_t-G_s\|_\infty
\le
C_{\full,\infty}(\mfc)|t-s|.
\]
Finally,
\[
\mathbf 1_{\mathcal A_n}\sup_{|t|\le \delta_n}\bar P_n^*[H_t^2]=o_p(n^{-1/2}),
\qquad
\mathbf 1_{\mathcal A_n}\sup_{|t|\le \delta_n}\bar P_n^*[|G_t|]=o_p(1).
\]
\end{lemma}

\begin{proof}
We begin by establishing deterministic inequalities that hold on $\mathcal A_n$ for the
realized empirical design law $Q_n$ and the corresponding path $t\mapsto U_t$. These pathwise inequalities will then be
combined with unconditional stochastic bounds established earlier, in particular
the estimate
\[
Q_n[r_0^2]=o_p(n^{-1/2}),
\]
proved in \Cref{lem:score-small-root}.
All equalities and inequalities below are asserted on the event $\mathcal A_n$.

Since $|t|\le \mft_2<\mft_1$, \Cref{mainthm:ODE-C1} ensures that
$U_s\in\mathcal B_\eta(U_0)$ for every $s$ between $0$ and $t$. Therefore,
by \eqref{eq:MF},
\[
\|F(Q_n,U_s)\|_\infty\le \frac{1}{2\mfc^4}
\qquad\text{for all such }s.
\]
Using
\[
U_t-U_0=\int_0^t F(Q_n,U_s)\,ds,
\]
we obtain
\[
\|U_t-U_0\|_\infty
\le
\int_{\min\{0,t\}}^{\max\{0,t\}}\|F(Q_n,U_s)\|_\infty\,ds
\le
\frac{|t|}{2\mfc^4}
\qquad\text{for all }|t|\le \mft_2.
\]
Hence
\[
r_t(x)\le r_0(x)+C_F(\mfc)|t|,
\]
and therefore
\[
Q_n[r_t^2]
\le 2Q_n[r_0^2]+2C_F(\mfc)^2 t^2.
\]
Combining this deterministic inequality with the unconditional bound
$Q_n[r_0^2]=o_p(n^{-1/2})$ from \Cref{lem:score-small-root}, and using
$\delta_n=o(n^{-1/4})$, gives
\[
\sup_{|t|\le \delta_n}Q_n[r_t^2]=o_p(n^{-1/2}).
\]

We now turn to the design-to-population part of the argument, namely the terms
that measure the effect of replacing the realized design law $Q_n$ by $Q^*$. Set
\[
a_t(x)=\lambda(e_t(x)),
\qquad
b_t(x)=\lambda(e_t(x))\tau_t(x).
\]
Because $t\mapsto U_t$ is Lipschitz in supremum norm and $\lambda$ is $C^1$ on
$[\eta,1-\eta]$, there exists a constant $C_{ab}(\mfc)>0$ such that
\[
\|a_t-a_s\|_\infty+\|b_t-b_s\|_\infty
\le
C_{ab}(\mfc)|t-s|
\qquad\text{for all }|s|,|t|\le \mft_2.
\]
Thus
\[
|\Omega_t^n-\Omega_t^*|
=
|(Q_n-Q^*)[a_t]|
\le
|(Q_n-Q^*)[a_0]|+2C_{ab}|t|,
\]
and similarly
\[
|N_t^n-N_t^*|
\le
|(Q_n-Q^*)[b_0]|+2C_{ab}|t|.
\]

Define
\[
\mathcal A_n^{\mathrm{pos}}
=
\left\{
e_0(x)\in[\eta,1-\eta]\text{ for all }x\in[0,1]^d
\right\}.
\]
Since \(\mathcal A_n\) means that the pair \((Q_n,U_0)\) satisfies the local
bracketing conditions with constants \((\mfc,c_{\mathrm{init}})\), \Cref{as::positivity} 
holds on \(\mathcal A_n\). Therefore
\[
\mathcal A_n\subseteq \mathcal A_n^{\mathrm{pos}}.
\]
Moreover, \(\mathcal A_n^{\mathrm{pos}}\) is \(U_0\)-measurable. On
\(\mathcal A_n^{\mathrm{pos}}\),
\[
\|a_0\|_\infty+\|b_0\|_\infty\le C(\mfc).
\]
Because \(U_0\) is constructed on the opposite fold, conditional on \(U_0\) the
evaluation-fold covariates \(X_1,\dots,X_n\) are i.i.d. with common law \(Q^*\).
Hence
\[
\E\!\left[\mathbf 1_{\mathcal A_n^{\mathrm{pos}}}\bigl((Q_n-Q^*)[a_0]\bigr)^2 \,\middle|\, U_0\right]
\le
\mathbf 1_{\mathcal A_n^{\mathrm{pos}}}\frac{C(\mfc)}{n},
\]
and similarly for \(b_0\). Taking expectations and using
\(\mathbf 1_{\mathcal A_n}\le \mathbf 1_{\mathcal A_n^{\mathrm{pos}}}\), we get
\[
\E\!\left[\mathbf 1_{\mathcal A_n}\bigl((Q_n-Q^*)[a_0]\bigr)^2\right]
\le \frac{C(\mfc)}{n},
\qquad
\E\!\left[\mathbf 1_{\mathcal A_n}\bigl((Q_n-Q^*)[b_0]\bigr)^2\right]
\le \frac{C(\mfc)}{n}.
\]
Therefore, by Chebyshev,
\[
\mathbf 1_{\mathcal A_n}(Q_n-Q^*)[a_0]=O_p(n^{-1/2}),
\qquad
\mathbf 1_{\mathcal A_n}(Q_n-Q^*)[b_0]=O_p(n^{-1/2}).
\]
Hence
\[
\mathbf 1_{\mathcal A_n}\sup_{|t|\le \delta_n}
\left(
|\Omega_t^n-\Omega_t^*|+|N_t^n-N_t^*|
\right)
=
O_p(n^{-1/2})+O(\delta_n)
=
o_p(n^{-1/4}).
\]
Since $\Omega_t^n,\Omega_t^*\in[\Lambda_{\min},\Lambda_{\max}]$ and
$|N_t^n|,|N_t^*|\le \Lambda_{\max}$, it follows that
\[
\sup_{|t|\le \delta_n}|\psi_t^n-\psi_t^*|=o_p(n^{-1/4}).
\]

Now inspect the explicit formula for the full EIF:
\begin{align*}
D_{\full}^*(o;\widetilde P(Q_n,U_t))
&=
\frac{\lambda(e_t(x))}{\Omega_t^n}
\left[
\frac{a}{e_t(x)}(y-q_{1,t}(x))
-
\frac{1-a}{1-e_t(x)}(y-q_{0,t}(x))
+
\tau_t(x)-\psi_t^n
\right] \\
&\quad+
\frac{\dot\lambda(e_t(x))}{\Omega_t^n}
(\tau_t(x)-\psi_t^n)(a-e_t(x)).
\end{align*}
Every coefficient in this display is Lipschitz in
\[
(q_{1,t}(x),q_{0,t}(x),e_t(x),\Omega_t^n,\psi_t^n)
\]
on the positivity region. Since
\[
\|U_t-U_s\|_\infty+|\Omega_t^n-\Omega_s^n|+|\psi_t^n-\psi_s^n|
\le C(\mfc)|t-s|,
\]
it follows that there exists a constant $C_{\full,\infty}(\mfc)>0$ such that
\[
\|D_{\full}^*(\cdot;\widetilde P(Q_n,U_t))
-
D_{\full}^*(\cdot;\widetilde P(Q_n,U_s))\|_\infty
\le
C_{\full,\infty}(\mfc)|t-s|.
\]
By \Cref{lem:full-eif-bdd}, applied with \(Q=Q_n\), we have on \(\mathcal A_n\)
\[
\sup_{|u|\le \mft_2}
\|D_{\full}^*(\cdot;\widetilde P(Q_n,U_u))\|_\infty
\le C_{\full}(\mfc),
\]
because \(U_u\in\mathcal B_\eta(U_0)\subset\mathcal B_{2\eta}(U_0)\) for \(|u|\le \mft_2\).
Therefore
\begin{align*}
\|G_t-G_s\|_\infty
&=
\left\|
D_{\full}^*(\cdot;\widetilde P(Q_n,U_t))^2
-
D_{\full}^*(\cdot;\widetilde P(Q_n,U_s))^2
\right\|_\infty\\
&\le
2C_{\full}(\mfc)
\left\|
D_{\full}^*(\cdot;\widetilde P(Q_n,U_t))
-
D_{\full}^*(\cdot;\widetilde P(Q_n,U_s))
\right\|_\infty\\
&\le
2C_{\full}(\mfc)C_{\full,\infty}(\mfc)|t-s|.
\end{align*}
After enlarging \(C_{\full,\infty}(\mfc)\) if necessary, this proves
\[
\|H_t-H_s\|_\infty+\|G_t-G_s\|_\infty
\le
C_{\full,\infty}(\mfc)|t-s|.
\]

It remains to bound $\bar P_n^*[H_t^2]$ and $\bar P_n^*[|G_t|]$.

First note that
\[
Q^*[r_t^2]
\le
2Q^*[r_0^2]+2C_F(\mfc)^2 t^2
\qquad\text{for all }|t|\le \mft_2.
\]
By \Cref{as::initial_convergence},
\[
Q^*[r_0^2]
\lesssim
\|q_{1,0}-q_1^*\|_{2,*}^2
+
\|q_{0,0}-q_0^*\|_{2,*}^2
+
\|e_0-e^*\|_{2,*}^2
=
o_p(n^{-1/2}).
\]
Since $\delta_n=o(n^{-1/4})$, it follows that
\begin{equation}\label{e:qstarhalf}
\sup_{|t|\le \delta_n}Q^*[r_t^2]=o_p(n^{-1/2}).
\end{equation}

Set
\[
A_t
=
D_{\full}^*(\cdot;\widetilde P(Q_n,U_t))
-
D_{\full}^*(\cdot;\widetilde P(Q^*,U_t)),
\qquad
B_t
=
D_{\full}^*(\cdot;\widetilde P(Q^*,U_t))
-
D_{\full}^*(\cdot;P^*),
\]
so that
\[
H_t=A_t+B_t.
\]

For the marginal-law difference $A_t$, the dependence on the marginal law enters
only through $\Omega$ and $\psi$. Hence, by direct inspection of the explicit
formula for $D_{\full}^*$,
\[
\|A_t\|_\infty
\le
C(\mfc)
\left(
|\Omega_t^n-\Omega_t^*|
+
|\psi_t^n-\psi_t^*|
\right).
\]
Therefore
\[
\bar P_n^*[A_t^2]
\le
C(\mfc)
\left(
|\Omega_t^n-\Omega_t^*|^2
+
|\psi_t^n-\psi_t^*|^2
\right).
\]

For the nuisance difference $B_t$, write, for $U=(q_1,q_0,e)$,
\[
\tau(U)=q_1-q_0,
\qquad
\Omega(U)=Q^*[\lambda(e)],
\qquad
\psi(U)=\psi(Q^*,U),
\]
and define
\[
c_1(U)(x,a)=\frac{\lambda(e(x))}{\Omega(U)}\frac{a}{e(x)},
\qquad
c_0(U)(x,a)=\frac{\lambda(e(x))}{\Omega(U)}\frac{1-a}{1-e(x)},
\]
\[
b(U)(x,a)=\frac{\dot\lambda(e(x))}{\Omega(U)}(a-e(x)),
\qquad
d(U)(x)=\frac{\lambda(e(x))}{\Omega(U)}.
\]
Then, for $o=(x,a,y)$,
\begin{align*}
D_{\full}^*(o;\widetilde P(Q^*,U))
&=
c_1(U)(x,a)(y-q_1(x))
-
c_0(U)(x,a)(y-q_0(x)) \\
&\quad+
b(U)(x,a)\bigl(\tau(U)(x)-\psi(U)\bigr)
+
d(U)(x)\bigl(\tau(U)(x)-\psi(U)\bigr).
\end{align*}

Also set
\[
\Omega^*=Q^*[\lambda(e^*)],
\qquad
N^*=Q^*[\lambda(e^*)\tau^*],
\qquad
\psi^*=\psi(P^*).
\]
By the mean value theorem and positivity, for every $(x,a)$,
\[
|c_1(U_t)-c_1(U^*)|
+
|c_0(U_t)-c_0(U^*)|
+
|b(U_t)-b(U^*)|
+
|d(U_t)-d(U^*)|
\le
C(\mfc)\bigl(r_t(x)+|\Omega_t^*-\Omega^*|\bigr).
\]
Moreover,
\[
|(\tau_t(x)-\psi_t^*)-(\tau^*(x)-\psi^*)|
\le
r_t(x)+|\psi_t^*-\psi^*|.
\]
Since
\[
|y-q_a^*(x)|\le 1,
\qquad
|a-e^*(x)|\le 1,
\qquad
|\tau^*(x)-\psi^*|
\le
1+\frac{\Lambda_{\max}}{\Lambda_{\min}},
\]
we conclude that
\[
|B_t((x,a,y))|
\le
C(\mfc)
\left(
r_t(x)+|\Omega_t^*-\Omega^*|+|\psi_t^*-\psi^*|
\right).
\]

Now
\[
|\Omega_t^*-\Omega^*|
=
|Q^*[\lambda(e_t)-\lambda(e^*)]|
\le
\dot\Lambda_{\max}Q^*[|e_t-e^*|]
\le
C(\mfc)Q^*[r_t^2]^{1/2}.
\]
Similarly,
\begin{align*}
|N_t^*-N^*|
&=
|Q^*[\lambda(e_t)\tau_t-\lambda(e^*)\tau^*]| \\
&\le
Q^*\!\left[|\lambda(e_t)-\lambda(e^*)|\,|\tau_t|\right]
+
Q^*\!\left[\lambda(e^*)\,|\tau_t-\tau^*|\right] \\
&\le
C(\mfc)Q^*[r_t^2]^{1/2},
\end{align*}
and therefore
\[
|\psi_t^*-\psi^*|
\le
C(\mfc)Q^*[r_t^2]^{1/2}.
\]
Hence
\[
|B_t((x,a,y))|
\le
C(\mfc)\left(r_t(x)+Q^*[r_t^2]^{1/2}\right),
\]
so
\[
\bar P_n^*[B_t^2]
\le
C(\mfc)\left(Q_n[r_t^2]+Q^*[r_t^2]\right).
\]

Combining the bounds for $A_t$ and $B_t$, we obtain
\begin{align*}
\bar P_n^*[H_t^2]
&\le
2\bar P_n^*[A_t^2]+2\bar P_n^*[B_t^2] \\
&\le
C(\mfc)
\left(
Q_n[r_t^2]
+
Q^*[r_t^2]
+
|\Omega_t^n-\Omega_t^*|^2
+
|\psi_t^n-\psi_t^*|^2
\right).
\end{align*}
By the bounds already proved in this lemma, together with
\[
\sup_{|t|\le \delta_n}Q^*[r_t^2]=o_p(n^{-1/2}),
\]
from \eqref{e:qstarhalf}, the right-hand side is $o_p(n^{-1/2})$ uniformly over $|t|\le \delta_n$. Thus
\[
\sup_{|t|\le \delta_n}\bar P_n^*[H_t^2]=o_p(n^{-1/2}).
\]

Finally, on \(\mathcal A_n\), \Cref{lem:full-eif-bdd} applied with \(Q=Q_n\) and \(Q=Q^*\) gives
\[
\|D_{\full}^*(\cdot;\widetilde P(Q_n,U_t))\|_\infty
+
\|D_{\full}^*(\cdot;P^*)\|_\infty
\le 2C_{\full}(\mfc),
\]
because \(U_t\in\mathcal B_\eta(U_0)\) and \(\mathcal A_n\) implies \(U^*\in\mathcal B_\eta(U_0)\).
Therefore
\[
|G_t|
=
\left|
D_{\full}^*(\cdot;\widetilde P(Q_n,U_t))^2
-
D_{\full}^*(\cdot;P^*)^2
\right|
\le
2C_{\full}(\mfc)\,|H_t|,
\]
and so by Cauchy--Schwarz,
\[
\bar P_n^*[|G_t|]
\le
2C_{\full}(\mfc)\,\bar P_n^*[H_t^2]^{1/2}.
\]
Therefore,
\[
\sup_{|t|\le \delta_n}\bar P_n^*[|G_t|]=o_p(1).
\]
This proves the final claim.
\end{proof}

\begin{lemma}\label{lem:path-empirical}
Let $(\delta_n)$ be any deterministic sequence with $\delta_n=o(n^{-1/4})$.
Then
\[
\mathbf 1_{\mathcal A_n}\sup_{|t|\le \delta_n}
\sqrt n\left|(\P_n-\bar P_n^*)[H_t]\right|
=
o_p(1),
\]
and
\[
\mathbf 1_{\mathcal A_n}\sup_{|t|\le \delta_n}
\left|(\P_n-\bar P_n^*)[G_t]\right|
=
o_p(1).
\]
\end{lemma}

\begin{proof}
We work throughout on the event $\mathcal A_n$.  
We first treat $H_t$. Conditional on $\mathcal G_n$, the path
$t\mapsto U_t$ and therefore the family $t\mapsto H_t$ are deterministic.
By \Cref{lem:path-controls}, there exists a deterministic sequence
$a_n\downarrow 0$ such that
\[
a_n=o(n^{-1/2}),
\qquad
\P\!\left(
\sup_{|t|\le \delta_n}\bar P_n^*[H_t^2]\le a_n
\right)\to 1.
\]
Let
\[
A_n=\left\{\sup_{|t|\le \delta_n}\bar P_n^*[H_t^2]\le a_n\right\}.
\]
Also, by \Cref{lem:path-controls},
\[
\|H_t-H_s\|_\infty\le C_{\full,\infty}(\mfc)|t-s|
\qquad\text{for all }|s|,|t|\le \delta_n.
\]

Fix $\eps>0$ and choose the grid
\[
\mathcal T_n
=
\{-\delta_n,-\delta_n+\eta_n,\ldots,\delta_n\},
\qquad
\eta_n=\frac{\eps}{4C_{\full,\infty}\sqrt n}.
\]
Then
\[
|\mathcal T_n|
\le
1+\frac{2\delta_n}{\eta_n}
=
1+\frac{8C_{\full,\infty}\delta_n\sqrt n}{\eps}
=
o(n^{1/4}).
\]
For any $t\in[-\delta_n,\delta_n]$, choose $t'\in\mathcal T_n$ with
$|t-t'|\le \eta_n$. Then
\[
\left|(\P_n-\bar P_n^*)[H_t]\right|
\le
\left|(\P_n-\bar P_n^*)[H_{t'}]\right|
+
2\|H_t-H_{t'}\|_\infty
\le
\left|(\P_n-\bar P_n^*)[H_{t'}]\right|
+
\frac{\eps}{2\sqrt n}.
\]
Hence
\[
\sup_{|t|\le \delta_n}\sqrt n\left|(\P_n-\bar P_n^*)[H_t]\right|
\le
\max_{t\in\mathcal T_n}
\sqrt n\left|(\P_n-\bar P_n^*)[H_t]\right|
+\frac{\eps}{2}.
\]

On the event $A_n$, for each fixed $t\in\mathcal T_n$, conditional on
$\mathcal G_n$ the variable
\[
\sqrt n\,(\P_n-\bar P_n^*)[H_t]
\]
has mean zero and variance at most $a_n$. Thus conditional Chebyshev gives
\[
\P\!\left(
\sqrt n\left|(\P_n-\bar P_n^*)[H_t]\right|>\frac{\eps}{2}
\Bigm|\mathcal G_n
\right)\mathbf 1_{A_n}
\le
\frac{4a_n}{\eps^2}.
\]
Applying the union bound over $\mathcal T_n$,
\[
\P\!\left(
\max_{t\in\mathcal T_n}
\sqrt n\left|(\P_n-\bar P_n^*)[H_t]\right|>\frac{\eps}{2}
\Bigm|\mathcal G_n
\right)\mathbf 1_{A_n}
\le
\frac{4|\mathcal T_n|a_n}{\eps^2}.
\]
Since $|\mathcal T_n|a_n=o(1)$, we conclude
\[
\max_{t\in\mathcal T_n}
\sqrt n\left|(\P_n-\bar P_n^*)[H_t]\right|
\mathbf 1_{A_n}
=o_p(1).
\]
Because $\P(A_n)\to 1$, this yields
\[
\max_{t\in\mathcal T_n}
\sqrt n\left|(\P_n-\bar P_n^*)[H_t]\right|
=o_p(1),
\]
and the interpolation bound proves
\[
\sup_{|t|\le \delta_n}
\sqrt n\left|(\P_n-\bar P_n^*)[H_t]\right|
=
o_p(1).
\]

Now consider $G_t$. By \Cref{lem:path-controls}, there exists a deterministic
sequence $b_n\downarrow 0$ such that
\[
\P\!\left(
\sup_{|t|\le \delta_n}\bar P_n^*[|G_t|]\le b_n
\right)\to 1.
\]
Let
\[
B_n=\left\{\sup_{|t|\le \delta_n}\bar P_n^*[|G_t|]\le b_n\right\}.
\]
Again by \Cref{lem:path-controls},
\[
\|G_t-G_s\|_\infty\le C_{\full,\infty}(\mfc)|t-s|
\qquad\text{for all }|s|,|t|\le \delta_n.
\]
Fix $\eps>0$ and choose the grid
\[
\mathcal S_n
=
\{-\delta_n,-\delta_n+\rho,\ldots,\delta_n\},
\qquad
\rho=\frac{\eps}{4C_{\full,\infty}}.
\]
Since $\delta_n\downarrow 0$, the grid size $|\mathcal S_n|$ is bounded by a
deterministic constant for all large $n$.

For any $t\in[-\delta_n,\delta_n]$, choose $t'\in\mathcal S_n$ with
$|t-t'|\le \rho$. Then
\[
\left|(\P_n-\bar P_n^*)[G_t]\right|
\le
\left|(\P_n-\bar P_n^*)[G_{t'}]\right|
+
2\|G_t-G_{t'}\|_\infty
\le
\left|(\P_n-\bar P_n^*)[G_{t'}]\right|
+\frac{\eps}{2}.
\]

On the event $B_n$, for each fixed $t'\in\mathcal S_n$, conditional on
$\mathcal G_n$,
\[
(\P_n-\bar P_n^*)[G_{t'}]
\]
has mean zero and variance at most
\[
\frac{1}{n}\bar P_n^*[G_{t'}^2].
\]

By \Cref{lem:full-eif-bdd}, applied with \(Q=Q_n\), we have on \(\mathcal A_n\)
\[
\|D_{\full}^*(\cdot;\widetilde P(Q_n,U_{t'}))\|_\infty\le C_{\full}(\mfc),
\]
because \(U_{t'}\in\mathcal B_\eta(U_0)\subset\mathcal B_{2\eta}(U_0)\).
Since \(\mathcal A_n\) also implies \(U^*\in\mathcal B_\eta(U_0)\), the same lemma applied with \(Q=Q^*\) yields
\[
\|D_{\full}^*(\cdot;P^*)\|_\infty\le C_{\full}(\mfc).
\]
Hence
\[
|G_{t'}|
=
\left|
D_{\full}^*(\cdot;\widetilde P(Q_n,U_{t'}))^2
-
D_{\full}^*(\cdot;P^*)^2
\right|
\le 2C_{\full}(\mfc)^2,
\]
so
\[
G_{t'}^2\le 2C_{\full}(\mfc)^2\,|G_{t'}|.
\]
Thus, with \(M(\mfc)=2C_{\full}(\mfc)^2\), on \(B_n\),
\[
\bar P_n^*[G_{t'}^2]\le M(\mfc)\,\bar P_n^*[|G_{t'}|]\le M(\mfc)b_n.
\]

Conditional Chebyshev yields
\[
\P\!\left(
\left|(\P_n-\bar P_n^*)[G_{t'}]\right|>\frac{\eps}{2}
\Bigm|\mathcal G_n
\right)\mathbf 1_{B_n}
\le
\frac{4M(\mfc)b_n}{n\eps^2}.
\]
Since $|\mathcal S_n|$ is bounded, a union bound over $\mathcal S_n$ gives
\[
\max_{t\in\mathcal S_n}\left|(\P_n-\bar P_n^*)[G_t]\right|
\mathbf 1_{B_n}
=o_p(1).
\]
Because $\P(B_n)\to 1$, it follows that
\[
\max_{t\in\mathcal S_n}\left|(\P_n-\bar P_n^*)[G_t]\right|
=o_p(1),
\]
and the interpolation bound proves
\[
\sup_{|t|\le \delta_n}
\left|(\P_n-\bar P_n^*)[G_t]\right|
=
o_p(1).
\]
\end{proof}

\begin{lemma}\label{lem:bias-full}
Let $(\delta_n)$ be any deterministic sequence with $\delta_n=o(n^{-1/4})$.
Then
\[
\mathbf 1_{\mathcal A_n}\sup_{|t|\le \delta_n}
\sqrt n\left|
\psi_t^n-\psi(P^*)+\bar P_n^*[H_t]
\right|
=
o_p(1).
\]
\end{lemma}

\begin{proof}
We work throughout on the event $\mathcal A_n$.
For $|t|\le \delta_n$, define
\[
B_t
=
\psi_t^n-\psi(P^*)+\bar P_n^*[H_t].
\]
We will decompose $B_t$ as 
\[
B_t=A_t+C_n,
\]
where \(A_t\) is the path-dependent remainder obtained by comparing
\(\psi(Q_n,U_t)\) to \(\psi(Q_n,U^*)\) at the realized empirical design law
\(Q_n\), and \(C_n\) is the separate design-correction term needed to replace
the empirical-design target \(\psi(Q_n,U^*)\) by the population target
\(\psi(P^*)\). Thus \(A_t\) captures the second-order remainder along the
targeted path with \(Q_n\) held fixed, whereas \(C_n\) captures
the effect of replacing \(Q_n\) by \(Q^*\) at the true nuisance \(U^*\).

Since
\[
\bar P_n^*\!\left[D_{\full}^*(\cdot;\widetilde P(Q_n,U_t))\right]
=
\bar P_n^*\!\left[D^*(\cdot;\widetilde P(Q_n,U_t))\right]
\]
because the $D_X^*$ term averages to zero under $Q_n$, and since
\[
\bar P_n^*\!\left[D_{\full}^*(\cdot;P^*)\right]=Q_n[D_X^*(\cdot;P^*)],
\]
we may write
\[
B_t=A_t+C_n,
\]
where
\[
A_t
=
\psi(Q_n,U_t)-\psi(Q_n,U^*)
+
\bar P_n^*\!\left[D^*(\cdot;\widetilde P(Q_n,U_t))\right],
\]
and
\[
C_n
=
\psi(Q_n,U^*)-\psi(P^*)-Q_n[D_X^*(\cdot;P^*)].
\]

We first bound $A_t$. For $|t|\le \delta_n$, set
\[
U_t=(q_{1,t},q_{0,t},e_t),
\qquad
\Delta_{1,t}=q_{1,t}-q_1^*,
\qquad
\Delta_{0,t}=q_{0,t}-q_0^*,
\qquad
\Delta e_t=e_t-e^*,
\]
and
\[
\Delta\tau_t=\Delta_{1,t}-\Delta_{0,t},
\qquad
\tau_t=q_{1,t}-q_{0,t},
\qquad
\tau^*=q_1^*-q_0^*.
\]
Also write
\[
\Omega_t^n=Q_n[\lambda(e_t)],
\qquad
\Omega_n^*=Q_n[\lambda(e^*)],
\qquad
\psi_t^n=\psi(Q_n,U_t),
\qquad
\psi_n^*=\psi(Q_n,U^*).
\]

Since
\[
\psi_t^n=\frac{Q_n[\lambda(e_t)\tau_t]}{\Omega_t^n},
\qquad
\psi_n^*=\frac{Q_n[\lambda(e^*)\tau^*]}{\Omega_n^*},
\]
direct algebra gives
\begin{align}
\psi_t^n-\psi_n^*
&=
\frac{Q_n[\lambda(e_t)\Delta\tau_t]}{\Omega_t^n}
+
\frac{Q_n[(\lambda(e_t)-\lambda(e^*))(\tau^*-\psi_n^*)]}{\Omega_t^n}.
\label{eq:biasfull-psi}
\end{align}

Next,
\begin{align}
\bar P_n^*\!\left[D_q^*(\cdot;\widetilde P(Q_n,U_t))\right]
&=
\frac{1}{\Omega_t^n}
Q_n\!\left[
\lambda(e_t)
\left(
\frac{e^*}{e_t}(q_1^*-q_{1,t})
-\frac{1-e^*}{1-e_t}(q_0^*-q_{0,t})
\right)
\right]
\nonumber\\
&=
-\frac{Q_n[\lambda(e_t)\Delta\tau_t]}{\Omega_t^n}
+
\frac{1}{\Omega_t^n}
Q_n\!\left[
\lambda(e_t)\Delta e_t
\left(
\frac{\Delta_{1,t}}{e_t}+\frac{\Delta_{0,t}}{1-e_t}
\right)
\right].
\label{eq:biasfull-Dq}
\end{align}
Also,
\[
\bar P_n^*\!\left[D_e^*(\cdot;\widetilde P(Q_n,U_t))\right]
=
-\frac{1}{\Omega_t^n}
Q_n\!\left[\dot\lambda(e_t)(\tau_t-\psi_t^n)\Delta e_t\right].
\tag{\ref{eq:biasfull-Dq}$'$}
\]
Combining \eqref{eq:biasfull-psi}, \eqref{eq:biasfull-Dq}, and
\eqref{eq:biasfull-Dq}$'$ gives
\begin{align}
A_t
&=
\frac{1}{\Omega_t^n}
Q_n\!\left[
(\tau^*-\psi_n^*)
\bigl(\lambda(e_t)-\lambda(e^*)-\dot\lambda(e_t)\Delta e_t\bigr)
\right]
\nonumber\\
&\quad+
\frac{1}{\Omega_t^n}
Q_n\!\left[
\lambda(e_t)\Delta e_t
\left(
\frac{\Delta_{1,t}}{e_t}+\frac{\Delta_{0,t}}{1-e_t}
\right)
\right]
\nonumber\\
&\quad-
\frac{1}{\Omega_t^n}
Q_n\!\left[
\dot\lambda(e_t)(\Delta\tau_t-(\psi_t^n-\psi_n^*))\Delta e_t
\right].
\label{eq:biasfull-master}
\end{align}

We now bound the three terms in \eqref{eq:biasfull-master}. By the mean value theorem,
\[
\left|\lambda(e_t)-\lambda(e^*)-\dot\lambda(e_t)\Delta e_t\right|
\le
\Lambda_{2,\max}(\Delta e_t)^2.
\]
Also,
\[
|\tau^*-\psi_n^*|
\le 1+\frac{\Lambda_{\max}}{\Lambda_{\min}},
\qquad
\Omega_t^n\ge \Lambda_{\min}.
\]
Therefore the first term in \eqref{eq:biasfull-master} is bounded by
\[
C(\mfc)\,Q_n[(\Delta e_t)^2].
\]

For the second term, using $1/e_t\le 1/\eta$, $1/(1-e_t)\le 1/\eta$,
$\lambda(e_t)\le \Lambda_{\max}$, and $ab\le (a^2+b^2)/2$,
\[
\left|
\frac{1}{\Omega_t^n}
Q_n\!\left[
\lambda(e_t)\Delta e_t
\left(
\frac{\Delta_{1,t}}{e_t}+\frac{\Delta_{0,t}}{1-e_t}
\right)
\right]
\right|
\le
C(\mfc)\,Q_n[r_t^2].
\]

For the third term, \eqref{eq:biasfull-psi} gives
\[
|\psi_t^n-\psi_n^*|
\le
\frac{Q_n[\lambda(e_t)|\Delta\tau_t|]}{\Omega_t^n}
+
\frac{Q_n[|\lambda(e_t)-\lambda(e^*)||\tau^*-\psi_n^*|]}{\Omega_t^n}.
\]
Since
\[
\Omega_t^n\ge \Lambda_{\min},
\qquad
\lambda(e_t)\le \Lambda_{\max},
\qquad
|\tau^*-\psi_n^*|\le 1+\frac{\Lambda_{\max}}{\Lambda_{\min}},
\]
and
\[
|\lambda(e_t)-\lambda(e^*)|
\le
\dot\Lambda_{\max}|\Delta e_t|,
\qquad
|\Delta\tau_t|\le r_t,
\qquad
|\Delta e_t|\le r_t,
\]
we obtain
\[
|\psi_t^n-\psi_n^*|
\le
C(\mfc)\left(
Q_n[|\Delta\tau_t|]+Q_n[|\Delta e_t|]
\right).
\]
By Cauchy--Schwarz,
\[
Q_n[|\Delta\tau_t|]+Q_n[|\Delta e_t|]
\le
2Q_n[r_t^2]^{1/2},
\]
so
\[
|\psi_t^n-\psi_n^*|
\le
C(\mfc)\,Q_n[r_t^2]^{1/2}.
\]
Therefore,
\[
Q_n\!\left[
|\Delta e_t|\cdot|\Delta\tau_t-(\psi_t^n-\psi_n^*)|
\right]
\le
Q_n[|\Delta e_t||\Delta\tau_t|]
+
|\psi_t^n-\psi_n^*|\,Q_n[|\Delta e_t|]
\le
C(\mfc)\,Q_n[r_t^2].
\]
Using $|\dot\lambda(e_t)|\le \dot\Lambda_{\max}$ and $\Omega_t^n\ge \Lambda_{\min}$,
the third term in \eqref{eq:biasfull-master} is therefore also bounded by
\[
C(\mfc)\,Q_n[r_t^2].
\]

Combining the three bounds gives
\[
|A_t|\le C(\mfc)\,Q_n[r_t^2].
\]
By \Cref{lem:path-controls},
\[
\sup_{|t|\le \delta_n}Q_n[r_t^2]=o_p(n^{-1/2}),
\]
hence
\[
\sup_{|t|\le \delta_n}|A_t|=o_p(n^{-1/2}).
\]

Next, we treat $C_n$. Define
\[
\Omega_n^*=Q_n[\lambda(e^*)],
\qquad
N_n^*=Q_n[\lambda(e^*)\tau^*],
\qquad
\psi_n^*=\frac{N_n^*}{\Omega_n^*},
\]
and
\[
\Omega^*=Q^*[\lambda(e^*)],
\qquad
N^*=Q^*[\lambda(e^*)\tau^*],
\qquad
\psi^*=\frac{N^*}{\Omega^*}=\psi(P^*).
\]
Then
\[
\psi(Q_n,U^*)=\psi_n^*,
\]
and
\[
Q_n[D_X^*(\cdot;P^*)]
=
Q_n\!\left[
\frac{\lambda(e^*)}{\Omega^*}(\tau^*-\psi^*)
\right]
=
\frac{\Omega_n^*}{\Omega^*}(\psi_n^*-\psi^*).
\]
Therefore
\[
C_n
=
\psi_n^*-\psi^*-\frac{\Omega_n^*}{\Omega^*}(\psi_n^*-\psi^*)
=
\frac{\Omega^*-\Omega_n^*}{\Omega^*}(\psi_n^*-\psi^*).
\]
Since
\[
\Omega_n^*-\Omega^*=(Q_n-Q^*)[\lambda(e^*)]
\qquad\text{and}\qquad
N_n^*-N^*=(Q_n-Q^*)[\lambda(e^*)\tau^*],
\]
and on $\mathcal A_n$ the functions $\lambda(e^*)$ and $\lambda(e^*)\tau^*$ are fixed and uniformly bounded, the usual sample-mean bound yields
\[
\Omega_n^*-\Omega^*=O_p(n^{-1/2}),
\qquad
N_n^*-N^*=O_p(n^{-1/2}).
\]
Moreover, since $\Omega_n^*,\Omega^*\ge \Lambda_{\min}>0$ and
\[
\psi_n^*-\psi^*
=
\frac{N_n^*-N^*}{\Omega_n^*}
+
N^*\left(\frac{1}{\Omega_n^*}-\frac{1}{\Omega^*}\right),
\]
it follows that $\psi_n^*-\psi^*=O_p(n^{-1/2})$. Therefore
\[
C_n
=
\frac{\Omega^*-\Omega_n^*}{\Omega^*}(\psi_n^*-\psi^*)
=
O_p(n^{-1})
=
o_p(n^{-1/2}).
\]
Combining the bounds for $A_t$ and $C_n$ proves that on $\mathcal A_n$,
\[
\sup_{|t|\le \delta_n}
\left|
\psi_t^n-\psi(P^*)+\bar P_n^*[H_t]
\right|
=o_p(n^{-1/2}),
\]
which is equivalent to the stated claim.
\end{proof}

\subsection{Final Preparations and Conclusion}
We begin with a preliminary lemma. 
This is  where using the empirical marginal law \(\hat Q^{(k)}\) rather than \(Q^*\) is algebraically essential: the \(X\)-component of the full EIF averages exactly to zero under the same empirical design law used to define \(\hat\psi_k\). 
\begin{lemma}\label{lem:full-vs-restricted-score}
Fix $k\in\{0,1\}$. On $\mathcal E_m^{(k)}$, let
\[
\hat P_k
=
\widetilde P\bigl(\hat Q^{(k)},\hat U_{\hat t_k}^{(k,1-k)}\bigr).
\]
Then
\[
\mathbb P_m^{(k)}[D_{\full}^*(\cdot;\hat P_k)]
=
0.
\]
\end{lemma}

\begin{proof}
On $\mathcal E_m^{(k)}$, the definition of $\hat t_k$ gives
\[
\mathbb P_m^{(k)}[D^*(\cdot;\hat P_k)]=0.
\]

Write
\[
\hat U_{\hat t_k}^{(k,1-k)}
=
(\hat q_{1,k}^\dagger,\hat q_{0,k}^\dagger,\hat e_k^\dagger),
\qquad
\hat\tau_k^\dagger
=
\hat q_{1,k}^\dagger-\hat q_{0,k}^\dagger,
\]
and set
\[
\hat\Omega_k
=
\hat Q^{(k)}[\lambda(\hat e_k^\dagger)],
\qquad
\hat\psi_k
=
\frac{
\hat Q^{(k)}[\lambda(\hat e_k^\dagger)\hat\tau_k^\dagger]
}{
\hat\Omega_k
}.
\]
By the explicit formulas for the full and restricted EIFs,
\[
D_{\full}^*(o;\hat P_k)
=
D^*(o;\hat P_k)
+
\frac{\lambda(\hat e_k^\dagger(x))}{\hat\Omega_k}
\bigl(\hat\tau_k^\dagger(x)-\hat\psi_k\bigr),
\qquad o=(x,a,y).
\]
The second term depends only on $X$, so
\begin{align*}
\mathbb P_m^{(k)}
\left[
\frac{\lambda(\hat e_k^\dagger(X))}{\hat\Omega_k}
\bigl(\hat\tau_k^\dagger(X)-\hat\psi_k\bigr)
\right]
&=
\hat Q^{(k)}
\left[
\frac{\lambda(\hat e_k^\dagger)}{\hat\Omega_k}
\bigl(\hat\tau_k^\dagger-\hat\psi_k\bigr)
\right] \\
&=
\frac{1}{\hat\Omega_k}
\left(
\hat Q^{(k)}[\lambda(\hat e_k^\dagger)\hat\tau_k^\dagger]
-
\hat\psi_k\,\hat Q^{(k)}[\lambda(\hat e_k^\dagger)]
\right) \\
&=0.
\end{align*}
Therefore
\[
\mathbb P_m^{(k)}[D_{\full}^*(\cdot;\hat P_k)]
=
\mathbb P_m^{(k)}[D^*(\cdot;\hat P_k)]
=
0.
\]
\end{proof}

We are now ready for the proof of \Cref{mainthm:AN_cross}. We use the notation
defined in \Cref{s:results}.

\begin{proof}[Proof of \Cref{mainthm:AN_cross}]

For each $k\in\{0,1\}$, condition on $\mathcal G_m^{(k)}$. Then $\hat Q^{(k)}$ and
$\hat U_0^{(1-k)}$ are fixed, while the evaluation-fold treatment and outcome
variables are conditionally independent with conditional law
$p^*(\cdot,\cdot\mid X_i)$. Therefore, by the foldwise cross-fitted setup fixed at the start of the
remainder subsection, \Cref{lem:score-small-root,lem:path-controls,lem:path-empirical,lem:bias-full}
apply with $n=m$, $Q_n=\hat Q^{(k)}$, $U_0=\hat U_0^{(1-k)}$, and
$\mathcal G_n=\mathcal G_m^{(k)}$.
In particular,
\[
\mathbf 1_{\mathcal A_m^{(k)}}\hat t_k=o_p(m^{-1/4}).
\]
By condition 2 of \Cref{mainthm:AN_cross},
\[
\P\big((\mathcal A_m^{(k)})^c\big)\to 0,
\]
so
\[
\hat t_k=o_p(m^{-1/4}).
\]

Hence, for each $k\in\{0,1\}$,
\[
m^{1/4}\hat t_k=o_p(1).
\]
Therefore
\[
Z_m
=
m^{1/4}\max_{k\in\{0,1\}}|\hat t_k|
\le
m^{1/4}|\hat t_0|+m^{1/4}|\hat t_1|
=
o_p(1).
\]
Therefore there exists a deterministic sequence $a_m\downarrow 0$ such that
\[
\P(Z_m\le a_m)\to 1.
\]
Indeed, for each $j\ge 1$, choose $N_j$ so large that
\[
\P(Z_m>j^{-1})\le j^{-1}
\qquad\text{for all }m\ge N_j,
\]
and define $a_m=j^{-1}$ for $N_j\le m<N_{j+1}$. Then $a_m\downarrow 0$ and
$\P(Z_m>a_m)\to 0$. Now set
\[
\delta_m=m^{-1/4}a_m.
\]
Then
\[
\delta_m=o(m^{-1/4}),
\qquad
\P\!\left(\max_{k\in\{0,1\}}|\hat t_k|\le \delta_m\right)\to 1.
\]

Set
\[
\Gamma_m
=
\bigcap_{k=0}^1
\left(
\mathcal E_m^{(k)}\cap\{|\hat t_k|\le \delta_m\}
\right).
\]
Also, by condition 2 of \Cref{mainthm:AN_cross} and \Cref{mainthm:det-bracket-one}, for each $k\in\{0,1\}$,
\begin{align*}
\P\big((\mathcal E_m^{(k)})^{\mathrm c}\big)
&\le
\P\big((\mathcal A_m^{(k)})^{\mathrm c}\big)
+
\E\!\left[
\mathbf 1_{\mathcal A_m^{(k)}}
\P\!\left((\mathcal E_m^{(k)})^{\mathrm c}\mid \mathcal G_m^{(k)}\right)
\right] \\
&\le
\P\big((\mathcal A_m^{(k)})^{\mathrm c}\big)
+
\frac{1}{m_0}e^{-m_0 m}.
\end{align*}
The first term tends to zero by condition 2 of \Cref{mainthm:AN_cross}, and the
second term also tends to zero. Hence
\[
\P\big((\mathcal E_m^{(k)})^{\mathrm c}\big)\to 0.
\]
Therefore, by the union bound,
\[
\P(\Gamma_m^c)
\le
\sum_{k=0}^1 \P\big((\mathcal E_m^{(k)})^c\big)
+
\P\!\left(\max_{k\in\{0,1\}}|\hat t_k|>\delta_m\right)
\to 0,
\]
so $\P(\Gamma_m)\to 1$.

For $k\in\{0,1\}$ and $|t|\le \delta_m$, define
\[
H_{k,t}(\cdot)
=
D_{\full}^*\!\left(\cdot;\widetilde P(\hat Q^{(k)},\hat U_t^{(k,1-k)})\right)
-
D_{\full}^*(\cdot;P^*),
\]
and
\[
G_{k,t}(\cdot)
=
D_{\full}^*\!\left(\cdot;\widetilde P(\hat Q^{(k)},\hat U_t^{(k,1-k)})\right)^2
-
D_{\full}^*(\cdot;P^*)^2.
\]
Also define
\[
\bar P_{m,k}^*[f]
=
\frac{1}{m}\sum_{i\in I_k}\sum_{a,y\in\{0,1\}} f(X_i,a,y)\,p^*(a,y\mid X_i).
\]

On $\Gamma_m$, we have
\[
\hat P_k=\widetilde P(\hat Q^{(k)},\hat U_{\hat t_k}^{(k,1-k)}),
\qquad
\hat\psi_k=\psi(\hat Q^{(k)},\hat U_{\hat t_k}^{(k,1-k)}).
\]
Moreover,
\[
\mathbb P_m^{(k)}\!\left[D_{\full}^*(\cdot;\hat P_k)\right]=0.
\]
Thus
\begin{align*}
0
&=
\mathbb P_m^{(k)}\!\left[D_{\full}^*(\cdot;P^*)\right]
+
\mathbb P_m^{(k)}[H_{k,\hat t_k}] \\
&=
(\mathbb P_m^{(k)}-P^*)\!\left[D_{\full}^*(\cdot;P^*)\right]
+
(\mathbb P_m^{(k)}-\bar P_{m,k}^*)[H_{k,\hat t_k}]
+
\bar P_{m,k}^*[H_{k,\hat t_k}],
\end{align*}
because $P^*[D_{\full}^*(\cdot;P^*)]=0$.

Therefore, still on $\Gamma_m$,
\begin{align*}
\hat\psi_k-\psi(P^*)
&=
(\mathbb P_m^{(k)}-P^*)\!\left[D_{\full}^*(\cdot;P^*)\right] \\
&\quad+
(\mathbb P_m^{(k)}-\bar P_{m,k}^*)[H_{k,\hat t_k}] \\
&\quad+
\left(
\hat\psi_k-\psi(P^*)+\bar P_{m,k}^*[H_{k,\hat t_k}]
\right).
\end{align*}

On $\Gamma_m$, the random time $\hat t_k$ belongs to the deterministic interval
$[-\delta_m,\delta_m]$, so the uniform-in-$t$ bounds from
\Cref{lem:path-empirical,lem:bias-full} may be evaluated at $t=\hat t_k$. 
Apply \Cref{lem:path-empirical} with $n$ replaced by $m$, $Q_n$ replaced by
$\hat Q^{(k)}$, $U_t$ replaced by $\hat U_t^{(k,1-k)}$, and
$\mathcal G_n$ replaced by $\mathcal G_m^{(k)}$.
Since $|\hat t_k|\le \delta_m$ on $\Gamma_m$, this gives
\[
\sqrt m\,(\mathbb P_m^{(k)}-\bar P_{m,k}^*)[H_{k,\hat t_k}]=o_p(1).
\]
Likewise, \Cref{lem:bias-full}, applied foldwise in the same way, yields
\[
\sqrt m
\left(
\hat\psi_k-\psi(P^*)+\bar P_{m,k}^*[H_{k,\hat t_k}]
\right)
=
o_p(1).
\]
Consequently, on $\Gamma_m$,
\[
X_{m,k}
:=
\sqrt m\left(
\hat\psi_k-\psi(P^*)
-
(\mathbb P_m^{(k)}-P^*)\!\left[D_{\full}^*(\cdot;P^*)\right]
\right)
=o_p(1).
\]
Equivalently,
\[
\mathbf 1_{\Gamma_m}X_{m,k}=o_p(1).
\]
Since $\P(\Gamma_m^c)\to 0$, this implies
\[
X_{m,k}=o_p(1),
\]
because for every $\eps>0$,
\[
\P(|X_{m,k}|>\eps)
\le
\P(\Gamma_m^c)
+
\P\!\left(\mathbf 1_{\Gamma_m}|X_{m,k}|>\eps\right)
\to 0.
\]
That is,
\[
\sqrt m\left(
\hat\psi_k-\psi(P^*)
-
(\mathbb P_m^{(k)}-P^*)\!\left[D_{\full}^*(\cdot;P^*)\right]
\right)
=
o_p(1).
\]

Averaging over $k\in\{0,1\}$ and using
\[
\P_n
=
\frac12\mathbb P_m^{(0)}+\frac12\mathbb P_m^{(1)},
\qquad
n=2m,
\]
we obtain
\[
\sqrt n\left(\hat\psi_{\mathrm{CF}}-\psi(P^*)\right)
=
\frac{1}{\sqrt n}\sum_{i=1}^n D_{\full}^*(O_i;P^*)+o_p(1).
\]

It remains to show that $\sigma^2>0$. We first claim that
\[
q_1^*(x),\ q_0^*(x),\ e^*(x)\in[\eta,1-\eta]
\qquad\text{for all }x\in\mathcal X.
\]
Indeed, on any realization for which the local bracketing conditions hold with
constants $(\mfc,c_\init)$, \Cref{as::TV} gives $U^*\in\mathcal B_\eta(U_0)$,
and hence
\[
U^*\in \mathcal B_\eta(U_0)\subset \mathcal B_{2\eta}(U_0).
\]
By \Cref{as::positivity}, every element of $\mathcal B_{2\eta}(U_0)$ takes
values in $[\eta,1-\eta]^3$, so the displayed bound follows for $U^*$. Since
$U^*$ is deterministic and condition~2 implies that such realizations occur
with probability tending to one, this bound is a deterministic property of
$U^*$.

In particular, if we write
\[
\Omega^*=\E_{P_X^*}[\lambda(e^*(X))],
\]
then
\[
\lambda(e^*(X))\ge \Lambda_{\min}
\qquad\text{and}\qquad
\Omega^*\le \Lambda_{\max}.
\]

Now write
\[
D_{\full}^*(O;P^*)
=
D^*(O;P^*)
+
\frac{\lambda(e^*(X))}{\Omega^*}\bigl(\tau^*(X)-\psi(P^*)\bigr).
\]
Because the second term is $X$-measurable and $\E^*[D^*(O;P^*)\mid X]=0$, we
have
\begin{align*}
\E^*\!\left[(D_{\full}^*(O;P^*))^2\mid X\right]
&=
\E^*\!\left[(D^*(O;P^*))^2\mid X\right]
+
\left(
\frac{\lambda(e^*(X))}{\Omega^*}\bigl(\tau^*(X)-\psi(P^*)\bigr)
\right)^2 \\
&\ge
\E^*\!\left[(D^*(O;P^*))^2\mid X\right].
\end{align*}

Using the explicit formula for $D^*(O;P^*)$, the cross-term vanishes
conditionally on $X$, and therefore
\begin{align*}
\E^*\!\left[(D^*(O;P^*))^2\mid X\right]
&=
\frac{\lambda(e^*(X))^2}{(\Omega^*)^2}
\left\{
\frac{q_1^*(X)(1-q_1^*(X))}{e^*(X)}
+
\frac{q_0^*(X)(1-q_0^*(X))}{1-e^*(X)}
\right\} \\
&\quad+
\frac{\dot\lambda(e^*(X))^2}{(\Omega^*)^2}
(\tau^*(X)-\psi(P^*))^2 e^*(X)(1-e^*(X)).
\end{align*}
Dropping the nonnegative second term and using
$q_1^*(X),q_0^*(X),e^*(X)\in[\eta,1-\eta]$ gives
\[
\E^*\!\left[(D^*(O;P^*))^2\mid X\right]
\ge
\frac{\Lambda_{\min}^2}{(\Omega^*)^2}(\eta+\eta)
\ge
\frac{2\eta\,\Lambda_{\min}^2}{\Lambda_{\max}^2}
\ge
2\mfc^5.
\]
Taking expectations, we obtain
\[
\sigma^2
=
\E^*\!\left[(D_{\full}^*(O;P^*))^2\right]
\ge
2\mfc^5
>
0.
\]

Since $D_{\full}^*(O;P^*)$ has mean zero and finite variance
\[
\sigma^2=\Var_{P^*}\big(D_{\full}^*(O;P^*)\big),
\]
the central limit theorem gives
\[
\sqrt n\left(\hat\psi_{\mathrm{CF}}-\psi(P^*)\right)
\xrightarrow{d}
N(0,\sigma^2).
\]

For the variance estimator, on $\Gamma_m$ we have, for $i\in I_k$,
\[
\hat D_{\mathrm{CF},i}
=
D_{\full}^*\!\left(O_i;\widetilde P(\hat Q^{(k)},\hat U_{\hat t_k}^{(k,1-k)})\right),
\]
and therefore
\[
\frac{1}{m}\sum_{i\in I_k}
\left(
\hat D_{\mathrm{CF},i}^2-D_{\full}^*(O_i;P^*)^2
\right)
=
\mathbb P_m^{(k)}[G_{k,\hat t_k}].
\]
Since $|\hat t_k|\le \delta_m$ on $\Gamma_m$, \Cref{lem:path-empirical} gives
\[
(\mathbb P_m^{(k)}-\bar P_{m,k}^*)[G_{k,\hat t_k}]=o_p(1),
\]
while \Cref{lem:path-controls} gives
\[
\bar P_{m,k}^*[|G_{k,\hat t_k}|]=o_p(1).
\]
Hence
\[
\mathbb P_m^{(k)}[G_{k,\hat t_k}]=o_p(1).
\]

Define
\[
Y_{m,k}
=
\frac{1}{m}\sum_{i\in I_k}
\left(
\hat D_{\mathrm{CF},i}^2-D_{\full}^*(O_i;P^*)^2
\right).
\]
On $\Gamma_m$, we have $Y_{m,k}=\mathbb P_m^{(k)}[G_{k,\hat t_k}]$, so
\[
\mathbf 1_{\Gamma_m}Y_{m,k}
=
\mathbf 1_{\Gamma_m}\mathbb P_m^{(k)}[G_{k,\hat t_k}]
=
o_p(1).
\]
Since $\P(\Gamma_m^c)\to 0$, the same argument used above for $X_{m,k}$ yields
\[
Y_{m,k}=o_p(1).
\]
Therefore
\[
\frac{1}{n}\sum_{i=1}^n
\left(
\hat D_{\mathrm{CF},i}^2-D_{\full}^*(O_i;P^*)^2
\right)
=
\frac12 Y_{m,0}+\frac12 Y_{m,1}
=
o_p(1).
\]
Since
\[
\frac{1}{n}\sum_{i=1}^n D_{\full}^*(O_i;P^*)^2
\to_p
\E^*\!\left[D_{\full}^*(O;P^*)^2\right]
=
\sigma^2,
\]
it follows that
\[
\hat\sigma_{\mathrm{CF}}^2
=
\frac{1}{n}\sum_{i=1}^n \hat D_{\mathrm{CF},i}^2
\to_p \sigma^2.
\]

Finally, Slutsky's theorem yields asymptotic validity of the stated Wald interval.
\end{proof}

\section{Spline Regression Estimator}
\label{a:spline}

\subsection{Spline Background}
For the spline constructions below, we use an isotropic tensor-product B-spline basis on $[0,1]^d$. Fix an integer spline degree $r\ge 1$ and an integer $J\ge r+1$. 
For each coordinate $j\in\{1,\dots,d\}$ and all $u \in [0, 1]$, let
\[
b^{(j)}(u)=\bigl(b^{(j)}_1(u),\dots,b^{(j)}_J(u)\bigr)^\top
\]
denote the univariate degree-$r$ B-spline basis on $[0,1]$ with uniformly spaced interior knots and boundary knots repeated $r+1$ times (as is conventional).

The corresponding tensor-product basis on $[0,1]^d$ is
\[
B(x)
=
\bigl(B_\nu(x)\bigr)_{\nu\in\{1,\dots,J\}^d},
\qquad
B_\nu(x)=\prod_{j=1}^d b^{(j)}_{\nu_j}(x_j),
\]
for $x=(x_1,\dots,x_d)\in[0,1]^d$. Hence the total number of basis functions is $K=J^d$.

\begin{definition}\label{def:spline}
Let $\{(X_i,A_i,Y_i)\}_{i=1}^m$ be independent and identically distributed observations, where
$X_i\in[0,1]^d$, $A_i\in\{0,1\}$, and $Y_i\in\{0,1\}$. Fix a truncation parameter $\eta^\circ>0$, an integer spline degree $r\ge1$, and an integer $J\ge r+1$. Let $K = J^d$, 
and let
\[
B(x)=\bigl(B_1(x),\dots,B_K(x)\bigr)^\top
\]
denote the tensor-product degree-$r$ B-spline basis on $[0,1]^d$ constructed above.

Define the untruncated propensity-score estimator $\bar e$ by least squares regression:
\[
\hat\gamma_e
\in
\arg\min_{\gamma\in\mathbb R^{K}}
\frac1m\sum_{i=1}^m \bigl(A_i-B(X_i)^\top\gamma\bigr)^2,
\qquad
\bar e(x)=B(x)^\top\hat\gamma_e.
\]
If the minimizer is not unique, we take \(\hat\gamma_e\) to be the unique minimum-Euclidean-norm least-squares solution (equivalently, the Moore--Penrose pseudoinverse solution).
For each $a\in\{0,1\}$, let
\[
I_a=\{i\in\{1,\dots,m\}:A_i=a\},
\qquad
N_a=|I_a|.
\]
If $N_a\ge 1$, define the untruncated outcome-regression estimator
\[
\hat\gamma_{q,a}
\in
\arg\min_{\gamma\in\mathbb R^{K}}
\frac1{N_a}\sum_{i\in I_a}\bigl(Y_i-B(X_i)^\top\gamma\bigr)^2,
\qquad
\bar q_a(x)=B(x)^\top\hat\gamma_{q,a},
\]
where we handle potential non-uniqueness as before. 
If $N_a=0$, set $\bar q_a(x)= 1/2$ for all $x \in [0,1]^d$.

Finally, define the truncated versions
\[
\tilde e(x)=\max\{\eta^\circ,\min(\bar e(x),1-\eta^\circ)\},
\qquad
\tilde q_a(x)=\max\{\eta^\circ,\min(\bar q_a(x),1-\eta^\circ)\},
\quad a\in\{0,1\}.
\]
\end{definition}

We also recall the definition of a H\"older ball.
\begin{definition}\label{d:holderball}
Let $\beta>0$ and $L>0$.

\begin{enumerate}
\item If $\beta\notin\mathbb{N}$, write $\beta=k+\alpha$ with
$k=\lfloor \beta\rfloor$ and $\alpha\in(0,1)$. Then
$\mathcal{C}^\beta([0,1]^d,L)$ is the set of all functions
$f\colon [0,1]^d\to\mathbb{R}$ such that
\[
\max_{|\gamma|\le k}\|D^\gamma f\|_\infty
+
\max_{|\gamma|=k}\sup_{x\neq y}
\frac{|D^\gamma f(x)-D^\gamma f(y)|}{\|x-y\|^\alpha}
\le L.
\]

    \item If $\beta\in\mathbb{N}$, then $\mathcal{C}^\beta([0,1]^d,L)$ is the set
    of all functions $f\colon [0,1]^d\to\mathbb{R}$ such that
    \[
    \max_{|\gamma|\le \beta-1}\|D^\gamma f\|_\infty
    +
    \max_{|\gamma|=\beta-1}\sup_{x\neq y}
    \frac{|D^\gamma f(x)-D^\gamma f(y)|}{\|x-y\|}
    \le L.
    \]
\end{enumerate}

Here $\gamma=(\gamma_1,\dots,\gamma_d)$ is a multi-index of non-negative integers,
$|\gamma|=\sum_{j=1}^d \gamma_j$, and
\[
D^\gamma f
=
\frac{\partial^{|\gamma|}f}{\partial x_1^{\gamma_1}\cdots\partial x_d^{\gamma_d}}.
\]
\end{definition}

\subsection{Proof of \Cref{lem:spline}}
In this section, we prove \Cref{lem:spline}.

\begin{proof}[Proof of \Cref{lem:spline}]
Set
\[
r_m
=
\left(\frac{m}{\log(m\vee 3)}\right)^{-\beta/(2\beta+d)},
\qquad
U_0=(q_{1,0},q_{0,0},e_0)=(\tilde q_1,\tilde q_0,\tilde e),
\qquad
\eta=\eta^\circ/4.
\]
We first prove the displayed \(L^2(P_X^*)\) rates, and then verify the local bracketing conditions.

\medskip
\noindent
\textit{Step 1: convergence rates.}
We begin with the propensity score estimator \(\bar e\).

Since \(A\in\{0,1\}\), the least-squares regression target is
\[
e^*(x)=\E^*[A\mid X=x].
\]
Under \Cref{ass:spline}(1)--(2), the support of \(X\) is \([0,1]^d\), the design density is bounded above and below on \([0,1]^d\), and \(e^*\in \mathcal C^\beta([0,1]^d,L)\). Also, the regression error
\[
\varepsilon_i^{(e)}=A_i-e^*(X_i)
\]
is uniformly bounded. 
Since \(r>\max\{\beta,1\}\) and
\[
K=J_m^d\asymp \left(\frac{m}{\log(m\vee 3)}\right)^{d/(2\beta+d)},
\]
Theorem 2.1 of \cite{chen2015optimal} applies to the least-squares spline estimator \(\bar e\). Because the covariate support is the compact set \([0,1]^d\), the weighted sup norm appearing in that theorem is equivalent here to the ordinary sup norm. Therefore
\[
\|\bar e-e^*\|_{\infty}
=
O_p\left(\left(\frac{m}{\log(m)}\right)^{-\beta/(2\beta+d)}\right)
=
O_p(r_m).
\]
In particular,
\[
\|\bar e-e^*\|_{\infty}=o_p(1).
\]
Since \(\|f\|_{2,*}\le \|f\|_{\infty}\) for every measurable \(f\), we also have
\[
\|\bar e-e^*\|_{2,*}=O_p(r_m).
\]

Now fix \(a\in\{0,1\}\), and let
\[
N_a=\sum_{i=1}^m \mathbf 1\{A_i=a\}.
\]
By \Cref{ass:spline}(3),
\[
P^*(A=1)=\E[e^*(X)]\ge 2\eta^*,
\qquad
P^*(A=0)=\E[1-e^*(X)]\ge 2\eta^*.
\]
Hence \(N_a/m\to P^*(A=a)\) in probability, so \(N_a\asymp m\) with probability tending to one.

Conditional on \(N_a\ge 1\), the regression target in the subgroup \(A=a\) is
\[
q_a^*(x)=\E^*[Y\mid A=a,X=x].
\]
The conditional density of \(X\mid A=a\) is
\[
f_{X\mid A=a}(x)
=
\frac{P^*(A=a\mid X=x)f_X(x)}{P^*(A=a)},
\]
and \Cref{ass:spline}(1) and \Cref{ass:spline}(3) imply that this density is bounded above and below on \([0,1]^d\). Also, \(q_a^*\in \mathcal C^\beta([0,1]^d,L)\) by \Cref{ass:spline}(2), and the regression error
\[
\varepsilon_i^{(q,a)}=Y_i-q_a^*(X_i)
\]
is uniformly bounded on the subsample with \(A_i=a\). 
Let
\[
I_a=\{i\in\{1,\dots,m\}:A_i=a\}.
\]
Conditional on the random index set \(I_a\), the observations
\[
\{(X_i,Y_i)\}_{i\in I_a}
\]
are independent and identically distributed from the conditional law of \((X,Y)\mid A=a\). Moreover, on an event of probability tending to one, we have \(c m\le N_a\le C m\) for some deterministic constants \(0<c<C<\infty\), and hence
\[
J_m\asymp \left(\frac{N_a}{\log(N_a\vee 3)}\right)^{1/(2\beta+d)},
\qquad
K=J_m^d\asymp \left(\frac{N_a}{\log(N_a\vee 3)}\right)^{d/(2\beta+d)}.
\]
Therefore, conditional on \(I_a\), Theorem 2.1 of \cite{chen2015optimal} applies to the least-squares estimator \(\bar q_a\) with sample size \(N_a\). Since the support is \([0,1]^d\), the weighted sup norm in that theorem coincides here with the ordinary sup norm, so
\[
\|\bar q_a-q_a^*\|_{\infty}
=
O_p\left(\left(\frac{N_a}{\log(N_a\vee 3)}\right)^{-\beta/(2\beta+d)}\right).
\]
Since \(N_a\asymp m\) with probability tending to one, it follows that
\[
\|\bar q_a-q_a^*\|_{\infty}=O_p(r_m)=o_p(1).
\]
Again using \(\|f\|_{2,*}\le \|f\|_{\infty}\), we obtain
\[
\|\bar q_a-q_a^*\|_{2,*}=O_p(r_m).
\]

Now define the truncation map
\[
\Pi_{\eta^\circ}(u)=\max\{\eta^\circ,\min(u,1-\eta^\circ)\}.
\]
This map is \(1\)-Lipschitz. Also, since
\[
\eta^\circ<2\min\{\eta^*,\kappa^*\},
\]
\Cref{ass:spline}(3)--(4) imply that
\[
e^*(x),\,q_1^*(x),\,q_0^*(x)\in[\eta^\circ,1-\eta^\circ]
\qquad
\text{for all }x\in[0,1]^d.
\]
Therefore \(\Pi_{\eta^\circ}(e^*)=e^*\) and \(\Pi_{\eta^\circ}(q_a^*)=q_a^*\), so
\[
\|\tilde e-e^*\|_{\infty}\le \|\bar e-e^*\|_{\infty},
\qquad
\|\tilde e-e^*\|_{2,*}\le \|\bar e-e^*\|_{2,*},
\]
and similarly
\[
\|\tilde q_a-q_a^*\|_{\infty}\le \|\bar q_a-q_a^*\|_{\infty},
\qquad
\|\tilde q_a-q_a^*\|_{2,*}\le \|\bar q_a-q_a^*\|_{2,*}.
\]
Hence
\[
\|\tilde q_1-q_1^*\|_{\infty}=O_p(r_m),
\qquad
\|\tilde q_0-q_0^*\|_{\infty}=O_p(r_m),
\qquad
\|\tilde e-e^*\|_{\infty}=O_p(r_m),
\]
and
\[
\|\tilde q_1-q_1^*\|_{2,*}=O_p(r_m),
\qquad
\|\tilde q_0-q_0^*\|_{2,*}=O_p(r_m),
\qquad
\|\tilde e-e^*\|_{2,*}=O_p(r_m).
\]
Since \(\beta>d/2\), we have \(r_m=o(m^{-1/4})\). This proves the displayed conclusions in \Cref{lem:spline}.

\medskip
\noindent
\textit{Step 2: local bracketing conditions.}
We now show that there exists \(c_\init>0\) such that, with probability tending to one, the pair \((Q,U_0)\) satisfies \Cref{as::positivity}, \Cref{as::square_bound}, \Cref{as::TV}, and \Cref{ass:mu0-small} with constants \((\mfc,c_\init)\).

\medskip
\noindent
\textit{\Cref{as::positivity}.}
By construction,
\[
q_{1,0}(x),\,q_{0,0}(x),\,e_0(x)\in[\eta^\circ,1-\eta^\circ]
\qquad
\text{for all }x\in[0,1]^d.
\]
If \(U=(q_1,q_0,e)\in \mathcal B_{2\eta}(U_0)\), then \(2\eta=\eta^\circ/2\), so each coordinate of \(U\) lies in
\[
[\eta^\circ-2\eta,\;1-\eta^\circ+2\eta]
=
[\eta^\circ/2,\;1-\eta^\circ/2]
\subset[\eta,1-\eta].
\]
Thus every \(U\in \mathcal B_{2\eta}(U_0)\) satisfies
\[
q_1(x),\,q_0(x),\,e(x)\in[\eta,1-\eta]
\qquad
\text{for all }x\in[0,1]^d.
\]
The \(C^3\) and boundedness conditions on \(\lambda\) are exactly those imposed in the statement of \Cref{lem:spline}. Hence \Cref{as::positivity} holds.

\medskip
\noindent
\textit{\Cref{as::square_bound}.}
Let
\[
P_0=\widetilde P(Q,U_0),
\qquad
\tau_0=q_{1,0}-q_{0,0},
\qquad
\Omega_0=\E_Q[\lambda(e_0(X))],
\qquad
\psi_0=\psi(P_0),
\]
and abbreviate
\[
D_0(o)=D^*\left(o;\widetilde P(Q,U_0)\right).
\]
Using the explicit form of the EIF, we have 
\[
D_0(O)
=
\frac{\lambda(e_0(X))}{\Omega_0}
\left[
\frac{A}{e_0(X)}\left(Y-q_{1,0}(X)\right)
-
\frac{1-A}{1-e_0(X)}\left(Y-q_{0,0}(X)\right)
\right]
+
\frac{\dot\lambda(e_0(X))}{\Omega_0}\left(\tau_0(X)-\psi_0\right)\left(A-e_0(X)\right).
\]
Under \(P_0\), conditional on \(X\), the first bracket has conditional mean zero, the second term has conditional mean zero, and the cross-term vanishes. Therefore
\begin{align*}
\E_0[D_0(O)^2\mid X]
&=
\left(\frac{\lambda(e_0(X))}{\Omega_0}\right)^2
\left[
\frac{q_{1,0}(X)(1-q_{1,0}(X))}{e_0(X)}
+
\frac{q_{0,0}(X)(1-q_{0,0}(X))}{1-e_0(X)}
\right]\\
&\quad
+
\left(\frac{\dot\lambda(e_0(X))}{\Omega_0}\right)^2
\left(\tau_0(X)-\psi_0\right)^2 e_0(X)(1-e_0(X)).
\end{align*}
Dropping the nonnegative second term gives
\[
\E_0[D_0(O)^2\mid X]
\ge
\left(\frac{\lambda(e_0(X))}{\Omega_0}\right)^2
\left[
\frac{q_{1,0}(X)(1-q_{1,0}(X))}{e_0(X)}
+
\frac{q_{0,0}(X)(1-q_{0,0}(X))}{1-e_0(X)}
\right].
\]
Now \(q_{1,0},q_{0,0},e_0\in[\eta^\circ,1-\eta^\circ]\), so for every \(x\),
\[
\frac{q_{1,0}(x)(1-q_{1,0}(x))}{e_0(x)}
\ge
\frac{\eta^\circ(1-\eta^\circ)}{1-\eta^\circ}
=
\eta^\circ,
\qquad
\frac{q_{0,0}(x)(1-q_{0,0}(x))}{1-e_0(x)}
\ge
\eta^\circ.
\]
Also,
\[
\lambda(e_0(x))\ge \Lambda_{\min}\ge \mfc,
\qquad
\Omega_0=\E_Q[\lambda(e_0(X))]\le \Lambda_{\max}\le \mfc^{-1}.
\]
Hence
\[
\E_0[D_0(O)^2\mid X]
\ge
\frac{\Lambda_{\min}^2}{\Lambda_{\max}^2}\cdot 2\eta^\circ
\ge
2\eta^\circ \mfc^4
=
8\eta \mfc^4
\ge
8\mfc^5.
\]
Taking expectations yields
\[
\E_0[D_0(O)^2]\ge 8\mfc^5.
\]
Thus \Cref{as::square_bound} holds deterministically with, for example,
\[
c_\init=\mfc^5.
\]
Since \(\mfc\le \eta<1/4\), we have \(c_\init\in(0,1)\), as required.

\medskip
\noindent
\textit{\Cref{as::TV}.}
Define
\[
g_m(x)=\tv\left(p^*(\cdot\mid x),p_{U_0}(\cdot\mid x)\right),
\qquad
x\in[0,1]^d.
\]
For binary \((A,Y)\), a direct calculation gives the pointwise bound
\[
g_m(x)
\le
|e_0(x)-e^*(x)|
+
|q_{1,0}(x)-q_1^*(x)|
+
|q_{0,0}(x)-q_0^*(x)|.
\]
Hence
\[
\sup_{x\in[0,1]^d} g_m(x)
\le
\|e_0-e^*\|_{\infty}
+
\|q_{1,0}-q_1^*\|_{\infty}
+
\|q_{0,0}-q_0^*\|_{\infty}
=
O_p(r_m)
=
o_p(1)
\]
by Step 1. Therefore
\[
\E_Q\left[\tv\left(p^*(\cdot\mid X),p_{U_0}(\cdot\mid X)\right)\right]
\le
\sup_{x\in[0,1]^d} g_m(x)
=
o_p(1).
\]

Also, Step 1 gives
\[
\|U_0-U^*\|_\infty
\le
\|q_{1,0}-q_1^*\|_{\infty}
+
\|q_{0,0}-q_0^*\|_{\infty}
+
\|e_0-e^*\|_{\infty}
=
o_p(1).
\]
Since \(\eta>0\) is fixed, this implies
\[
\P\left(U^*\in\mathcal B_\eta(U_0)\right)\to 1.
\]

Because \(c_\init=\mfc^5\), the threshold in \Cref{as::TV} is
\[
\frac{\mfc^{10}c_\init}{600}
=
\frac{\mfc^{15}}{600},
\]
which is a fixed positive constant. Since
\[
\E_Q\left[\tv\left(p^*(\cdot\mid X),p_{U_0}(\cdot\mid X)\right)\right]=o_p(1),
\]
it follows that
\[
\E_Q\left[\tv\left(p^*(\cdot\mid X),p_{U_0}(\cdot\mid X)\right)\right]
\le
\frac{\mfc^{10} c_\init}{600}
\]
with probability tending to one. Hence \Cref{as::TV} holds with probability tending to one.

\medskip
\noindent
\textit{\Cref{ass:mu0-small}.}
Let
\[
D_Q(o)=D^*\left(o;\widetilde P(Q,U_0)\right).
\]
By definition,
\[
\mu_0
=
\E_Q\left[
\sum_{a,y\in\{0,1\}}
D_Q((X,a,y))\,p^*(a,y\mid X)
\right].
\]
For each \(x\), the EIF has mean zero under the conditional law \(p_{U_0}(\cdot\mid x)\), so
\[
\sum_{a,y\in\{0,1\}}
D_Q((x,a,y))\,p_{U_0}(a,y\mid x)=0.
\]
Therefore
\[
\mu_0
=
\E_Q\left[
\sum_{a,y\in\{0,1\}}
D_Q((X,a,y))
\left(p^*(a,y\mid X)-p_{U_0}(a,y\mid X)\right)
\right].
\]
Now \(q_{1,0},q_{0,0},e_0\in[\eta,1-\eta]\), so a direct bound from the explicit EIF formula gives
\[
\|D_Q\|_\infty\le 4\mfc^{-4}.
\]
Hence
\begin{align*}
|\mu_0|
&\le
\|D_Q\|_\infty
\E_Q\left[
\sum_{a,y\in\{0,1\}}
\left|p^*(a,y\mid X)-p_{U_0}(a,y\mid X)\right|
\right]\\
&=
2\|D_Q\|_\infty
\E_Q\left[\tv\left(p^*(\cdot\mid X),p_{U_0}(\cdot\mid X)\right)\right]\\
&\le
8\mfc^{-4}
\E_Q\left[\tv\left(p^*(\cdot\mid X),p_{U_0}(\cdot\mid X)\right)\right].
\end{align*}
By the previous step,
\[
\E_Q\left[\tv\left(p^*(\cdot\mid X),p_{U_0}(\cdot\mid X)\right)\right]=o_p(1),
\]
so
\[
\mu_0=o_p(1).
\]
Since
\[
\frac{c_{\mathrm{init}}\delta_\init}{8}
=
\frac{c_\init^2\mfc^{20}}{8\cdot 10^6}
>0
\]
is a fixed deterministic constant, it follows that
\[
|\mu_0| \le \frac{c_{\mathrm{init}}\delta_\init}{8}
\]
with probability tending to one. Thus \Cref{ass:mu0-small} holds with probability tending to one.

Combining the preceding steps, we conclude that with probability tending to one, the pair \((Q,U_0)\) satisfies \Cref{as::positivity}, \Cref{as::square_bound}, \Cref{as::TV}, and \Cref{ass:mu0-small} with constants \((\mfc,c_\init)\), where \(c_\init=\mfc^5\). Equivalently, with probability tending to one, \((Q,U_0)\) satisfies the local bracketing conditions with constants \((\mfc,c_\init)\).

Together with Step 1, this proves \Cref{lem:spline}.
\end{proof}

\section{Analysis in Banach Spaces}

In this section, we recall some facts about analysis in Banach spaces needed in our proofs. We begin by stating the Fundamental Theorem of Calculus for Banach-space-valued functions. See, for example, \cite[Chapter 16]{pata2019fixed}.

\begin{lemma}\label{l:fixed_point}
Let $X$ be a Banach space, and fix a compact interval $[a,b] \subset \R$. For every $f \in C([a,b],X)$, the function
\[
F(t)=\int_a^t f(s)\,ds,
\qquad t\in [a,b],
\]
is well defined and differentiable on $(a,b)$, with
\[
F'(t)=f(t),
\qquad t\in(a,b).
\]
Equivalently, if the derivative is viewed as a Fr\'echet derivative, then
\[
DF(t)[h]=h\,f(t),
\qquad h\in\R,\ t\in(a,b).
\]
\end{lemma}

\begin{lemma}
\label{lem:commute-der}
Let $E,F$ be Banach spaces and $T\colon E\to F$ be a bounded linear map. Fix $c>0$ and let 
$U:[-c,c]\to E$ be a function differentiable at a given point $t_0 \in (-c,c)$. Then $T\circ U\colon [-c,c] \rightarrow F$ is differentiable at $t_0$, and
\[
(T\circ U)'(t_0) = T\bigl(U'(t_0)\bigr).
\]
Equivalently, if the derivative is viewed as a Fr\'echet derivative, then
\[
D(T\circ U)(t_0)[h]=h\,T\bigl(U'(t_0)\bigr),
\qquad h\in\R.
\]
\end{lemma}

\begin{proof}
By the linearity and continuity of $T$,
\[
\frac{T(U(t))-T(U(t_0))}{t-t_0}
= T\left(\frac{U(t)-U(t_0)}{t-t_0}\right)\xrightarrow[t\to t_0]{} T\left(U'(t_0)\right).
\]
\end{proof}

For the next result, refer to \cite[Sec.~4.7]{zeidler2012applied}. 
\begin{lemma}[Banach-space chain rule]
\label{l:chain_rule}
Let $X,Y,Z$ be normed spaces, and let $U\subset X$ and $V\subset Y$ be open.
Suppose $f:U\to Y$ is Fr\'echet differentiable at $x_0\in U$, $f(U)\subset V$,
and $g:V\to Z$ is Fr\'echet differentiable at $f(x_0)$.
Then $g\circ f$ is Fr\'echet differentiable at $x_0$ and
\[
D(g\circ f)(x_0)=Dg(f(x_0))\circ Df(x_0).
\]
\end{lemma}

\begin{corollary}[Chain rule along a $C^1$ curve]
\label{cor:C1_chainrule}
Let $X,Y$ be normed spaces, $U\subset X$ open, and $F\in C^1(U;Y)$.
If $u\in C^1(I;U)$ for an open interval $I\subset\mathbb R$, then
$F\circ u\in C^1(I;Y)$ and for all $t\in I$,
\[
\frac{d}{dt}F(u(t)) = DF(u(t))[u'(t)].
\]
\end{corollary}

\begin{proof}
By \Cref{l:chain_rule}, the map $F\circ u$ is differentiable at each $t\in I$ and satisfies
\[
\frac{d}{dt}F(u(t)) = DF(u(t))[u'(t)].
\]
Because $F\in C^1(U;Y)$ and both $u$ and $u'$ are continuous on $I$, the map
\[
t\mapsto DF(u(t))[u'(t)]
\]
is continuous. Hence $F\circ u\in C^1(I;Y)$.
\end{proof}

\printbibliography

@ARTICLE{cai2020,
  title     = "One-step targeted maximum likelihood estimation for
               time-to-event outcomes",
  author    = "Cai, Weixin and van der Laan, Mark J",
  journal   = "Biometrics",
  publisher = "Wiley",
  volume    =  76,
  number    =  3,
  pages     = "722--733",
  month     =  sep,
  year      =  2020,
  copyright = "http://onlinelibrary.wiley.com/termsAndConditions\#vor"
}

@book{pata2019fixed,
  title={Fixed Point Theorems and Applications},
  author={Pata, Vittorino},
  volume={116},
  year={2019},
  publisher={Springer}
}

@book{folland1999real,
  title={Real Analysis: Modern Techniques and Their Applications},
  author={Folland, Gerald},
  year={1999},
  publisher={John Wiley \& Sons}
}

@article{chen2015optimal,
  title={Optimal uniform convergence rates and asymptotic normality for series estimators under weak dependence and weak conditions},
  author={Chen, Xiaohong and Christensen, Timothy M},
  journal={Journal of Econometrics},
  volume={188},
  number={2},
  pages={447--465},
  year={2015},
  publisher={Elsevier}
}

@article{wang2025rate,
  title={Rate doubly robust estimation for weighted average treatment effects},
  author={Wang, Yiming and Liu, Yi and Yang, Shu},
  journal={Journal of Causal Inference},
  volume={13},
  number={1},
  pages={20240073},
  year={2025},
  publisher={De Gruyter}
}

@book{zeidler2012applied,
  title={Applied functional analysis: main principles and their applications},
  author={Zeidler, Eberhard},
  volume={109},
  year={2012},
  publisher={Springer Science \& Business Media}
}

@article{Neyman1990,
  author    = {Jerzy Splawa-Neyman and D. M. Dabrowska and T. P. Speed},
  title     = {On the Application of Probability Theory to Agricultural Experiments. Essay on Principles. Section 9},
  journal   = {Statistical Science},
  volume    = {5},
  number    = {4},
  pages     = {465--472},
  year      = {1990},
  doi       = {10.1214/ss/1177012031},
  url       = {https://projecteuclid.org/journals/statistical-science/volume-5/issue-4/On-the-Application-of-Probability-Theory-to-Agricultural-Experiments-Essay/10.1214/ss/1177012031.full}
}

@article{Rubin1974,
  author    = {Donald B. Rubin},
  title     = {Estimating Causal Effects of Treatments in Randomized and Nonrandomized Studies},
  journal   = {Journal of Educational Psychology},
  volume    = {66},
  number    = {5},
  pages     = {688--701},
  year      = {1974},
  doi       = {10.1037/h0037350},
  url       = {https://doi.org/10.1037/h0037350}
}

@article{RosenbaumRubin1983,
  author    = {Paul R. Rosenbaum and Donald B. Rubin},
  title     = {The Central Role of the Propensity Score in Observational Studies for Causal Effects},
  journal   = {Biometrika},
  volume    = {70},
  number    = {1},
  pages     = {41--55},
  year      = {1983},
  doi       = {10.1093/biomet/70.1.41}
}

@book{Rosenbaum2002,
  author    = {Rosenbaum, Paul R.},
  title     = {Observational Studies},
  edition   = {2},
  series    = {Springer Series in Statistics},
  publisher = {Springer New York},
  address   = {New York, NY, USA},
  year      = {2002},
  doi       = {10.1007/978-1-4757-3692-2},
}

@book{ImbensRubin2015,
  author    = {Imbens, Guido W. and Rubin, Donald B.},
  title     = {Causal Inference for Statistics, Social, and Biomedical Sciences: An Introduction},
  publisher = {Cambridge University Press},
  address   = {Cambridge, UK},
  year      = {2015},
  doi       = {10.1017/CBO9781139025751}
}

@article{Hirano2003,
  author  = {Hirano, Keisuke and Imbens, Guido W. and Ridder, Geert},
  title   = {Efficient Estimation of Average Treatment Effects Using the Estimated Propensity Score},
  journal = {Econometrica},
  year    = {2003},
  volume  = {71},
  number  = {4},
  pages   = {1161--1189}
}

@article{TaoFu2019,
  author  = {Tao, Yebin and Fu, Haoda},
  title   = {Doubly Robust Estimation of the Weighted Average Treatment Effect for a Target Population},
  journal = {Statistics in Medicine},
  year    = {2019},
  volume  = {38},
  number  = {3},
  pages   = {315--325},
  doi     = {10.1002/sim.7980}
}

@article{Crump2009,
  author  = {Crump, Richard K. and Hotz, V. Joseph and Imbens, Guido W. and Mitnik, Oscar A.},
  title   = {Dealing with Limited Overlap in Estimation of Average Treatment Effects},
  journal = {Biometrika},
  year    = {2009},
  volume  = {96},
  number  = {1},
  pages   = {187--199}
}

@article{YangDing2018,
  author  = {Yang, Shu and Ding, Peng},
  title   = {Asymptotic Inference of Causal Effects with Observational Studies Trimmed by the Estimated Propensity Scores},
  journal = {Biometrika},
  year    = {2018},
  volume  = {105},
  number  = {2},
  pages   = {487--493},
  doi     = {10.1093/biomet/asy008}
}

@article{LiGreene2013,
  author  = {Li, Liang and Greene, Tom},
  title   = {A Weighting Analogue to Pair Matching in Propensity Score Analysis},
  journal = {The International Journal of Biostatistics},
  year    = {2013},
  volume  = {9},
  number  = {2},
  pages   = {215--234},
  doi     = {10.1515/ijb-2012-0030}
}

@article{Mao2019,
  author  = {Mao, Huzhang and Li, Liang and Greene, Tom},
  title   = {Propensity Score Weighting Analysis and Treatment Effect Discovery},
  journal = {Statistical Methods in Medical Research},
  year    = {2019},
  volume  = {28},
  number  = {8},
  pages   = {2439--2454},
  doi     = {10.1177/0962280218781171}
}

@article{Li2018Overlap,
  author  = {Li, Fan and Morgan, Kari L. and Zaslavsky, Alan M.},
  title   = {Balancing Covariates via Propensity Score Weighting},
  journal = {Journal of the American Statistical Association},
  year    = {2018},
  volume  = {113},
  number  = {521},
  pages   = {390--400}
}

@article{zhou2020propensity,
  title={Propensity score weighting under limited overlap and model misspecification},
  author={Zhou, Yunji and Matsouaka, Roland A and Thomas, Laine},
  journal={Statistical methods in medical research},
  volume={29},
  number={12},
  pages={3721--3756},
  year={2020},
  publisher={SAGE Publications Sage UK: London, England}
}

@article{matsouaka2024causal,
  title={Causal inference in the absence of positivity: The role of overlap weights},
  author={Matsouaka, Roland A and Zhou, Yunji},
  journal={Biometrical Journal},
  volume={66},
  number={4},
  pages={2300156},
  year={2024},
  publisher={Wiley Online Library}
}

@book{bickel1998, 
  title={Efficient and adaptive estimation for semiparametric models}, 
  author={Bickel, Peter J. and Klaassen, Chris and Ritov, Yaacov and Wellner, Jon},
  publisher={Springer}, 
  year={1998} }

@article{one_step_tmle,
  author = {van der Laan, Mark and Gruber, Susan},
  year = {2016},
  month = {05},
  pages = {351-378},
  title = {One-Step Targeted Minimum Loss-based Estimation Based on Universal Least Favorable One-Dimensional Submodels},
  volume = {12},
  journal = {The International Journal of Biostatistics}, }

@techreport{vanderLaanRubin2006TMLE,
  author      = {van der Laan, Mark J. and Rubin, Daniel},
  title       = {Targeted Maximum Likelihood Learning},
  institution = {U.C. Berkeley Division of Biostatistics},
  type        = {Working Paper},
  number      = {213},
  year        = {2006},
  month       = oct,
  note        = {U.C. Berkeley Division of Biostatistics Working Paper Series},
  url         = {https://biostats.bepress.com/ucbbiostat/paper213}
}

@book{book2011,
  Author = {van der Laan, M.J. and Rose, S.},
  Publisher = {Springer},
  Title = {Targeted Learning: Causal Inference for Observational and Experimental Data (Springer Series in Statistics)},
  Year = {2011} }

@book{book2018,
  Author = {van der Laan, M.J. and Rose, S.},
  Publisher = {Springer Science and Business Media},
  Title = {Targeted Learning in Data Science: Causal Inference for Complex Longitudinal Studies},
  Year = {2018} }

@inproceedings{
li2025targeted,
title={Targeted Maximum Likelihood Learning: An Optimization Perspective},
author={Diyang Li and Kyra Gan},
booktitle={The Thirty-ninth Annual Conference on Neural Information Processing Systems},
year={2025},
url={https://openreview.net/forum?id=n63KgrgVHG}
}

@misc{vanderlaan2021higherordertargetedmaximum,
      title={Higher Order Targeted Maximum Likelihood Estimation}, 
      author={Mark van der Laan and Zeyi Wang and Lars van der Laan},
      year={2021},
      eprint={2101.06290},
      archivePrefix={arXiv},
      primaryClass={math.ST},
      url={https://arxiv.org/abs/2101.06290}, 
}

@techreport{carone2014higher,
  title        = {Higher-order Targeted Minimum Loss-based Estimation},
  author       = {Carone, Marco and D{\'i}az, Iv{\'a}n and van der Laan, Mark J.},
  institution  = {University of California, Berkeley, Division of Biostatistics},
  number       = {Working Paper 331},
  year         = {2014},
  url          = {https://biostats.bepress.com/ucbbiostat/paper331/}
}

@article{vanderlaan2006tmle,
  title   = {Targeted Maximum Likelihood Learning},
  author  = {van der Laan, Mark J. and Rubin, Daniel},
  journal = {The International Journal of Biostatistics},
  volume  = {2},
  number  = {1},
  year    = {2006},
  pages   = {Article 11},
  doi     = {10.2202/1557-4679.1043}
}

@article{schuler2017targeted,
  title={Targeted maximum likelihood estimation for causal inference in observational studies},
  author={Schuler, Megan  and Rose, Sherri},
  journal={American journal of epidemiology},
  volume={185},
  number={1},
  pages={65--73},
  year={2017},
  publisher={Oxford University Press}
}

@article{SMITH202334,
title = {Application of targeted maximum likelihood estimation in public health and epidemiological studies: a systematic review},
journal = {Annals of Epidemiology},
volume = {86},
pages = {34-48.e28},
year = {2023},
issn = {1047-2797},
doi={10.1016/j.annepidem.2023.06.004},
url = {https://www.sciencedirect.com/science/article/pii/S1047279723001151},
author = {Matthew J. Smith and Rachael V. Phillips and Miguel Angel Luque-Fernandez and Camille Maringe},
}

\end{document}